\documentclass[10pt,reqno,twoside]{amsart}
\usepackage{amssymb,amsmath,amsthm,soul,color,paralist}
\usepackage{t1enc}
\usepackage{comment}
\usepackage[cp1250]{inputenc}
\usepackage{a4,indentfirst,latexsym}
\usepackage{graphics}
\usepackage{mathrsfs}
\usepackage{cite,enumitem,graphicx}
\usepackage[colorlinks=true,urlcolor=blue,
citecolor=red,linkcolor=blue,linktocpage,pdfpagelabels,
bookmarksnumbered,bookmarksopen]{hyperref}
\usepackage[english]{babel}
\usepackage[left=2.50cm,right=2.50cm,top=2.72cm,bottom=2.72cm]{geometry}
\usepackage[metapost]{mfpic}
\usepackage[hyperpageref]{backref}
\usepackage[colorinlistoftodos]{todonotes}
\usepackage[normalem]{ulem}

\numberwithin{equation}{section}

\def\Xint#1{\mathchoice
  {\XXint\displaystyle\textstyle{#1}}%
  {\XXint\textstyle\scriptstyle{#1}}%
  {\XXint\scriptstyle\scriptscriptstyle{#1}}%
  {\XXint\scriptscriptstyle\scriptscriptstyle{#1}}%
  \!\int}
\def\XXint#1#2#3{{\setbox0=\hbox{$#1{#2#3}{\int}$}
  \vcenter{\hbox{$#2#3$}}\kern-.5\wd0}}
\def\-int{\Xint -}

\newcommand{\R}{\mathbb{R}}

\newcommand{\N}{\mathcal{N}}

\newcommand{\M}{\mathcal{M}}

\newcommand{\h}{H^{s}(\R^{N})}

\newcommand{\x}{X^{s}(\R^{N+1}_{+})}
\newcommand{\y}{Y_{\mu}}

\newcommand{\La}{\Lambda}

\newcommand{\2}{2^{*}_{s}}

\newcommand{\ri}{\rightarrow}

\DeclareMathOperator{\dive}{div}
\DeclareMathOperator{\supp}{supp}
\DeclareMathOperator{\e}{\varepsilon}

\newtheorem{prop}{Proposition}[section]
\newtheorem{lem}{Lemma}[section]
\newtheorem{thm}{Theorem}[section]
\newtheorem{defn}{Definition}[section]
\newtheorem{cor}{Corollary}[section]

\newtheorem{remark}{Remark}[section]

\title[The nonlinear fractional relativistic Schr\"odinger equation]{The nonlinear fractional relativistic Schr\"odinger equation: existence, multiplicity, decay and concentration results}

\author[V. Ambrosio]{Vincenzo Ambrosio}
\address{Vincenzo Ambrosio\hfill\break\indent
Dipartimento di Ingegneria Industriale e Scienze Matematiche \hfill\break\indent
Universit\`a Politecnica delle Marche\hfill\break\indent
Via Brecce Bianche, 12\hfill\break\indent
60131 Ancona (Italy)}
\email{v.ambrosio@staff.univpm.it}


\date{}

\begin{document}

\begin{abstract}
In this paper we study the following class of fractional relativistic Schr\"odinger equations:
\begin{equation*}
\left\{
\begin{array}{ll}
(-\Delta+m^{2})^{s}u + V(\e x) u= f(u) &\mbox{ in } \R^{N}, \\
u\in H^{s}(\R^{N}), \quad u>0 &\mbox{ in } \R^{N},
\end{array}
\right.
\end{equation*}
where $\e>0$ is a small parameter, $s\in (0, 1)$, $m>0$, $N> 2s$, $(-\Delta+m^{2})^{s}$ is the fractional relativistic Schr\"odinger operator, $V:\R^{N}\ri \R$ is a  continuous potential satisfying a local condition, and $f:\R\ri \R$ is a continuous subcritical nonlinearity. 
By using a variant of the extension method and a penalization technique, we first prove that, for $\e>0$ small enough, the above problem admits a weak solution $u_{\e}$ which concentrates around a local minimum point of $V$ as $\e\ri 0$. We also show that $u_{\e}$ has an exponential decay at infinity by constructing a suitable comparison function and by performing some refined estimates.
Secondly, by combining the generalized Nehari manifold method and Ljusternik-Schnirelman theory, we relate the number of positive solutions with the topology of the set where the potential $V$ attains its minimum value.  
\end{abstract}
\maketitle

\section{Introduction}
In this paper we consider the following class of nonlinear fractional elliptic problems:
\begin{equation}\label{P}
\left\{
\begin{array}{ll}
(-\Delta+m^{2})^{s}u + V(\e x) u= f(u) &\mbox{ in } \R^{N}, \\
u\in H^{s}(\R^{N}), \quad u>0 &\mbox{ in } \R^{N},
\end{array}
\right.
\end{equation}
where $\e>0$ is a small parameter, $s\in (0, 1)$, $N> 2s$, $m>0$, $V:\R^{N}\ri \R$ is a continuous potential and $f:\R\ri \R$ is a continuous nonlinearity.
The nonlocal operator $(-\Delta+m^{2})^{s}$ appearing in \eqref{P} is defined via Fourier transform by
$$
(-\Delta+m^{2})^{s}u(x):=\mathcal{F}^{-1}((|k|^{2}+m^{2})^{s}\mathcal{F}u(k))(x), \quad x\in \R^{N}, 
$$
for any $u:\R^{N}\ri \R$ belonging to the Schwartz space $\mathcal{S}(\R^{N})$ of rapidly decaying functions, 
or equivalently (see \cite{FF, LL}) as
\begin{align}\label{FFdef}
(-\Delta+m^{2})^{s}u(x):=m^{2s}u(x)+C(N,s) m^{\frac{N+2s}{2}} P.V. \int_{\R^{N}} \frac{u(x)-u(y)}{|x-y|^{\frac{N+2s}{2}}}K_{\frac{N+2s}{2}}(m|x-y|)\, dy, \quad x\in \R^{N},
\end{align}
where $P. V.$ stands for the Cauchy principal value, $K_{\nu}$ is the modified Bessel function of the third kind (or Macdonald function) of index $\nu$ (see \cite{ArS, Erd}) which
satisfies the following well-known asymptotic formulas for $\nu\in \R$ and $r>0$:
\begin{align}
&K_{\nu}(r)\sim \frac{\Gamma(\nu)}{2} \left(\frac{r}{2}\right)^{-\nu} \quad \mbox{ as } r\rightarrow 0, \mbox{ for } \nu>0, \label{Watson1}\\
&K_{0}(r)\sim \log\left(\frac{1}{r}\right)   \quad \mbox{ as } r\rightarrow 0, \nonumber \\
&K_{\nu}(r)\sim \sqrt{\frac{\pi}{2}}r^{-\frac{1}{2}} e^{-r}  \quad \mbox{ as } r \rightarrow \infty, \mbox{ for } \nu\in \R  \label{Watson2},
\end{align}
and $C(N, s)$ is a positive constant whose exact value is given by
$$
C(N, s):=2^{-\frac{N+2s}{2}+1} \pi^{-\frac{N}{2}} 2^{2s} \frac{s(1-s)}{\Gamma(2-s)}.
$$
Equations involving  $(-\Delta+m^{2})^{s}$ arise in the study of standing waves $\psi(x, t)$ for Schr\"odinger-Klein-Gordon equations of the form 
\begin{align*}
\imath \frac{\partial \psi}{\partial t}=(-\Delta+m^{2})^{s}\psi-f(x, \psi), \quad \mbox{ in } \R^{N}\times \R, 
\end{align*}
which describe the behavior of bosons, spin-$0$ particles in relativistic fields. In particular, when $s=1/2$, the operator $\sqrt{-\Delta+m^{2}}-m$
plays an important role in relativistic quantum mechanics because it corresponds to the kinetic energy of a relativistic particle with mass $m>0$. If $p$ is the momentum of the particle then its relativistic kinetic energy is given by $E=\sqrt{p^{2}+m^{2}}$. 
In the process of quantization the momentum $p$ is replaced by the differential operator $-\imath \nabla$ and the quantum analog of the relativistic kinetic energy is the free relativistic Hamiltonian $\sqrt{-\Delta+m^{2}}-m$.
Physical models related to this operator have been widely studied over the past 30 years and there exists a huge literature on the spectral properties of relativistic Hamiltonians, most of it has been strongly influenced by the works of Lieb on the stability of relativistic matter; see  \cite{Herbst, LY1, LL, Weder1, Weder2} for more physical background.
On the other hand, there is also a deep connection  between $(-\Delta+m^{2})^{s}$ and the theory of L\'evy processes. Indeed, $m^{2s}-(-\Delta+m^{2})^{s}$ is the infinitesimal generator of a L\'evy process $X^{2s, m}_{t}$ called $2s$-stable relativistic process having the following characteristic function
$$
E^{0} e^{\imath k\cdot X^{2s, m}_{t}}=e^{-t[(|k|^{2}+m^{2})^{s}-m^{2s}]}, \quad k\in \R^{N};
$$
we refer to \cite{BMR, CMS, ryznar} for a more detailed discussion on relativistic stable processes.
When $m=0$, the previous operator boils down to the fractional Laplacian operator $(-\Delta)^{s}$ which has been extensively studied in these last years due to its great applications in several fields of the research; see \cite{BuVa, DPV, MBRS} for an introduction on this topic. In particular, a great interest has been devoted to the existence and multiplicity of solutions for fractional Schr\"odinger equations like
\begin{equation}\label{ANNALISA}
\e^{2s}(-\Delta)^{s}u + V(x) u= f(u) \quad \mbox{ in } \R^{N}, 
\end{equation}
and the asymptotic behavior of its solutions as $\e\ri 0$; see for instance \cite{AM, Aampa, Armi, Ana, DDPW, FS} and the references therein. 
We also mention \cite{DDPDV} in which the authors investigated a nonlocal Schr\"odinger equation with Dirichlet datum.

When $m>0$ and $\e=1$ in \eqref{P}, some interesting existence, multiplicity, and qualitative results can be found in \cite{Ajmp, BMP, CZN1, CZN2, FV, Mugnai, SecchiJDDE}, while in \cite{CS2} the semiclassical  limit $\e\ri 0$ for a pseudo-relativistic Hartree-equation involving $\sqrt{-\e^{2}\Delta+m^{2}}$ is considered.

Motivated by the above papers, in this work we focus our attention on the existence, multiplicity, decay and concentration phenomenon as $\e\ri 0$ of solutions to \eqref{P}. Firstly, we suppose that $V:\R^{N}\ri \R$ is a continuous function which satisfies the following del Pino-Felmer type conditions \cite{DF}:
\begin{compactenum}[$(V_1)$]
\item there exists $V_{1}\in (0, m^{2s})$ such that $-V_{1}:=\inf_{x\in \R^{N}}V(x)$,
\item there exists a bounded open set $\Lambda\subset \R^{N}$ such that
$$
-V_{0}:=\inf_{x\in \Lambda} V(x)<\min_{x\in \partial \Lambda} V(x),
$$
with $V_{0}>0$. 
We also set $M:=\{x\in \Lambda: V(x)=-V_{0}\}$. Without loss of generality, we may assume that $0\in M$.
\end{compactenum}
Concerning the nonlinearity $f$, we assume that $f:\R\ri \R$ is continuous, $f(t)=0$ for $t\leq 0$, and $f$ fulfills the following  hypotheses:
\begin{compactenum}[$(f_1)$]
\item $\lim_{t\ri 0} \frac{f(t)}{t}=0$,
\item $\limsup_{t\ri \infty} \frac{f(t)}{t^{p}}<\infty$ for some $p\in (1, \2-1)$, where $\2:=\frac{2N}{N-2s}$ is the fractional critical exponent,
\item there exists $\theta\in (2, \2)$ such that $0<\theta F(t)\leq t f(t)$ for all $t>0$, where $F(t):=\int_{0}^{t} f(\tau)\, d\tau$,
\item the function $t\mapsto \frac{f(t)}{t}$ is increasing in $(0, \infty)$.
\end{compactenum}
The first main result  of this paper can be stated as follows:
\begin{thm}\label{thm1}
Assume that $(V_1)$-$(V_2)$ and $(f_1)$-$(f_4)$ are satisfied.
Then, for every small $\e>0$, there exists a solution $u_{\e}$ to \eqref{P} such that $u_{\e}$ has a maximum point $x_{\e}$ satisfying
$$
\lim_{\e\ri 0} {\rm dist}(\e x_{\e}, M)=0,
$$
and for which
$$
0<u_{\e}(x)\leq Ce^{-c|x-x_{\e}|} \quad \mbox{ for all } x\in \R^{N},
$$
for suitable constants $C, c>0$. Moreover, for any sequence $(\e_{n})$ with $\e_{n}\ri 0$, there exists a subsequence, still denoted by itself, such that there exist a point $x_{0}\in M$ with $\e_{n}x_{\e_{n}}\ri x_{0}$, and a positive least energy solution $u\in \h$ of the limiting problem  
$$
(-\Delta+m^{2})^{s}u-V_{0} u=f(u) \quad \mbox{ in } \R^{N},
$$
for which we have
$$
u_{\e_{n}}(x)=u(x-x_{\e_{n}})+\mathcal{R}_{n}(x),
$$
where $\lim_{n\ri \infty} \|\mathcal{R}_{n}\|_{\h}=0$.
\end{thm}
The proof of Theorem \ref{thm1} is obtained through suitable variational techniques.  Firstly, we start by observing that $(-\Delta+m^{2})^{s}$  is a nonlocal operator and that does not scale like the fractional Laplacian operator $(-\Delta)^{s}$. More precisely, the first operator is not compatible with the semigroup  $\R_{+}$ acting on functions as $t*u\mapsto u(t^{-1}x)$  for $t>0$. This means that several arguments performed to deal with \eqref{ANNALISA} do not work in our context.
To overcome the nonlocality of $(-\Delta+m^{2})^{s}$, we use a variant of the extension method \cite{CS} (see \cite{CZN1, FF, StingaT}) which permits to study via local variational methods a degenerate elliptic equation in a half-space with a nonlinear Neumann boundary condition. Clearly, some additional difficulties arise in the investigation of this new problem because we have to handle the trace terms of the involved functions and thus a more careful analysis will be needed.
Due to the lack of information on the behavior of  $V$ at infinity, we carry out a penalization argument \cite{DF} which consists in modifying appropriately the nonlinearity $f$ outside $\Lambda$, and thus consider a modified problem whose corresponding energy functional fulfills all the assumptions of the mountain pass theorem \cite{AR}. Then we need to check that, for $\e>0$ small enough, the solutions of the auxiliary problem are indeed solutions of the original one. This goal will be achieved by combing an appropriate Moser iteration argument \cite{Moser} 
with some regularity estimates for degenerate elliptic equations; see Lemmas \ref{moser} and \ref{lem2.6AM}.
To our knowledge, this is the first time that the penalization trick is used to study concentration phenomena for the fractional relativistic Schr\"odinger operator $(-\Delta+m^{2})^{s}$ for all $s\in (0, 1)$ and $m>0$. 
Moreover, we show that  the solutions of \eqref{P} have an exponential decay, contrary to the case $m=0$ for which is well-known that the solutions of \eqref{ANNALISA} satisfy the power-type decay $|x|^{-(N+2s)}$ as $|x|\ri \infty$; see \cite{Armi, FQT}.
To investigate the decay of solutions to \eqref{P}, we construct a suitable comparison function and we carry out some refined estimates which take care of an adequate estimate concerning $2s$-stable relativistic density with parameter $m$ found in \cite{GR}, and that the modified Bessel function $K_{\nu}$ has an exponential decay at infinity. 
We stress that exponential type estimates for equations like \eqref{P}, appear in \cite{CZN1} where $s=\frac{1}{2}$, $V$ is bounded,  $f$ is a Hartree type nonlinearity, and in \cite{FV} where $s\in (0, 1)$, $m=1$, $V\equiv 0$ and $|f(u)|\leq C |u|^{p}$ for some $p\in (1, \2-1)$. Anyway, our approach to obtain the decay estimate is completely different from the above mentioned papers, it is more general and we believe that can be applied in other situations to deal with fractional problems driven by $(-\Delta+m^{2})^{s}$.

In the second part of this paper, we provide a multiplicity result for \eqref{P}. In this case, we assume the following conditions on $V$:
\begin{compactenum}[$(V'_1)$]
\item  there exists $V_{0}\in (0, m^{2s})$ such that $-V_{0}:=\inf_{x\in \R^{N}}V(x)$,
\item there exists a bounded open set $\Lambda\subset \R^{N}$ such that
$$
-V_{0}<\min_{x\in \partial \Lambda} V(x),
$$
and $0\in M:=\{x\in \Lambda: V(x)=-V_{0}\}$.
\end{compactenum}
In order to state precisely  the next theorem, we recall that if $Y$ is a given closed subset of a topological space $X$, we denote by $cat_{X}(Y)$ the Ljusternik-Schnirelman category of $Y$ in $X$, that is the least number of closed and contractible sets in $X$ which cover $Y$ (see \cite{W}).
Our second main result reads as follows:
\begin{thm}\label{thm2}
Assume that $(V'_1)$-$(V'_2)$ and $(f_1)$-$(f_4)$ hold. Then, for any $\delta>0$ such that
$$
M_{\delta}:=\{x\in \R^{N}: {\rm dist}(x, M)\leq \delta\}\subset \Lambda,
$$ 
there exists $\e_{\delta}>0$ such that, for any $\e\in (0, \e_{\delta})$, problem \eqref{P} has at least $cat_{M_{\delta}}(M)$ positive solutions. Moreover, if $u_{\e}$ denotes one of these solutions and $x_{\e}$ is a global maximum point of $u_{\e}$, then we have 
$$
\lim_{\e\rightarrow 0} V(\e x_{\e})=-V_{0}.
$$	
\end{thm}
The proof of Theorem \ref{thm2} will be obtained by combining variational methods, a penalization technique and Ljusternik-Schnirelman category theory. 
Since the nonlinearity $f$ is only continuous, we cannot use standard arguments for $C^{1}$-Nehari manifold (see \cite{W}) and we overcome the non differentiability of the Nehari manifold by taking advantage of some variants of critical point theorems from Szulkin and Weth \cite{SW}. We stress that a similar approach for nonlocal problems appeared in \cite{Aampa, FJ}, but with respect to these papers our analysis is rather tough due to the different structure of the energy functional associated to the extended modified problem (this difficulty is again caused by the lack of scaling invariance for $(-\Delta+m^{2})^{s}$).
Therefore, some clever and intriguing estimates will be needed to achieve our purpose.
As far as we known, this is the first time that the penalization scheme and topological arguments are mixed to get multiple solutions for fractional relativistic Schr\"odinger equations.

The paper is organized as follows. In section $2$ we collect some notations and preliminary results which will be used along the paper. In section $3$ we introduce a penalty functional in order to apply suitable variational arguments. The section $4$ is devoted to the proof of Theorem \ref{thm1}. The proof of Theorem \ref{thm2} is given in Section $5$. Finally, we give some interesting results for $(-\Delta+m^{2})^{s}$ in the appendices A and B.

\section{preliminaries}
\subsection{Notations and functional setting}
We denote the upper half-space in $\R^{N+1}$ by
$$
\R^{N+1}_{+}:=\left\{(x, y)\in \R^{N+1}: y>0\right\},
$$
and for $(x, y)\in \R^{N+1}_{+}$ we consider the Euclidean norm $|(x, y)|:=\sqrt{|x|^{2}+y^{2}}$. By $B_{r}(x)$ we mean the (open) ball in $\R^{N}$ centered at $x\in \R^{N}$ with radius $r>0$. When $x=0$, we simply write $B_{r}$.

Let $p\in [1, \infty]$ and $A\subset \R^{N}$ be a measurable set. The notation $A^{c}:=\R^{N}\setminus A$ stands for the complement of $A$ in $\R^{N}$. 
We indicate by $L^{p}(A)$ the set of measurable functions $u: \R^{N}\ri \R$ such that
\begin{equation*}
|u|_{L^{p}(A)}:=\left\{
\begin{array}{ll}
\left(\int_{A} |u|^{p}\, dx\right)^{1/p}<\infty &\mbox{ if } p<\infty, \\
{\rm esssup}_{x\in A} |u(x)|  &\mbox{ if } p=\infty.
\end{array}
\right.
\end{equation*}
If $A=\R^{N}$, we simply write $|u|_{p}$ instead of $|u|_{L^{p}(\R^{N})}$. With $\|w\|_{L^{p}(\R^{N+1}_{+})}$ we denote the norm of $w\in L^{p}(\R^{N+1}_{+})$. For a generic real-valued function $v$, we set $v^{+}:=\max\{v, 0\}$, $v^{-}:=\min\{v, 0\}$, and $v_{-}:=\max\{-v, 0\}$.

Let $D\subset \R^{N+1}$ be a bounded domain, that is a bounded connected open set, with boundary $\partial D$. We denote by $\partial' D$ the interior of $\overline{D}\cap \partial\R^{N+1}_{+}$ in $\R^{N}$, and we set $\partial ''D:=\partial D\setminus \partial'D$.
For $R>0$, we put
\begin{align*}
&B^{+}_{R}:=\{(x,y)\in \mathbb{R}^{N+1}_{+}: |(x,y)|<R\}, \\
&\Gamma_{R}^{0}:=\partial'B_{R}^{+}=\{(x,0)\in \partial \mathbb{R}^{N+1}_{+}: |x|<R \}, \\
&\Gamma_{R}^{+}:=\partial''B_{R}^{+}=\{(x,y)\in \mathbb{R}^{N+1}: y\geq 0, \, |(x, y)|= R \}.
\end{align*}
Now, we introduce the Lebesgue spaces with weight (see \cite{FKS, JLX} for more details).
Let $D\subset \R^{N+1}_{+}$ be an open set and $r\in (1, \infty)$. 
Denote by $L^{r}(D, y^{1-2s})$ the weighted Lebesgue space \index{weighted Lebesgue space} of all measurable functions $v:D\rightarrow \R$ such that 
$$
\|v\|_{L^{r}(D, y^{1-2s})}:=\left(\iint_{D} y^{1-2s} |v|^{r}\, dx dy\right)^{\frac{1}{r}}<\infty.
$$
We say that $v\in H^{1}(D, y^{1-2s})$ if $v\in L^{2}(D, y^{1-2s})$ and its weak derivatives, collectively denoted by $\nabla v$, exist and belong to $L^{2}(D, y^{1-2s})$.
The norm of $v$ in $H^{1}(D, y^{1-2s})$ is given by
$$
\|v\|_{H^{1}(D, y^{1-2s})}:=\left(\iint_{D} y^{1-2s} (|\nabla v|^{2}+ v^{2}) \, dx dy\right)^{\frac{1}{2}}.
$$
It is clear that  $H^{1}(D, y^{1-2s})$ is a Hilbert space with the inner product
$$
\iint_{D} y^{1-2s} (\nabla v \cdot\nabla w+ v w) \, dx dy.
$$

Let $H^{s}(\R^{N})$ be the fractional Sobolev space defined as the completion of $C^{\infty}_{c}(\R^{N})$ with respect to the norm
$$
|u|_{H^{s}(\R^{N})}:=\left( \int_{\R^{N}} (|k|^{2}+m^{2})^{s} |\mathcal{F}u(k)|^{2} dk \right)^{\frac{1}{2}}.
$$
Then, $H^{s}(\R^{N})$ is continuously embedded in $L^{p}(\R^{N})$ for all $p\in [2, \2)$ and compactly in  $L^{p}_{loc}(\R^{N})$ for all $p\in [1, \2)$; see \cite{Adams, ArS, DPV, LL}. 
Next we define $X^{s}(\R^{N+1}_{+}):=H^{1}(\R^{N+1}_{+}, y^{1-2s})$ as the completion of $C^{\infty}_{c}(\overline{\R^{N+1}_{+}})$ with respect to the norm
$$
\|u\|_{X^{s}(\R^{N+1}_{+})}:=  \left(\iint_{\R^{N+1}_{+}} y^{1-2s} (|\nabla u|^{2}+m^{2}u^{2})\, dx dy  \right)^{\frac{1}{2}}.
$$
By Lemma 3.1 in \cite{FF}, we deduce that $\x$ is continuously embedded in $L^{2\gamma}(\R^{N+1}_{+}, y^{1-2s})$, 
that is
\begin{equation}\label{weightedE}
\|u\|_{L^{2\gamma}(\R^{N+1}_{+}, y^{1-2s})}\leq C_{*}\|u\|_{\x} \quad \mbox{ for all } u\in \x,
\end{equation}
where $\gamma:=1+\frac{2}{N-2s}$, and $L^{r}(\R^{N+1}_{+}, y^{1-2s})$ is the weighted Lebesgue space, with $r\in (1, \infty)$, endowed with the norm
$$
\|u\|_{L^{r}(\R^{N+1}_{+}, y^{1-2s})}:=\left(\iint_{\R^{N+1}_{+}} y^{1-2s} |u|^{r}\, dx dy\right)^{\frac{1}{r}}.
$$
Moreover, by Lemma 3.1.2 in \cite{DMV}, we also have that $\x$ is compactly embedded in $L^{2}(B_{R}^{+}, y^{1-2s})$ for all $R>0$. 
From Proposition 5 in \cite{FF}, we know that there exists a (unique) linear trace operator ${\rm Tr}: \x\ri \h$ such that 
\begin{equation}\label{traceineq}
\sqrt{\sigma_{s}} |{\rm Tr}(u)|_{H^{s}(\R^{N})}\leq \|u\|_{\x} \quad \mbox{ for all } u\in \x,
\end{equation}
where $\sigma_{s}:=2^{1-2s}\Gamma(1-s)/\Gamma(s)$. 
We also note the \eqref{traceineq} and the definition of $H^{s}$-norm imply that
\begin{equation}\label{m-ineq}
\sigma_{s}m^{2s} \int_{\R^{N}} |{\rm Tr}(u)|^{2}\, dx\leq \iint_{\R^{N+1}_{+}} y^{1-2s} (|\nabla u|^{2}+m^{2}u^{2})\, dx dy.
\end{equation}
In what follows, in order to simplify the notation, we denote ${\rm Tr}(u)$ by $u(\cdot, 0)$.

Since ${\rm Tr}(X^{s}(\R^{N+1}_{+}))\subset \h$ and the embedding $\h\subset L^{q}(\R^{N})$ is continuous for all $q\in [2, \2)$ and locally compact for all $q\in [1, \2)$, we obtain the following result:
\begin{thm}\label{Sembedding}
${\rm Tr}(X^{s}(\R^{N+1}_{+}))$ is continuously embedded in $L^{q}(\R^{N})$ for all $q\in [2, \2)$ and compactly embedded in $L^{q}_{loc}(\R^{N})$ for all $q\in [1, \2)$.
\end{thm}

In order to circumvent the nonlocal character of the pseudo-differential operator $(-\Delta+m^{2})^{s}$, we make use of a variant of the extension method \cite{CS} given in \cite{CZN1, FF, StingaT}. 
More precisely, for any $u\in \h$ there exists a unique function $U\in \x$ solving the following problem
\begin{align*}
\left\{
\begin{array}{ll}
-\dive(y^{1-2s} \nabla U)+m^{2}y^{1-2s}U=0 &\mbox{ in } \R^{N+1}_{+}, \\
U(\cdot, 0)=u &\mbox{ on } \partial \R^{N+1}_{+}\cong\R^{N}. 
\end{array}
\right.
\end{align*}
The function $U$ is called the extension of $u$ and possesses the following properties:
\begin{compactenum}[$(i)$]
\item
$$
\frac{\partial U}{\partial \nu^{1-2s}}:=-\lim_{y\ri 0} y^{1-2s} \frac{\partial U}{\partial y}(x,y)=\sigma_{s}(-\Delta+m^{2})^{s}u(x) \quad \mbox{ in distribution sense, }
$$
\item $\sqrt{\sigma_{s}}|u|_{\h}=\|U\|_{\x}\leq \|V\|_{\x}$ for all $V\in \x$ such that $V(\cdot,0)=u$.
\item if $u\in \mathcal{S}(\R^{N})$ then $U\in C^{\infty}(\R^{N+1}_{+})\cap C(\overline{\R^{N+1}_{+}})$ and it can be expressed as
$$
U(x, y)=\int_{\R^{N}} P_{s, m}(x-z,y) u(z)\, dz,
$$
with
$$
P_{s, m}(x,y):=c'_{N,s} y^{2s} m^{\frac{N+2s}{2}} |(x, y)|^{-\frac{N+2s}{2}} K_{\frac{N+2s}{2}}(m|(x, y)|), 
$$
and 
$$
c'_{N, s}:=p_{N, s}\frac{2^{\frac{N+2s}{2}-1}}{\Gamma(\frac{N+2s}{2})},
$$
where $p_{N, s}$ is the constant for the (normalized) Poisson kernel with $m=0$; see \cite{StingaT}.\\
We note that $P_{s, m}$ is the Fourier transform of $k\mapsto \vartheta(\sqrt{|k|^{2}+m^{2}})$ and that
\begin{align}\label{Nkernel}
\int_{\R^{N}} P_{s, m}(x, y)\, dx=\vartheta(m y),
\end{align}
where $\vartheta\in H^{1}(\R_{+}, y^{1-2s})$ solves the following ordinary differential equation
\begin{align*}
\left\{
\begin{array}{ll}
\vartheta''+\frac{1-2s}{y} \vartheta'-\vartheta=0 &\mbox{ in } \R_{+}, \\
\vartheta(0)=1, \quad  \lim_{y\ri \infty} \vartheta(y)=0. 
\end{array}
\right.
\end{align*}
We also recall that $\vartheta$ can be expressed via modified Bessel functions, more precisely, $\vartheta(r)=\frac{2}{\Gamma(s)} (\frac{r}{2})^{s}K_{s}(r)$; see \cite{FF} for more details.
\end{compactenum}

\noindent
Taking into account the previous facts, problem \eqref{P} can be realized in a local manner through the following nonlinear boundary value problem:
\begin{align}\label{EP}
\left\{
\begin{array}{ll}
-\dive(y^{1-2s} \nabla w)+m^{2}y^{1-2s}w=0 &\mbox{ in } \R^{N+1}_{+}, \\
\frac{\partial w}{\partial \nu^{1-2s}}=\sigma_{s} [-V_{\e}(x) w(\cdot, 0)+f(w(\cdot, 0))] &\mbox{ on } \R^{N}, 
\end{array}
\right.
\end{align}
where $V_{\e}(x):=V(\e x)$. For simplicity of notation, we will omit the constant $\sigma_{s}$ from the second equation in \eqref{EP}.
For all $\e>0$, we define
$$
X_{\e}:=\left\{u\in X^{s}(\R^{N+1}_{+}): \int_{\R^{N}} V_{\e}(x) u^{2}(x, 0)\, dx<\infty\right\}
$$
endowed with the norm
$$
\|u\|_{\e}:=\left(\|u\|^{2}_{\x}+ \int_{\R^{N}} V_{\e}(x) u^{2}(x, 0)\, dx \right)^{\frac{1}{2}}.
$$
We note that $\|\cdot\|_{\e}$ is actually a norm. Indeed, 
$$
\|u\|_{\e}^{2}=\left[\|u\|^{2}_{\x}-V_{1}\int_{\R^{N}}u^{2}(x, 0)\, dx\right]+\int_{\R^{N}} (V_{\e}(x)+V_{1})u^{2}(x, 0)\, dx,
$$
and using \eqref{m-ineq} and $(V_1)$ we can see that 
$$
\left(1-\frac{V_{1}}{m^{2s}}\right)\|u\|_{\x}^{2}\leq\left[\|u\|^{2}_{\x}-V_{1}\int_{\R^{N}}u^{2}(x, 0)\, dx\right]\leq \|u\|_{\x}^{2}
$$
that is 
$$
\|u\|^{2}_{\x}-V_{1}\int_{\R^{N}}u^{2}(x, 0)\, dx
$$
is a norm equivalent to $\|\cdot\|^{2}_{\x}$. This observation yields the required claim.
Clearly, $X_{\e}\subset X^{s}(\R^{N+1}_{+})$, and by \eqref{m-ineq} and $(V_1)$, we have
\begin{equation}\label{equivalent}
\|u\|_{X^{s}(\R^{N+1}_{+})}^{2}\leq \left(\frac{m^{2s}}{m^{2s}-V_{1}}\right) \|u\|^{2}_{\e} \quad \mbox{ for all } u\in X_{\e}.
\end{equation}
Moreover, $X_{\e}$ is a Hilbert space endowed with the inner product
\begin{align*}
\langle u, v\rangle_{\e}:=\iint_{\R^{N+1}_{+}} y^{1-2s}(\nabla u\cdot\nabla v+m^{2}uv)\, dx dy+\int_{\R^{N}} V_{\e}(x) u(x, 0)v(x, 0)\, dx
\end{align*}
for all $u, v\in X_{\e}$. With $X_{\e}^{*}$ we will denote the dual space of $X_{\e}$.
\subsection{Elliptic estimates}

For reader's convenience, we list some results about local Schauder estimates for degenerate elliptic equations involving the operator 
$$
-{\rm div}(y^{1-2s} \nabla v)+m^{2}y^{1-2s}v, \quad \mbox{ with }m>0.
$$
Firstly we give the following definition:
\begin{defn}
Let $D\subset \R^{N+1}_{+}$ be a bounded domain with $\partial'D\neq \emptyset$. Let $f\in L^{\frac{2N}{N+2s}}_{loc}(\partial'D)$ and $g\in L^{1}_{loc}(\partial'D)$.
Consider 
\begin{equation}\label{JLXproblem}
\left\{
\begin{array}{ll}
-{\rm div}(y^{1-2s} \nabla v)+m^{2} y^{1-2s}v=0  &\mbox{ in } D, \\
\frac{\partial v}{\partial \nu^{1-2s}}=f(x)v+g(x) &\mbox{ on } \partial'D.
\end{array}
\right.
\end{equation}
We say that $v\in H^{1}(D, y^{1-2s})$ is a weak supersolution
 (resp. subsolution) to \eqref{JLXproblem} in $D$ if for any nonnegative $\varphi \in C^{1}_{c}(D\cup \partial'D)$, 
\begin{equation*}
\iint_{D} y^{1-2s} (\nabla v\cdot \nabla \phi+m^{2} v\varphi) \, dx dy\geq (\leq) \int_{\partial'D} [f(x) v(x, 0)+g(x)] \varphi(x,0) \, dx.
\end{equation*}
We say that $v\in H^{1}(D, y^{1-2s})$ is a weak solution to \eqref{JLXproblem} in $D$ if it is both a weak supersolution and a weak subsolution.
\end{defn}
For $R>0$, we set $Q_{R}:=B_{R}\times (0, R)$.
Then we recall the following version of De Giorgi-Nash-Moser type theorems established in \cite{FF}
(see also \cite{FKS, JLX} for the case $m=0$).
\begin{prop}\cite{FF}\label{PROPFF}
Let $f, g\in L^{q}(B_{1})$ for some $q>\frac{N}{2s}$. 
\begin{compactenum}[$(i)$]
\item Let $v\in H^{1}(Q_{1},y^{1-2s})$ be a weak subsolution to \eqref{JLXproblem} in $Q_{1}$.
Then 
$$
\sup_{Q_{1/2}} v^{+}\leq C(\|v^{+}\|_{L^{2}(Q_{1}, y^{1-2s})}+|g^{+}|_{L^{q}(B_{1})}),
$$
where $C>0$ depends only on $m$, $N$, $s$, $q$, $|f^{+}|_{L^{q}(B_{1})}$.
\item (weak Harnack inequality) Let $v\in H^{1}(Q_{1},y^{1-2s})$ be a nonnegative weak supersolution to \eqref{JLXproblem} in $Q_{1}$. 
Then for some $p_{0}>0$ and any $0<\mu<\tau<1$ we have
$$
\inf_{\bar{Q}_{\mu}} v+|g_{-}|_{L^{q}(B_{1})}\geq C\| v \|_{L^{p_{0}}(Q_{\tau}, y^{1-2s})},
$$
where $C>0$ depends only on $m$, $N$, $s$, $q$, $\mu$, $\tau$, $|f_{-}|_{L^{q}(B_{1})}$.
\item 
Let $v\in H^{1}(Q_{1},y^{1-2s})$ be a nonnegative weak solution to \eqref{JLXproblem} in $Q_{1}$. 
Then $v\in C^{0,\alpha}(\overline{Q_{1/2}})$ and in addition
$$
\|v\|_{C^{0,\alpha}(\overline{Q_{1/2}})}\leq C(\|v\|_{L^{2}(Q_{1})}+|g|_{L^{q}(B_{1})}),
$$
with $C>0$, $\alpha\in (0, 1)$ depending only on $m$, $N$, $s$, $p$, $|f|_{L^{q}(B_{1})}$.
\end{compactenum}
\end{prop}

\section{The penalization argument}  
In order to find solutions to \eqref{EP}, we adapt the penalization approach in \cite{DF} 
which permits to study our problem via variational arguments.
Fix $\kappa>\max\{\frac{V_{1}}{m^{2s}-V_{1}}, \frac{\theta}{\theta-2}\}>1$ and $a>0$ such that $\frac{f(a)}{a}=\frac{V_{1}}{\kappa}$. Define 
\begin{equation*}
\tilde{f}(t):=\left\{
\begin{array}{ll}
f(t) &\mbox{ for } t<a, \\
\frac{V_{1}}{\kappa} t  &\mbox{ for } t\geq a.
\end{array}
\right.
\end{equation*}
Let us consider the following Carath\'eodory function
$$
g(x, t):=\chi_{\Lambda}(x) f(t)+(1-\chi_{\Lambda}(x)) \tilde{f}(t) \quad \mbox{ for } (x,t)\in \R^{N}\times \R,
$$
where $\chi_{\Lambda}$ denotes the characteristic function of $\Lambda$. Set $G(x,t):=\int_{0}^{t} g(x, \tau)\, d\tau$.
By $(f_1)$-$(f_4)$ it follows that
\begin{compactenum}[$(g_1)$]
\item $\lim_{t\ri 0} \frac{g(x, t)}{t}=0$ uniformly in $x\in \R^{N}$,
\item $g(x, t)\leq f(t)$ for all $x\in \R^{N}$, $t>0$,
\item $(i)$ $0< \theta G(x,t)\leq tg(x, t)$ for all $x\in \Lambda$ and $t>0$, \\
$(ii)$ $0\leq 2 G(x,t)\leq tg(x, t)\leq \frac{V_{1}}{\kappa} t^{2}$ for all $x\in \Lambda^{c}$ and $t>0$,
\item for each $x\in \Lambda$ the function $t\mapsto \frac{g(x, t)}{t}$ is increasing in $(0, \infty)$, and  for each $x\in \Lambda^{c}$ the function $t\mapsto \frac{g(x, t)}{t}$ is increasing in $(0, a)$.
\end{compactenum}

\noindent
Consider the following modified problem:
\begin{align}\label{MEP}
\left\{
\begin{array}{ll}
-\dive(y^{1-2s} \nabla u)+m^{2}y^{1-2s}u=0 &\mbox{ in } \R^{N+1}_{+}, \\
\frac{\partial u}{\partial \nu^{1-2s}}=-V_{\e}(x) u(\cdot, 0)+g_{\e}(\cdot, u(\cdot, 0)) &\mbox{ on } \R^{N},
\end{array}
\right.
\end{align}
where we set $g_{\e}(x, t):=g(\e x, t)$.  Obviously, if $u_{\e}$ is a positive solution of \eqref{MEP} satisfying $u_{\e}(x,0)< a$ for all $x\in \Lambda_{\e}^{c}$, where $\Lambda_{\e}:=\{x\in \R^{N}: \e x\in \Lambda\}$, then $u_{\e}$ is indeed a positive solution of \eqref{EP}.
We associate with problem \eqref{MEP} the energy functional $J_{\e}: X_{\e}\ri \R$ defined as
$$
J_{\e}(u):=\frac{1}{2}\|u\|^{2}_{\e}-\int_{\R^{N}} G_{\e}(x,u(x,0))\, dx.
$$
Clearly, $J_{\e}\in C^{1}(X_{\e}, \R)$ and its differential is given by
\begin{align*}
\langle J'_{\e}(u), v\rangle&=\iint_{\R^{N+1}_{+}} y^{1-2s} (\nabla u\cdot\nabla v+m^{2}u v)\, dx dy+\int_{\R^{N}} V_{\e}(x) u(x, 0) v(x, 0)\, dx\\
&\quad -\int_{\R^{N}} g_{\e}(x, u(x, 0)) v(x, 0)\, dx
\end{align*}
for all $u, v\in X_{\e}$. Thus, the critical points of $J_{\e}$ correspond to the weak solutions of \eqref{MEP}. To find these critical
points, we will apply the mountain pass theorem \cite{AR}.
We start by proving that $J_{\e}$ satisfies the geometric features required by this theorem.
\begin{lem}\label{lemma1}
$J_{\e}$ has a mountain pass geometry, that is:
\begin{compactenum}[$(i)$]
\item $J_{\e}(0)=0$,
\item there exist $\alpha, \rho>0$ such that $J_{\e}(u)\geq \alpha$ for all $u\in X_{\e}$ such that $\|u\|_{\e}=\rho$,
\item there exists $v\in X_{\e}$ such that $\|v\|_{\e}>\rho$ and $J_{\e}(v)<0$.
\end{compactenum}
\end{lem}
\begin{proof}
Clearly, $(i)$ is true. By $(f_1)$, $(f_2)$, $(g_1)$, $(g_2)$,  we deduce that for all $\eta>0$ there exists $C_{\eta}>0$ such that
\begin{equation}\label{growthg}
|g_{\e}(x, t)|\leq \eta |t|+C_{\eta}|t|^{\2-1} \quad \mbox{ for all } (x, t)\in \R^{N}\times \R,
\end{equation}
and
\begin{equation}\label{growthG}
|G_{\e}(x, t)|\leq \frac{\eta}{2} |t|^{2}+\frac{C_{\eta}}{\2}|t|^{\2} \quad \mbox{ for all } (x, t)\in \R^{N}\times \R.
\end{equation}
Fix $\eta\in (0, m^{2s}-V_{1})$. Using \eqref{growthG}, \eqref{m-ineq}, \eqref{equivalent} and Theorem \ref{Sembedding}, we obtain
\begin{align*}
J_{\e}(u)&\geq \frac{1}{2} \|u\|^{2}_{\e}-\frac{\eta}{2} |u(\cdot, 0)|_{2}^{2}-\frac{C_{\eta}}{\2}  |u(\cdot, 0)|^{\2}_{\2} \\
&=\frac{1}{2} \|u\|^{2}_{\e}-\frac{\eta}{2m^{2s}} m^{2s} |u(\cdot, 0)|_{2}^{2}-\frac{C_{\eta}}{\2}  |u(\cdot, 0)|^{\2}_{\2} \\
&\geq \frac{1}{2}\|u\|^{2}_{\e}-\frac{\eta}{2m^{2s}} \|u\|^{2}_{\x}-C_{\eta}C\|u\|^{\2}_{\x} \\
&\geq \left(\frac{1}{2}-  \frac{\eta}{2(m^{2s}-V_{1})} \right) \|u\|^{2}_{\e}-C_{\eta}C\|u\|^{\2}_{\e}
\end{align*}
from which we deduce that $(ii)$ holds.

To prove $(iii)$, let $v\in C^{\infty}_{c}(\overline{\R^{N+1}_{+}})$ be such that $v\not\equiv 0$ and $\supp(v(\cdot, 0))\subset \Lambda_{\e}$. Then, by $(f_3)$, we have
\begin{align*}
J_{\e}(tv)&=\frac{t^{2}}{2} \|v\|^{2}_{\e}-\int_{\R^{N}} F(t v(x, 0))\, dx\\
&\leq \frac{t^{2}}{2} \|v\|^{2}_{\e}-Ct^{\theta}\int_{\Lambda_{\e}} |v(x, 0)|^{\theta}\, dx+C'|\Lambda_{\e}|\ri -\infty  \quad \mbox{ as } t\ri \infty.
\end{align*}
This completes the proof of the lemma.
\end{proof}

By Lemma \ref{lemma1} and using a variant of the mountain pass theorem without the Palais-Smale condition (see Theorem $2.9$ in \cite{W}),
we can find a sequence $(u_{n})\subset X_{\e}$ such that
\begin{align*}
J_{\e}(u_{n})\ri c_{\e} \quad \mbox{ and } \quad J'_{\e}(u_{n})\ri 0 \mbox{ in } X^{*}_{\e},
\end{align*}
as $n\ri \infty$, where
$$
c_{\e}:=\inf_{\gamma\in \Gamma_{\e}}\max_{t\in [0, 1]}J_{\e}(\gamma(t))
$$
and
$$
\Gamma_{\e}:=\{\gamma\in C([0, 1], X_{\e}): \gamma(0)=0, \, J_{\e}(\gamma(1))< 0\}.
$$
In view of $(f_4)$, it is standard to check (see \cite{W}) that 
$$
c_{\e}=\inf_{u\in \mathcal{N}_{\e}} J_{\e}(u)=\inf_{u\in X_{\e}\setminus\{0\}} \max_{t\geq 0}J_{\e}(t u), 
$$ 
where
$$
\mathcal{N}_{\e}:=\{u\in X_{\e}: \langle J'_{\e}(u), u\rangle=0\}
$$
is the Nehari manifold associated with $J_{\e}$. 
From the growth conditions of $g$, we see that 
for any fixed $u\in \mathcal{N}_{\e}$ and $\delta>0$ small enough
\begin{align*}
0=\langle J'_{\e}(u), u\rangle&=\|u\|_{\e}^{2}-\int_{\R^{N}} g_{\e} (x, u)u\,dx\\
&\geq \|u\|_{\e}^{2}-\delta C_{1} \|u\|_{\e}^{2}-C_{\delta}\|u\|_{\e}^{\2} \\
&\geq C_{2}\|u\|^{2}_{\e}-C_{\delta}\|u\|_{\e}^{\2},
\end{align*}
so there exists $r>0$, independent of  $u$, such that 
\begin{equation}\label{uNr}
\|u\|_{\e}\geq r  \quad \mbox{ for all } u\in \mathcal{N}_{\e}.
\end{equation}
Now we prove the following fundamental compactness result:
\begin{lem}\label{lemma2}
$J_{\e}$ satisfies the Palais-Smale condition at any level $c\in \R$. 
\end{lem}
\begin{proof}
Let $c\in \R$ and $(u_{n})\subset X_{\e}$ be a Palais-Smale sequence at the level $c$, namely
\begin{align*}
J_{\e}(u_{n})\ri c \quad \mbox{ and } \quad J'_{\e}(u_{n})\ri 0 \mbox{ in } X^{*}_{\e},
\end{align*}
as  $n\ri \infty$. By this fact, $(g_3)$, \eqref{m-ineq} and \eqref{equivalent}, for $n$ big enough, we have
\begin{align*}
c+1+\|u_{n}\|_{\e}&\geq J_{\e}(u_{n})-\frac{1}{\theta}\langle J'_{\e}(u_{n}), u_{n}\rangle\\
&=\left(\frac{1}{2}-\frac{1}{\theta}\right) \|u_{n}\|^{2}_{\e}+\frac{1}{\theta} \int_{\R^{N}} g_{\e}(x, u_{n}(x, 0))u_{n}(x, 0)-\theta G_{\e}(x, u_{n}(x, 0))\, dx\\
&\geq \left(\frac{1}{2}-\frac{1}{\theta}\right) \|u_{n}\|^{2}_{\e}-\left(\frac{1}{2}-\frac{1}{\theta}\right) \frac{V_{1}}{\kappa} \int_{\R^{N}}  u_{n}^{2}(x, 0)\, dx \\
&= \left(\frac{1}{2}-\frac{1}{\theta}\right) \|u_{n}\|^{2}_{\e}-\left(\frac{1}{2}-\frac{1}{\theta}\right) \frac{V_{1}}{\kappa m^{2s}} m^{2s}\int_{\R^{N}}  u_{n}^{2}(x, 0)\, dx \\
&\geq \left(\frac{1}{2}-\frac{1}{\theta}\right) \|u_{n}\|^{2}_{\e}-\left(\frac{1}{2}-\frac{1}{\theta}\right) \frac{V_{1}}{\kappa m^{2s}}\|u_{n}\|^{2}_{\x} \\
&\geq \left(\frac{1}{2}-\frac{1}{\theta}\right) \left(1- \frac{V_{1}}{\kappa(m^{2s}-V_{1})}\right) \|u_{n}\|^{2}_{\e}.
\end{align*}
Since $\theta>2$ and $\kappa>\frac{V_{1}}{m^{2s}-V_{1}}$, we deduce that $(u_{n})$ is bounded in $X_{\e}$.
Hence, up to a subsequence, we may assume that $u_{n}\rightharpoonup u$ in $X_{\e}$. Now we prove that this convergence is indeed strong. 
Using $(g_{1})$, $(g_{2})$, $(f_{2})$, the density of $C^{\infty}_{c}(\overline{\R^{N+1}_{+}})$ in $X_{\e}$, and applying Theorem \ref{Sembedding}, it is easy to check that $\langle J'_{\e}(u), \varphi\rangle=0$ for all $\varphi\in X_{\e}$. In particular,
\begin{align}\label{AM1}
&\|u\|^{2}_{\x}-V_{1}|u(\cdot, 0)|^{2}_{2}+\int_{\Lambda_{\e}} (V_{\e}(x)+V_{1}) u^{2}(x,0)\, dx
+\int_{\Lambda^{c}_{\e}} \mathcal{C}_{\e}(x, u(x, 0))\, dx \nonumber\\
&=\int_{\Lambda_{\e}} f(u(x, 0)) u(x, 0)\, dx.
\end{align}
where $\mathcal{C}_{\e}(x, t):=(V_{\e}(x)+V_{1})t^{2}-g_{\e}(x, t)t$. Note that, by $(g_3)$-$(ii)$, it holds
\begin{align}\label{OCCHIALI}
(V_{\e}(x)+V_{1})t^{2}\geq \mathcal{C}_{\e}(x, t)\geq \left(1-\frac{1}{\kappa} \right)(V_{\e}(x)+V_{1})t^{2}\geq 0 \quad \mbox{ for all } (x, t)\in \Lambda^{c}_{\e}\times [0, \infty).
\end{align}
On the other hand, by $\langle J'_{\e}(u_{n}), u_{n}\rangle=o_{n}(1)$, we get
\begin{align}\label{AM2}
&\|u_{n}\|^{2}_{\x}-V_{1}|u_{n}(\cdot, 0)|^{2}_{2}+\int_{\Lambda_{\e}} (V_{\e}(x)+V_{1}) u^{2}_{n}(x,0)\, dx+\int_{\Lambda^{c}_{\e}} \mathcal{C}_{\e}(x, u_{n}(x, 0))\, dx \nonumber\\
&=\int_{\Lambda_{\e}} f(u_{n}(x, 0)) u_{n}(x, 0)\, dx+o_{n}(1).
\end{align}
Since $\Lambda_{\e}$ is bounded, by the compactness of Sobolev embeddings in Theorem \ref{Sembedding} we have
\begin{align}\label{AM3}
\lim_{n\ri \infty} \int_{\Lambda_{\e}} f(u_{n}(x, 0)) u_{n}(x, 0)\, dx=\int_{\Lambda_{\e}} f(u(x, 0)) u(x, 0)\, dx,
\end{align}
\begin{align}\label{AM4}
\lim_{n\ri \infty} \int_{\Lambda_{\e}} (V_{\e}(x)+V_{1}) u_{n}^{2}(x,0)\, dx=\int_{\Lambda_{\e}} (V_{\e}(x)+V_1) u^{2}(x,0)\, dx.
\end{align}
In view of \eqref{AM1}, \eqref{AM2}, \eqref{AM3}, \eqref{AM4}, we obtain
\begin{align*}
&\limsup_{n\ri \infty} \left(\|u_{n}\|^{2}_{\x}-V_{1} |u_{n}(\cdot, 0)|^{2}_{2}+\int_{\Lambda^{c}_{\e}} \mathcal{C}_{\e}(x, u_{n}(x, 0))  \, dx  \right)\\
&=\|u\|^{2}_{\x}-V_{1} |u(\cdot, 0)|^{2}_{2}+\int_{\Lambda^{c}_{\e}} \mathcal{C}_{\e}(x, u(x, 0))\, dx.
\end{align*}
On the other hand, by \eqref{OCCHIALI} and Fatou's lemma, we get
\begin{align*}
&\liminf_{n\ri \infty} \left( \|u_{n}\|^{2}_{\x}-V_{1} |u_{n}(\cdot, 0)|^{2}_{2}+\int_{\Lambda^{c}_{\e}}   \mathcal{C}_{\e}(x, u_{n}(x, 0))  \, dx  \right)\\
&\geq \|u\|^{2}_{\x}-V_{1} |u(\cdot, 0)|^{2}_{2} +\int_{\Lambda^{c}_{\e}}   \mathcal{C}_{\e}(x, u(x, 0))  \, dx.
\end{align*}
Hence,
\begin{align}\label{AM5}
\lim_{n\ri \infty} \left[\|u_{n}\|^{2}_{\x}-V_{1} |u_{n}(\cdot, 0)|^{2}_{2}\right]=\|u\|^{2}_{\x}-V_{1} |u(\cdot, 0)|^{2}_{2}
\end{align}
and
\begin{align*}
\lim_{n\ri \infty} \int_{\Lambda^{c}_{\e}} \mathcal{C}_{\e}(x, u_{n}(x, 0))\, dx=\int_{\Lambda^{c}_{\e}} \mathcal{C}_{\e}(x, u(x, 0))\, dx.
\end{align*}
The last limit combined with \eqref{OCCHIALI} yields
\begin{align*}
\int_{\Lambda^{c}_{\e}} (V_{\e}(x)+V_{1}) u^{2}(x,0)\, dx&\leq \liminf_{n\ri \infty} \int_{\Lambda^{c}_{\e}} (V_{\e}(x)+V_{1}) u_{n}^{2}(x,0)\, dx\\
&\leq \limsup_{n\ri \infty} \int_{\Lambda^{c}_{\e}} (V_{\e}(x)+V_{1}) u_{n}^{2}(x,0)\, dx\\
&\leq \left(\frac{\kappa}{\kappa-1}\right) \limsup_{n\ri \infty} \int_{\Lambda^{c}_{\e}} \mathcal{C}_{\e}(x, u_{n}(x, 0))\, dx\\
&= \left(\frac{\kappa}{\kappa-1}\right) \int_{\Lambda^{c}_{\e}} \mathcal{C}_{\e}(x, u(x, 0))\, dx\\
&\leq  \left(\frac{\kappa}{\kappa-1}\right) \int_{\Lambda^{c}_{\e}} (V_{\e}(x)+V_{1}) u^{2}(x,0)\, dx.
\end{align*}
By sending $\kappa\ri \infty$, we find
\begin{align*}
\lim_{n\ri \infty} \int_{\Lambda^{c}_{\e}} (V_{\e}(x)+V_{1}) u_{n}^{2}(x,0)\, dx=\int_{\Lambda^{c}_{\e}} (V_{\e}(x)+V_{1}) u^{2}(x,0)\, dx,
\end{align*}
which together with \eqref{AM4} implies that
\begin{align}\label{AM6}
\lim_{n\ri \infty} \int_{\R^{N}} (V_{\e}(x)+V_{1}) u_{n}^{2}(x,0)\, dx=\int_{\R^{N}} (V_{\e}(x)+V_{1}) u^{2}(x,0)\, dx.
\end{align}
From \eqref{AM5} and \eqref{AM6}, we have
$$
\lim_{n\ri \infty}  \|u_{n}\|^{2}_{\e}=\|u\|^{2}_{\e},
$$
and since $X_{\e}$ turns out to be a Hilbert space, we deduce that $u_{n}\ri u$ in $X_{\e}$ as $n\ri \infty$.
\end{proof}

In light of Lemmas \ref{lemma1} and  \ref{lemma2}, we can apply the mountain pass theorem \cite{AR} to see that, for all $\e>0$, there exists $u_{\e}\in X_{\e}\setminus\{0\}$ such that
\begin{align}\label{AM2.8}
J_{\e}(u_{\e})=c_{\e} \quad \mbox{ and } \quad J'_{\e}(u_{\e})=0.
\end{align} 
Since $g(\cdot, t)=0$ for $t\leq 0$, it follows from $\langle J'_{\e}(u_{\e}), u_{\e}^{-}\rangle=0$ that $u_{\e}\geq 0$ in $\R^{N+1}_{+}$.

Next, for $\mu>-m^{2s}$, we consider the following autonomous problem related to \eqref{EP}:
\begin{align}\label{AEP}
\left\{
\begin{array}{ll}
-\dive(y^{1-2s} \nabla w)+m^{2}y^{1-2s}w=0 &\mbox{ in } \R^{N+1}_{+}, \\
\frac{\partial w}{\partial \nu^{1-2s}}= -\mu w(\cdot, 0)+f(w(\cdot, 0)) &\mbox{ on } \R^{N},
\end{array}
\right.
\end{align}
and the corresponding energy functional
\begin{align*}
L_{\mu}(u):=\frac{1}{2} \|u\|^{2}_{Y_{\mu}}-\int_{\R^{N}} F(u(x, 0))\, dx,
\end{align*}
which is well-defined on $Y_{\mu}:=\x$ endowed with the norm 
\begin{align*}
\|u\|_{Y_{\mu}}:=\left( \|u\|^{2}_{\x}+\mu |u(\cdot, 0)|^{2}_{2}  \right)^{\frac{1}{2}}.
\end{align*}
Note that $\|\cdot\|_{Y_{\mu}}$ is a norm equivalent to $\|\cdot\|_{\x}$. In fact, using \eqref{m-ineq}, for $u\in \x$,
$$
A_{m,s}\|u\|^{2}_{\x}\leq \|u\|^{2}_{Y_{\mu}}\leq B_{m, s}\|u\|^{2}_{\x},
$$
where
\begin{align*}
A_{m,s}:=\frac{1}{m^{2s}}\min\{\mu+m^{2s}, m^{2s} \}>0 \quad \mbox{ and } \quad B_{m,s}:=\frac{1}{m^{2s}}\max\{\mu+m^{2s}, m^{2s} \}>0.
\end{align*}
Clearly, $Y_{\mu}$ is a Hilbert space with the inner product
$$
\langle u, v\rangle_{Y_{\mu}}:=\iint_{\R^{N+1}_{+}} y^{1-2s} (\nabla u\cdot\nabla v+m^{2}uv)\, dx dy+\mu\int_{\R^{N}} u(x, 0) v(x, 0)\, dx,
$$
for all $u, v\in Y_{\mu}$.
Denote by $\mathcal{M}_{\mu}$ the Nehari manifold associated with $L_{\mu}$, that is 
$$
\mathcal{M}_{\mu}:=\{u\in Y_{\mu}: \langle L'_{\mu}(u), u\rangle=0\}.
$$

It is easy to check that $L_{\mu}$ has a mountain pass geometry \cite{AR}. 
Then, using a variant of the mountain pass theorem without the Palais-Smale condition (see Theorem $2.9$ in \cite{W}),
we can find a Palais-Smale sequence $(u_{n})\subset \y$ at the mountain pass level $d_{\mu}$ of $L_{\mu}$. 
Note that $(u_{n})$ is bounded in $\y$. Indeed, by $(f_3)$, we get
\begin{align*}
C(1+\|u_{n}\|_{\y})&\geq L_{\mu}(u_{n})-\frac{1}{\theta}\langle L'_{\mu}(u_{n}), u_{n}\rangle\\
&=\left(\frac{1}{2}-\frac{1}{\theta}\right) \|u_{n}\|^{2}_{\y}+\frac{1}{\theta} \int_{\R^{N}} f(u_{n}(x, 0))u_{n}(x, 0)-\theta F(u_{n}(x, 0))\, dx\\
&\geq \left(\frac{1}{2}-\frac{1}{\theta}\right) \|u_{n}\|^{2}_{\y} 
\end{align*}
which implies the boundedness of $(u_{n})$ in $\y$. 
By $(f_4)$, we also have that
$$
d_{\mu}=\inf_{u\in \mathcal{M}_{\mu}} L_{\mu}(u)=\inf_{u\in \y\setminus\{0\}} \max_{t\geq 0}L_{\mu}(t u).
$$ 
Next we show the existence of a positive ground state solution to \eqref{AEP}. We first prove some useful technical lemmas.
The first one is a vanishing Lions type result (see \cite{Lions}).
\begin{lem}\label{Lions}
Let $t\in [2, \2)$ and $R>0$. If $(u_{n})\subset \x$ is a bounded sequence such that
$$
\lim_{n\ri \infty} \sup_{z\in \R^{N}} \int_{B_{R}(z)} |u_{n}(x, 0)|^{t}\, dx=0,
$$
then $u_{n}(\cdot, 0)\ri 0$ in $L^{r}(\R^{N})$ for all $r\in (2, \2)$.
\end{lem}
\begin{proof}
Take $q\in (t, \2)$. Given $R>0$ and $z\in \R^{N}$, by using the H\"older inequality, we get
\begin{align*}
|u_{n}(\cdot, 0)|_{L^{q}(B_{R}(z))}&\leq |u_{n}(\cdot, 0)|^{1-\lambda}_{L^{t}(B_{R}(z))} |u_{n}(\cdot, 0)|^{\lambda}_{L^{\2}(B_{R}(z))}  \quad \mbox{ for all } n\in \mathbb{N},
\end{align*}
where 
$$
\frac{1-\lambda}{t}+\frac{\lambda}{\2}=\frac{1}{q}.
$$
Now, covering $\R^{N}$ by balls of radius $R$, in such a way that each point of $\R^{N}$ is contained in at most $N+1$ balls, we find
\begin{align*}
|u_{n}(\cdot, 0)|^{q}_{q}\leq (N+1) |u_{n}(\cdot, 0)|^{(1-\lambda)q}_{L^{t}(B_{R}(z))}  |u_{n}(\cdot, 0)|^{\lambda q}_{\2},
\end{align*}
which combined with Theorem \ref{Sembedding} and the assumptions yields
\begin{align*}
|u_{n}(\cdot, 0)|^{q}_{q}&\leq C(N+1) |u_{n}(\cdot, 0)|^{(1-\lambda)q}_{L^{t}(B_{R}(z))} \|u_{n}\|^{\lambda q}_{\x} \\
&\leq C(N+1) \sup_{z\in \R^{N}}|u_{n}(\cdot, 0)|^{(1-\lambda)q}_{L^{t}(B_{R}(z))}\ri 0 \quad\mbox{ as } n\ri \infty.
\end{align*}
A standard interpolation argument leads to $u_{n}(\cdot, 0)\ri 0$ in $L^{r}(\R^{N})$ for all $r\in (2, \2)$.
\end{proof}
\begin{lem}\label{Lions2}
Let $(u_{n})\subset \y$ be a Palais-Smale sequence for $L_{\mu}$ at the level $c\in \R$ and such that $u_{n}\rightharpoonup 0$ in $\y$. Then we have either
\begin{compactenum}[$(a)$]
\item $u_{n}\ri 0$ in $\y$, or
\item there exist a sequence $(z_{n})\subset \R^{N}$ and constants $R, \beta>0$ such that 
$$
\liminf_{n\ri \infty} \int_{B_{R}(z_{n})} u_{n}^{2}(x, 0)\, dx\geq \beta.
$$
\end{compactenum}
\end{lem}
\begin{proof}
Assume that $(b)$ does not occur. Then, for all $R>0$, we have
$$
\lim_{n\ri \infty} \sup_{z\in \R^{N}} \int_{B_{R}(z)} u_{n}^{2}(x, 0)\, dx=0.
$$
Using Lemma \ref{Lions}, we can see that $u_{n}(\cdot, 0)\ri 0$ in $L^{q}(\R^{N})$ for all $q\in (2, \2)$. 
This fact and $(f_1)$-$(f_2)$ imply that 
$$
\int_{\R^{N}}f(u_{n}(x, 0))u_{n}(x, 0)\, dx\ri 0 \quad \mbox{ as } n\ri \infty.
$$ 
Hence, using $\langle L'_{\mu}(u_{n}), u_{n}\rangle=o_{n}(1)$, we get $\|u_{n}\|^{2}_{\y}= o_{n}(1)$,
that is $u_{n}\ri 0$ in $\y$ as $n\ri \infty$.
\end{proof}
Now we prove the following existence result for \eqref{AEP}.
\begin{thm}\label{EGS}
Let $\mu>-m^{2s}$. Then \eqref{AEP} has a positive ground state solution. 
\end{thm}
\begin{proof}
Since $L_{\mu}$ has a mountain pass geometry \cite{AR}, we can find a Palais-Smale sequence $(u_{n})\subset \y$ at the level $d_{\mu}$. Thus $(u_{n})$ is bounded in $\y$ and there exists $u\in \y$ such that $u_{n}\rightharpoonup u$ in $\y$. Using the growth assumptions on $f$ and the density of $C^{\infty}_{c}(\overline{\R^{N+1}_{+}})$ in $\x$, it is easy to check that $\langle L'_{\mu}(u), \varphi\rangle =0$ for all $\varphi\in \y$. If $u=0$, then $u_{n}\not\rightarrow 0$ in $\y$ because $d_{\mu}>0$. Hence we can use Lemma \ref{Lions2} to deduce that for some sequence $(z_{n})\subset \R^{N}$, $v_{n}(x, y):=u_{n}(x+z_{n}, y)$ is a bounded Palais-Smale sequence at the level $d_{\mu}$ and having a nontrivial weak limit $v$. Hence, $v\in \mathcal{M}_{\mu}$.
Moreover, using the weak lower semicontinuity of $\|\cdot\|_{\y}$, $(f_4)$ and Fatou's lemma, it is easy to see that $L_{\mu}(v)=d_{\mu}$. When $u\neq 0$, as before, we can deduce that $u$ is a ground state solution to \eqref{AEP}.
In conclusion, for $\mu>-m^{2s}$, there exists a ground state solution $w=w_{\mu}\in \y\setminus\{0\}$ such that
\begin{align*}
L_{\mu}(w)=d_{\mu} \quad \mbox{ and } \quad L'_{\mu}(w)=0.
\end{align*} 
Since $f(t)=0$ for $t\leq 0$, it follows from $\langle L'_{\mu}(w), w^{-} \rangle=0$  that $w\geq 0$ in $\R^{N+1}_{+}$ and $w\not\equiv 0$. 
A Moser iteration argument (see Lemma \ref{moser} below) shows that $w(\cdot, 0)\in L^{q}(\R^{N})$ for all $q\in [2, \infty]$, and that $w\in L^{\infty}(\R^{N+1}_{+})$. Using Proposition \ref{PROPFF}-$(iii)$, we obtain that $w\in C^{0, \alpha}(\overline{\R^{N+1}_{+}})$ for some $\alpha\in (0, 1)$. By Proposition \ref{PROPFF}-$(ii)$, we conclude that $w$ is positive.
\end{proof}

In the next lemma we establish an important connection between $c_{\e}$ and $d_{V(0)}=d_{-V_{0}}$ (we remark that $V(0)=-V_{0}>-m^{2s}$):
\begin{lem}\label{lem2.3AM}
The numbers $c_{\e}$ and $d_{V(0)}$ verify the following inequality:
$$
\limsup_{\e\ri 0} c_{\e}\leq d_{V(0)}.
$$
\end{lem}
\begin{proof}
By Theorem \ref{EGS}, we know that there exists a positive ground state solution $w$ to \eqref{AEP} with $\mu=V(0)$. Let $\eta\in C^{\infty}_{c}(\R)$ be such that $0\leq \eta\leq 1$, $\eta=1$ in $[-1, 1]$ and $\eta=0$ in $\R\setminus (-2, 2)$. Suppose that $B_{2}\subset \Lambda$. Define $w_{\e}(x, y):=\eta(\e |(x, y)|) w(x, y)$ and note that $\supp(w_{\e}(\cdot, 0))\subset \Lambda_{\e}$. 
It is easy to prove that $w_{\e}\ri w$ in $\x$ and that $L_{V(0)}(w_{\e})\ri L_{V(0)}(w)$ as $\e\ri 0$.
On the other hand, by definition of $c_{\e}$, we have
\begin{align}\label{15ADOM}
c_{\e}\leq \max_{t\geq 0} J_{\e}(t w_{\e})=J_{\e}(t_{\e} w_{\e})=\frac{t^{2}_{\e}}{2} \|w_{\e}\|^{2}_{\e}-\int_{\R^{N}} F(t_{\e} w_{\e}(x, 0))\, dx
\end{align}
for some $t_{\e}>0$. Recalling that $w\in \mathcal{M}_{V(0)}$ and using $(f_4)$, it is readily seen that $t_{\e}\ri 1$ as $\e\ri 0$.
Note that
\begin{align}\label{16ADOM}
J_{\e}(t_{\e}w_{\e})=L_{V(0)}(t_{\e}w_{\e})+\frac{t^{2}_{\e}}{2}\int_{\R^{N}} (V_{\e}(x)-V(0)) w_{\e}^{2}(x, 0)\, dx.
\end{align}
Since $V_{\e}(x)$ is bounded on the support of $w_{\e}(\cdot, 0)$, and $V_{\e}(x)\ri V(0)$ as $\e\ri 0$, we can apply the dominated convergence theorem and use \eqref{15ADOM} and \eqref{16ADOM} to conclude the proof.
\end{proof}

\begin{remark}\label{REMARKHEZOU}
Under assumptions $(V'_{1})$-$(V'_{2})$, we have that $\lim_{\e\ri 0}c_{\e}=d_{V(0)}$.
\end{remark}

Now we come back to study \eqref{MEP} and consider the mountain pass solutions $u_{\e}$ satisfying \eqref{AM2.8}.
\begin{lem}\label{lem2.4AM}
There exist $r, \beta, \e^{*}>0$ and $(y_{\e})\subset \R^{N}$ such that 
$$
\int_{B_{r}(y_{\e})} u_{\e}^{2}(x, 0)\, dx\geq \beta, \quad \mbox{ for all } \e\in (0, \e^{*}).
$$
\end{lem}
\begin{proof}
Since $u_{\e}$ verifies \eqref{AM2.8}, it follows from the growth assumptions on $g$ that there exist $\alpha>0$, independent of $\e>0$, such that 
\begin{equation}\label{contradiction}
\|u_{\e}\|_{\e}^{2}\geq \alpha \quad \mbox{  for all } \e>0.
\end{equation}
Let $(\e_{n})\subset (0, \infty)$ be such that $\e_{n}\ri 0$.
If by contradiction there exists $r>0$ such that
$$
\lim_{n\ri \infty}\sup_{y\in \R^{N}} \int_{B_{r}(y)} u_{\e_{n}}^{2}(x, 0)\, dx=0,
$$
thus we can use Lemma \ref{Lions} to deduce that $u_{\e_{n}}(\cdot, 0)\ri 0$ in $L^{q}(\R^{N})$ for all $q\in (2, \2)$. Then, \eqref{AM2.8} and the growth assumptions on $g$ imply that $\|u_{\e_{n}}\|_{\e_{n}}\ri 0$, as $n\ri \infty$, and this contradicts \eqref{contradiction}. 
\end{proof}

\begin{lem}\label{lem2.5AM}
For any $\e_{n}\ri 0$, consider the sequence $(y_{\e_{n}})\subset \R^{N}$ given in Lemma \ref{lem2.4AM} and $w_{n}(x, y):=u_{\e_{n}}(x+y_{\e_{n}}, y)$. Then there exists a subsequence of $(w_{n})$, still denoted by itself, and $w\in \x\setminus \{0\}$ such that
\begin{align*}
w_{n}\ri w  \mbox{ in } \x.
\end{align*}
Moreover, there exists $x_{0}\in \Lambda$ such that 
\begin{align*}
\e_{n}y_{\e_{n}}\ri x_{0} \quad \mbox{ and } \quad V(x_{0})=-V_{0}.
\end{align*}
\end{lem}
\begin{proof}
In what follows, we denote by $(y_{n})$ and $(u_{n})$, the sequences $(y_{\e_n})$ and $(u_{\e_n})$, respectively. Since each $u_{n}$ satisfies \eqref{AM2.8}, we can argue as in the proof of Lemma \ref{lemma2} and use Lemma \ref{lem2.3AM} and \eqref{equivalent} to deduce that $(u_{n})$ is bounded in $\x$. Thus $(w_{n})$ is bounded in $\x$ and
there exists $w\in \x\setminus \{0\}$ such that
\begin{align}\label{2.15AM}
w_{n}\rightharpoonup w \quad \mbox{ in } \x  \quad\mbox{ as } n\ri \infty,
\end{align}
and, by Lemma \ref{lem2.4AM},
\begin{align}\label{2.16AM}
\int_{B_{r}} w^{2}(x, 0)\, dx\geq \beta>0.
\end{align}
Next we show that $(\e_{n}y_{n})$ is bounded in $\R^N$. First of all, we prove that 
\begin{align}\label{Claim1}
{\rm dist}(\e_{n}y_{n}, \overline{\Lambda})\ri 0 \quad \mbox{ as } n\ri \infty.
\end{align}
If \eqref{Claim1} does not hold, there exists $\delta>0$ and a subsequence of $(\e_{n}y_{n})$, still denoted by itself, such that 
$$
{\rm dist}(\e_{n}y_{n}, \overline{\Lambda})\geq \delta \quad  \mbox{ for all } n\in \mathbb{N}.
$$
Then there is $R>0$ such that $B_{R}(\e_{n}y_{n})\subset \Lambda^{c}$ for all $n\in \mathbb{N}$. By the definition of $\x$ and using the fact that $w\geq 0$, we know that there exists $(\psi_{j})\subset \x$ such that $\psi_{j}\geq 0$, $\psi_{j}$ has compact support in $\overline{\mathbb{R}^{N+1}_{+}}$ and $\psi_{j}\ri w$ in $\x$ as $j\ri \infty$. Fix $j\in \mathbb{N}$. 
Taking $\psi_{j}$ as test function in $\langle J'_{\e}(u_{n}), \phi\rangle=0$ for $\phi\in X_{\e}$, we get
\begin{align}\label{2.17AM}
&\iint_{\R^{N+1}_{+}} y^{1-2s} (\nabla w_{n}\cdot\nabla \psi_{j}+m^{2}w_{n}\psi_{j})\, dx dy+\int_{\R^{N}} V(\e_{n}x+\e_{n}y_{n}) w_{n}(x, 0) \psi_{j}(x, 0)\, dx  \nonumber\\
&=\int_{\R^{N}} g(\e_{n}x+\e_{n}y_{n}, w_{n}(x, 0))\psi_{j}(x, 0)\, dx.
\end{align}
On the other hand, by the definition of $g_{\e}$ and $(g_{3})$, there holds
\begin{align*}
\int_{\R^{N}} g(\e_{n}x+\e_{n}y_{n}, w_{n}(x, 0))\psi_{j}(x, 0)\, dx&= \int_{B_{\frac{R}{\e_{n}}}} g(\e_{n}x+\e_{n}y_{n}, w_{n}(x, 0))\psi_{j}(x, 0)\, dx  \\
&+\int_{B^{c}_{\frac{R}{\e_{n}}}} g(\e_{n}x+\e_{n}y_{n}, w_{n}(x, 0))\psi_{j}(x, 0)\, dx \\
&\leq \frac{V_{1}}{\kappa} \int_{B_{\frac{R}{\e_{n}}}} w_{n}(x, 0) \psi_{j}(x, 0)\, dx+\int_{B^{c}_{\frac{R}{\e_{n}}}} f(w_{n}(x, 0))\psi_{j}(x, 0)\, dx.
\end{align*}
Then, using $(V_1)$ and \eqref{2.17AM}, we can see that
\begin{align*}
&\iint_{\R^{N+1}_{+}} y^{1-2s} (\nabla w_{n}\cdot\nabla \psi_{j}+m^{2}w_{n}\psi_{j})\, dx dy-V_{1}\left(1+\frac{1}{\kappa}\right) \int_{\R^{N}} w_{n}(x, 0) \psi_{j}(x, 0)\, dx  \\
&\leq \int_{B_{\frac{R}{\e_{n}}}^{c}} f(w_{n}(x, 0))\psi_{j}(x, 0)\, dx.
\end{align*}
Taking into account that $\psi_{j}$ has compact support, $\e_{n}\ri 0$, the growth assumptions on $f$, and \eqref{2.15AM}, we deduce that, as $n\ri \infty$,
\begin{align*}
\int_{B_{\frac{R}{\e_{n}}}^{c}} f(w_{n}(x, 0))\psi_{j}(x, 0)\, dx\ri 0,
\end{align*} 
and
\begin{align*}
&\iint_{\R^{N+1}_{+}} y^{1-2s} (\nabla w_{n}\cdot\nabla \psi_{j}+m^{2}w_{n}\psi_{j})\, dx dy-V_{1}\left(1+\frac{1}{\kappa}\right) \int_{\R^{N}} w_{n}(x, 0) \psi_{j}(x, 0)\, dx   \\
&\ri \iint_{\R^{N+1}_{+}} y^{1-2s} (\nabla w\cdot\nabla \psi_{j}+m^{2}w\psi_{j})\, dx dy-V_{1}\left(1+\frac{1}{\kappa}\right) \int_{\R^{N}} w(x, 0) \psi_{j}(x, 0)\, dx.
\end{align*}
The previous relations of limits give
\begin{align*}
\iint_{\R^{N+1}_{+}} y^{1-2s} (\nabla w\cdot\nabla \psi_{j}+m^{2}w\psi_{j})\, dx dy-V_{1}\left(1+\frac{1}{\kappa}\right) \int_{\R^{N}} w(x, 0) \psi_{j}(x, 0)\, dx\leq 0,
\end{align*}
and passing to the limit as $j\ri \infty$ we find
\begin{align*}
\|w\|^{2}_{\x}-V_{1}\left(1+\frac{1}{\kappa}\right) |w(\cdot, 0)|^{2}_{2}\leq 0.
\end{align*}
Thus \eqref{m-ineq} and $\kappa>\frac{V_{1}}{m^{2s}-V_{1}}$ yield 
$$
0\leq \left(1-\frac{V_{1}}{m^{2s}} \left(1+\frac{1}{\kappa}\right) \right)\|w\|^{2}_{\x}\leq 0,
$$
which implies $w\equiv 0$ in $\R^{N}$ and this is in contrast with \eqref{2.16AM}. Consequently, there exist a subsequence of $(\e_{n}y_{n})$, still denoted by itself, and $x_{0}\in \overline{\Lambda}$ such that $\e_{n}y_{n}\ri x_{0}$ as $n\ri \infty$. Next we prove that $x_{0}\in \Lambda$.

Using $(g_2)$ and \eqref{2.17AM}, we know that
\begin{align*}
\iint_{\R^{N+1}_{+}} y^{1-2s} (\nabla w_{n}\cdot\nabla \psi_{j}+m^{2}w_{n}\psi_{j})\, dx dy+\int_{\R^{N}} V(\e_{n}x+\e_{n}y_{n}) w_{n}(x, 0) \psi_{j}(x, 0)\, dx \leq \int_{\R^{N}} f(w_{n}(x, 0))\psi_{j}(x, 0)\, dx.
\end{align*}
Letting $n\ri \infty$ and using \eqref{2.15AM} and the continuity of $V$, we find
\begin{align*}
\iint_{\R^{N+1}_{+}} y^{1-2s} (\nabla w\cdot\nabla \psi_{j}+m^{2}w\psi_{j})\, dx dy+\int_{\R^{N}} V(x_{0}) w(x, 0) \psi_{j}(x, 0)\, dx \leq \int_{\R^{N}} f(w(x, 0)) \psi_{j}(x, 0)\, dx.
\end{align*}
By passing to the limit as $j\ri \infty$, we obtain 
\begin{align*}
\iint_{\R^{N+1}_{+}} y^{1-2s} (|\nabla w|^{2}+m^{2}w^{2})\, dx dy+\int_{\R^{N}} V(x_{0}) w^{2}(x, 0)\, dx \leq \int_{\R^{N}} f(w(x, 0)) w(x, 0)\, dx.
\end{align*}
Hence there exists $t_{1}\in (0, 1)$ such that $t_{1}w\in \mathcal{M}_{V(x_{0})}$. In view of Lemma \ref{lem2.3AM}, we have
\begin{align*}
d_{V(x_{0})}\leq L_{V(x_{0})}(t_{1}w)\leq \liminf_{n\ri \infty} J_{\e_{n}}(u_{n})=\liminf_{n\ri \infty} c_{\e_{n}}\leq d_{V(0)}
\end{align*}
from which $d_{V(x_{0})}\leq d_{V(0)}$ and thus $V(x_{0})\leq V(0)=-V_{0}$.
Since $-V_{0}=\inf_{x\in \overline{\Lambda}} V(x)$, we achieve $V(x_{0})=-V_{0}$. Using $(V_2)$, we conclude that $x_{0}\in \Lambda$.

Finally, we show that $w_{n}\ri w$ in $\x$ as $n\ri \infty$. For all $n\in \mathbb{N}$ and $x\in \R^{N}$, 
define
$$
\tilde{\Lambda}_{n} := \frac{\Lambda - \e_{n}\tilde{y}_{n}}{\e_{n}},
$$ 
and
\begin{align*}
&\tilde{\chi}_{n}^{1}(x):= \left\{
\begin{array}{ll}
1 \, &\mbox{ if } x\in \tilde{\Lambda}_{n},\\
0 \, &\mbox{ if } x\in \tilde{\Lambda}^{c}_{n}, 
\end{array}
\right.\\
&\tilde{\chi}_{n}^{2}(x):= 1- \tilde{\chi}_{n}^{1}(x).
\end{align*}
Let us also consider the following functions for $x\in \R^{N}$ and $n\in \mathbb{N}$:
\begin{align*}
&h_{n}^{1}(x):= \left(\frac{1}{2}-\frac{1}{\theta}\right) (V(\e_{n}x+ \e_{n}y_{n})+V_{1}) w^{2}_{n}(x, 0) \tilde{\chi}_{n}^{1}(x),\\
&h^{1}(x):=\left(\frac{1}{2}-\frac{1}{\theta}\right) (V(x_{0})+V_{1}) w^{2}(x, 0), \\
&h_{n}^{2}(x)\!\!:=\!\!\left[ \left(\frac{1}{2}-\frac{1}{\theta}\right) (V(\e_{n}x+ \e_{n}y_{n})+V_{1}) w^{2}_{n}(x, 0) + \frac{1}{\theta} g(\e_{n}x+ \e_{n}y_{n}, w_{n}(x, 0)) w_{n}(x, 0) - G(\e_{n}x+ \e_{n}y_{n}, w_{n}(x, 0))\right] \tilde{\chi}_{n}^{2}(x) \\
&\quad \quad \, \, \, \geq \left(\left(\frac{1}{2}-\frac{1}{\theta}\right) -\frac{1}{2\kappa}\right)(V(\e_{n}x+ \e_{n}\tilde{y}_{n})+V_{1}) w_{n}^{2}(x, 0) \tilde{\chi}_{n}^{2}(x), \\
&h_{n}^{3}(x):= \left(\frac{1}{\theta} g(\e_{n}x+ \e_{n}y_{n}, w_{n}(x, 0)) w_{n}(x, 0) - G(\e_{n}x+ \e_{n}y_{n}, w_{n}(x, 0))\right) \tilde{\chi}_{n}^{1}(x) \\
&\quad \quad \, \, \, =\left[\frac{1}{\theta} \left(f(w_{n}(x, 0))w_{n}(x, 0)- F(w_{n}(x, 0))\right) \right] \tilde{\chi}_{n}^{1}(x),  \\
&h^{3}(x):= \frac{1}{\theta} \left(f(w(x, 0))w(x, 0)- F(w(x, 0))\right). 
\end{align*}
From $(f_3)$, $(g_3)$, $(V_1)$ and our choice of $\kappa$, we see that the above functions are nonnegative in $\R^{N}$.
Since
\begin{align*}
&w_{n}(x, 0) \ri w(x, 0)\quad \mbox{ a.e. } x\in \R^{N}, \\
&\e_{n}y_{n}\ri x_{0}\in \Lambda,
\end{align*}
as $n\ri \infty$, we get
\begin{align*}
&\tilde{\chi}_{n}^{1}(x)\ri 1, \, h_{n}^{1}(x)\ri h^{1}(x), \, h_{n}^{2}(x)\ri 0 \, \mbox{ and } \, h_{n}^{3}(x)\ri h^{3}(x) \, \mbox{ a.e. } x\in \R^{N}. 
\end{align*}
Hence, observing that $\|\cdot\|^{2}_{\x}-V_{1} |\cdot|_{2}^{2}$ is weakly lower semicontinuous,  and using Fatou's lemma and the invariance of $\R^{N}$ by translation, we have
\begin{align*}
d_{V(0)} &\geq \limsup_{n\ri \infty} c_{\e_{n}} = \limsup_{n\ri \infty} \left( J_{\e_{n}}(u_{n}) - \frac{1}{\theta} \langle J'_{\e_{n}}(u_{n}), u_{n}\rangle \right)\\
&\geq \limsup_{n\ri \infty} \left\{\left(\frac{1}{2}-\frac{1}{\theta} \right)\left[\|w_{n}\|^{2}_{\x}-V_{1}|w_{n}(\cdot, 0)|^{2}_{2}\right]+ \int_{\R^{N}} (h_{n}^{1}+ h_{n}^{2}+ h_{n}^{3}) \, dx\right\}\\
&\geq \liminf_{n\ri \infty} \left\{\left(\frac{1}{2}-\frac{1}{\theta} \right)\left[\|w_{n}\|^{2}_{\x}-V_{1}|w_{n}(\cdot, 0)|^{2}_{2}\right]+ \int_{\R^{N}} (h_{n}^{1}+ h_{n}^{2}+ h_{n}^{3}) \, dx\right\} \\
&\geq \left(\frac{1}{2}-\frac{1}{\theta} \right)\left[\|w\|^{2}_{\x}-V_{1}|w(\cdot, 0)|^{2}_{2}\right]+ \int_{\R^{N}} (h^{1}+ h^{3}) \, dx
= d_{V(0)}.
\end{align*}
Accordingly,
\begin{align}\label{2.19AM}
\lim_{n\ri \infty}\|w_{n}\|^{2}_{\x}-V_{1}|w_{n}(\cdot, 0)|^{2}_{2}=\|w\|^{2}_{\x}-V_{1}|w(\cdot, 0)|^{2}_{2},
\end{align}
and 
\begin{align*}
h_{n}^{1}\ri h^{1}, \, h_{n}^{2}\ri 0 \, \mbox{ and }\, h_{n}^{3}\ri h^{3} \, \mbox{ in } \, L^{1}(\R^{N}). 
\end{align*}
Then,
\begin{align*}
\lim_{n\ri \infty} \int_{\R^{N}} (V(\e_{n} x+ \e_{n}y_{n})+V_{1}) w^{2}_{n}(x, 0) \, dx = \int_{\R^{N}} (V(x_{0})+V_{1}) w^{2}(x, 0) \, dx, 
\end{align*}
which implies that
\begin{align}\label{2.20AM}
\lim_{n\ri \infty} |w_{n}(\cdot, 0)|_{2}^{2}= |w(\cdot, 0)|_{2}^{2}. 
\end{align}
Putting together  \eqref{2.19AM} and \eqref{2.20AM}, and using the fact that $\x$ is a Hilbert space, we attain
\begin{align*}
\|w_{n}-w\|_{\x}\ri 0 \quad \mbox{ as } n\ri \infty.
\end{align*}
This ends the proof of the lemma.
\end{proof}

\section{The proof of Theorem \ref{thm1}}
In this last section we give the proof of Theorem \ref{thm1}. We start by proving the  following lemma which will be crucial to study the behavior of maximum points of the solutions. The proof is based on a variant of the Moser iteration argument \cite{Moser}. 
\begin{lem}\label{moser}
Let $(w_{n})$ be the sequence defined as in Lemma \ref{lem2.5AM}. Then, $w_{n}(\cdot, 0)\in L^{\infty}(\R^{N})$  and there exists $C>0$ such that
\begin{align*}
|w_{n}(\cdot, 0)|_{\infty}\leq C \quad \mbox{ for all } n\in \mathbb{N}.
\end{align*}
Moreover, $w_{n}\in L^{\infty}(\R^{N+1}_{+})$ and there exists $K>0$ such that
$$
\|w_{n}\|_{L^{\infty}(\R^{N+1}_{+})}\leq K \quad \mbox{ for all } n\in \mathbb{N}.
$$ 
\end{lem}
\begin{proof}
We note that $w_{n}$ is a weak solution to
\begin{align}\label{traslato}
\left\{
\begin{array}{ll}
-\dive(y^{1-2s} \nabla w_{n})+m^{2}y^{1-2s}w_{n}=0 &\mbox{ in } \R^{N+1}_{+}, \\
\frac{\partial w_{n}}{\partial \nu^{1-2s}}=-V(\e_{n}x+\e_{n}y_{n})w_{n}(\cdot, 0)+g(\e_{n}x+\e_{n}y_{n}, w_{n}(\cdot, 0)) &\mbox{ on } \R^{N}.
\end{array}
\right.
\end{align}
For each $n\in \mathbb{N}$ and $L>0$, let $w_{n,L}:=\min\{w_{n},L\}$ and $z_{n, L}:=w_{n}w_{n, L}^{2\beta}$ , where $\beta>0$ will be chosen later.
Taking $z_{L, n}$ as test function in the weak formulation of \eqref{traslato}, we deduce that 
\begin{align}\label{conto1JMP}
\iint_{\mathbb{R}^{N+1}_{+}}  &y^{1-2s}w^{2\beta}_{n,L}(|\nabla w_{n}|^{2}+m^{2}w^{2}_{n}) \, dxdy+\iint_{D_{n, L}} 2\beta y^{1-2s}w^{2\beta}_{n, L} |\nabla w_{n}|^{2} \, dx dy  \nonumber \\
&= -\int_{\mathbb{R}^{N}} V(\e_{n}x+\e_{n}y_{n}) w^{2}_{n}(x,0) w^{2\beta}_{n, L}(x,0) \,dx+ \int_{\mathbb{R}^{N}} g(\e_{n}x+\e_{n}y_{n}, w_{n}(x, 0)) w_{n}(x,0)w^{2\beta}_{n, L}(x,0) \,dx, 
\end{align}
where $D_{n, L}:=\{(x,y)\in \mathbb{R}^{N+1}_{+}: w_{n}(x, y)\leq L\}$. 
It is easy to check that
\begin{align}\label{conto2JMP}
\iint_{\mathbb{R}^{N+1}_{+}} &y^{1-2s}|\nabla (w_{n}w_{n, L}^{\beta})|^{2} \,dxdy =\iint_{\mathbb{R}^{N+1}_{+}} y^{1-2s}w_{n, L}^{2\beta} |\nabla w_{n}|^{2} \,dxdy \nonumber \\
&\quad+\iint_{D_{n, L}} (2\beta+\beta^{2}) y^{1-2s}w_{n, L}^{2\beta} |\nabla w_{n}|^{2} \,dxdy.
\end{align}
Then, putting together (\ref{conto1JMP}), (\ref{conto2JMP}), $(V_{1})$, $(f_1)$-$(f_2)$, $(g_1)$-$(g_2)$, we get 
\begin{align}\label{S1JMP}
&\|w_{n}w_{n, L}^{\beta}\|_{\x}^{2}=\iint_{\mathbb{R}^{N+1}_{+}} y^{1-2s}(|\nabla (w_{n}w_{n, L}^{\beta})|^{2}+m^{2}w^{2}_{n}w^{2\beta}_{n, L}) \,dxdy \nonumber \\
&=\iint_{\mathbb{R}^{N+1}_{+}} y^{1-2s}w_{n, L}^{2\beta} (|\nabla w_{n}|^{2}+m^{2}w^{2}_{n}) \,dxdy+\iint_{D_{n, L}} 2\beta \left(1+\frac{\beta}{2}\right) y^{1-2s}w_{n, L}^{2\beta} |\nabla w_{n}|^{2} \,dxdy \nonumber \\
&\leq c_{\beta} \left[\iint_{\mathbb{R}^{N+1}_{+}} y^{1-2s}w_{n, L}^{2\beta} (|\nabla w_{n}|^{2}+m^{2}w^{2}_{n}) \,dxdy+\iint_{D_{n, L}} 2\beta y^{1-2s}w_{n, L}^{2\beta} |\nabla w_{n}|^{2} \,dxdy\right] \nonumber \\
&=c_{\beta} \left[-\int_{\mathbb{R}^{N}} V(\e_{n}x+\e_{n}y_{n}) w^{2}_{n}(x,0) w^{2\beta}_{n, L}(x,0) \,dx+ \int_{\mathbb{R}^{N}} g(\e_{n}x+\e_{n}y_{n}, w_{n}(x, 0)) w_{n}(x,0)w^{2\beta}_{n, L}(x,0) \,dx\right] \nonumber \\
&\leq c_{\beta} \left[\int_{\R^{N}} (V_{1}+1)w^{2}_{n}(x,0) w^{2\beta}_{n, L}(x,0)+ C_{1} w^{p+1}_{n}(x,0)w^{2\beta}_{n, L}(x,0) \,dx   \right]
\end{align}
where 
\begin{align*}
c_{\beta}:=1+\frac{\beta}{2}>0.
\end{align*}
Now, we prove that there exist a constant $c>0$ independent of $n$, $L$, $\beta$, and $h_{n}\in L^{N/2s}(\mathbb{R}^{N})$, $h_{n}\geq 0$ and independent of  $L$ and $\beta$, such that
\begin{align}\label{S2JMP}
(V_{1}+1)w^{2}_{n}(\cdot,0) w^{2\beta}_{n, L}(\cdot,0)+ C_{1}w^{p+1}_{n}(\cdot,0)w^{2\beta}_{n, L}(\cdot,0)\leq (c+h_{n})w^{2}_{n}(\cdot,0) w_{n, L}^{2\beta}(\cdot,0) \quad \mbox{ on } \mathbb{R}^{N}.
\end{align}
Firstly, we notice that
\begin{align*}
&(V_{1}+1)w^{2}_{n}(\cdot,0) w^{2\beta}_{n, L}(\cdot,0)+ C_{1} w^{p+1}_{n}(\cdot,0)w^{2\beta}_{n, L}(\cdot,0) \\
&\quad \leq (V_{1}+1)w^{2}_{n}(\cdot,0) w^{2\beta}_{n, L}(\cdot,0)+C_{1} w_{n}^{p-1}(\cdot,0)w^{2}_{n}(\cdot,0) w_{n, L}^{2\beta}(\cdot,0) \quad \mbox{ on } \mathbb{R}^{N}.
\end{align*}
Moreover,
\begin{align}\label{IPHONE12}
w_{n}^{p-1}(\cdot,0)\leq 1+h_{n} \quad \mbox{ on } \mathbb{R}^{N},
\end{align}
where $h_{n}:=\chi_{\{w_{n}(\cdot,0)>1\}}w_{n}^{p-1}(\cdot,0)\in L^{N/2s}(\mathbb{R}^{N})$.
In fact, we can observe that
$$
w_{n}^{p-1}(\cdot,0)=\chi_{\{w_{n}(\cdot,0)\leq 1\}}w_{n}^{p-1}(\cdot,0)+\chi_{\{w_{n}(\cdot,0)>1\}}w_{n}^{p-1}(\cdot,0)\leq 1+\chi_{\{w_{n}(\cdot,0)>1\}}w_{n}^{p-1}(\cdot,0) \quad \mbox{ on } \mathbb{R}^{N}.
$$
If $(p-1)\frac{N}{2s}<2$ then, recalling that $(w_{n}(\cdot, 0))$ is bounded in $\h$, we have
$$
\int_{\mathbb{R}^{N}} \chi_{\{w_{n}(\cdot,0)>1\}}(w_{n}(x,0))^{\frac{N}{2s}(p-1)} \,dx \leq \int_{\mathbb{R}^{N}} \chi_{\{w_{n}(\cdot,0)>1\}} w^{2}_{n}(x,0) \, dx\leq C, \quad \mbox{ for all } n\in \mathbb{N}.
$$
If $2\leq (p-1)\frac{N}{2s}$, we deduce that $(p-1)\frac{N}{2s}\in [2,2^{*}_{s}]$, and by Theorem \ref{Sembedding} and the boundedness of $(w_{n})$ in $\x$, we find
$$
\int_{\R^{N}} \chi_{\{w_{n}(\cdot,0)>1\}} (w_{n}(x,0))^{\frac{N}{2s}(p-1)} \,dx\leq C\|w_{n}\|_{\x}^{\frac{N}{2s}(p-1)}\leq C,
$$
for some $C>0$ depending only on $N$, $s$ and $p$. Hence, \eqref{IPHONE12} holds.
Taking into account (\ref{S1JMP}), (\ref{S2JMP}), and \eqref{IPHONE12}, we obtain that
\begin{equation*}
\|w_{n}w_{n, L}^{\beta}\|_{\x}^{2}\leq c_{\beta} \int_{\mathbb{R}^{N}} (c+h_{n}(x))w^{2}_{n}(x,0)w^{2\beta}_{n, L}(x,0) \,dx,
\end{equation*}
and, by applying Fatou's lemma and the monotone convergence theorem, we can pass to the  limit as $L\ri \infty$ to get
\begin{equation}\label{i1JMP}
\|w_{n}^{\beta+1}\|_{\x}^{2}\leq cc_{\beta} \int_{\mathbb{R}^{N}} w_{n}^{2(\beta +1)}(x,0) \,dx + c_{\beta}\int_{\mathbb{R}^{N}}  h_{n}(x) w_{n}^{2(\beta +1)}(x,0)\,dx.
\end{equation}
Fix $M>1$ and let $\Omega_{1,n}:=\{h_{n}\leq M\}$ and $\Omega_{2,n}:=\{h_{n}>M\}$.
Then, by using the H\"older inequality,
\begin{align}\label{i2JMP}
\int_{\mathbb{R}^{N}}  h_{n}(x) w_{n}^{2(\beta +1)}(x,0) \,dx&=\int_{\Omega_{1, n}}  h_{n}(x) w_{n}^{2(\beta +1)}(x,0) \, dx+\int_{\Omega_{2, n}}  h_{n}(x) w_{n}^{2(\beta +1)}(x,0) \, dx  \nonumber \\
&\leq M |w_{n}^{\beta+1}(\cdot, 0)|_{2}^{2}+\epsilon(M)|w_{n}^{\beta+1}(\cdot, 0)|_{\2}^{2},
\end{align}
where 
\begin{align*}
\epsilon(M):=\sup_{n\in \mathbb{N}}\left(\int_{\Omega_{2,n}} h^{N/2s}_{n} dx \right)^{\frac{2s}{N}}\rightarrow 0 \quad \mbox{ as } M\rightarrow \infty,
\end{align*}
due to the fact that $w_{n}(\cdot, 0)\ri w(\cdot, 0)$ in $\h$.
In view of (\ref{i1JMP}) and (\ref{i2JMP}), we have
\begin{equation}\label{regvJMP}
\|w_{n}^{\beta+1}\|_{\x}^{2}\leq c_{\beta}(c+M)|w_{n}^{\beta+1}(\cdot, 0)|_{2}^{2}+c_{\beta}\epsilon(M)|w_{n}^{\beta+1}(\cdot, 0)|_{\2}^{2}.
\end{equation}
We note that Theorem \ref{Sembedding} yields
\begin{equation}\label{S3JMP}
|w_{n}^{\beta+1}(\cdot, 0)|_{\2}^{2}\leq C^{2}_{2^{*}_{s}}\|w_{n}^{\beta+1}\|_{\x}^{2}.
\end{equation}
Then, choosing $M$ large so that 
$$
\epsilon(M) c_{\beta} C^{2}_{2^{*}_{s}}<\frac{1}{2},
$$
and using $(\ref{regvJMP})$ and $(\ref{S3JMP})$, we obtain that
\begin{equation}\label{iterJMP}
|w_{n}^{\beta+1}(\cdot,0)|_{\2}^{2}\leq 2 C^{2}_{2^{*}_{s}} c_{\beta}(c+M)|w_{n}^{\beta+1}(\cdot,0)|^{2}_{2}.
\end{equation}
Then we can start a bootstrap argument: since $w_{n}(\cdot,0)\in L^{2^{*}_{s}}(\mathbb{R}^{N})$ and $|w_{n}(\cdot, 0)|_{\2}\leq C$ for all $n\in \mathbb{N}$, we can apply (\ref{iterJMP}) with $\beta_{1}+1=\frac{N}{N-2s}$ to deduce that $w_{n}(\cdot,0)\in L^{(\beta_{1}+1)\2}(\mathbb{R}^{N})=L^{\frac{2N^{2}}{(N-2s)^{2}}}(\mathbb{R}^{N})$. Applying again (\ref{iterJMP}), after $k$ iterations, we find $w_{n}(\cdot,0)\in L^{\frac{2N^{k}}{(N-2s)^{k}}}(\mathbb{R}^{N})$, and so $w_{n}(\cdot,0)\in L^{q}(\mathbb{R}^{N})$ for all $q\in[2,\infty)$ and $|w_{n}(\cdot, 0)|_{q}\leq C$ for all $n\in \mathbb{N}$.

Now we prove that actually $w_{n}(\cdot,0)\in L^{\infty}(\mathbb{R}^{N})$.
Since $w_{n}(\cdot,0)\in L^{q}(\mathbb{R}^{N})$ for all $q\in[2,\infty)$, we have that $h_{n}\in L^{\frac{N}{s}}(\mathbb{R}^{N})$ and $|h_{n}|_{\frac{N}{s}}\leq D$ for all $n\in \mathbb{N}$.
Then, by the generalized H\"older inequality and Young's inequality with $\lambda>0$, we can see that for all $\lambda>0$
\begin{align*}
\int_{\mathbb{R}^{N}} h_{n}(x) w_{n}^{2(\beta+1)}(x,0) \,dx &\leq |h_{n}|_{\frac{N}{s}} |w_{n}^{\beta+1}(\cdot,0)|_{2}  |w_{n}^{\beta+1}(\cdot, 0)|_{\2}  \nonumber \\
&\leq D\left(\lambda |w_{n}^{\beta+1}(\cdot,0)|_{2}^{2}+\frac{1}{\lambda}  |w_{n}^{\beta+1}(\cdot,0)|_{\2}^{2}\right).
\end{align*}
Consequently, using (\ref{i1JMP}) and (\ref{S3JMP}), we deduce that
\begin{align}\label{i3JMP}
&|w_{n}^{\beta+1}(\cdot, 0)|_{\2}^{2}\leq C^{2}_{2^{*}_{s}}\|w_{n}^{\beta+1}\|_{\x}^{2} \nonumber\\
&\leq  c_{\beta} C^{2}_{2^{*}_{s}} (c+ D \lambda) |w_{n}^{\beta+1}(\cdot,0)|_{2}^{2}+C^{2}_{2^{*}_{s}}\frac{c_{\beta}D}{\lambda}  |w_{n}^{\beta+1}(\cdot,0)|_{\2}^{2}.
\end{align}
Taking $\lambda>0$ such that
$$
\frac{c_{\beta}DC^{2}_{2^{*}_{s}}}{\lambda}=\frac{1}{2},
$$
we obtain that
\begin{align*}
|w_{n}^{\beta+1}(\cdot, 0)|_{\2}^{2} \leq  2c_{\beta}(c+D\lambda)C^{2}_{2^{*}_{s}}  |w_{n}^{\beta+1}(\cdot, 0)|_{2}^{2}= M_{\beta} |w_{n}^{\beta+1}(\cdot, 0)|_{2}^{2},
\end{align*}
where
\begin{align*}
M_{\beta}:=2c_{\beta}(c+D\lambda)C^{2}_{2^{*}_{s}}.
\end{align*}
Now we can control the dependence on $\beta$ of $M_{\beta}$ as follows:
$$
M_{\beta}\leq Cc^{2}_{\beta}\leq C(1+\beta)^{2}\leq M_{0}^{2}e^{2\sqrt{\beta+1}},
$$
for some $M_{0}>0$ independent of $\beta$. Consequently,
$$
|w_{n}(\cdot,0)|_{2^{*}_{s}(\beta+1)} \leq M_{0}^{\frac{1}{\beta+1}} e^{\frac{1}{\sqrt{\beta+1}}}|w_{n}(\cdot,0)|_{2(\beta+1)}.
$$
As before, iterating this last relation and choosing $\beta_{0}=0$ and $2(\beta_{j+1}+1)=2^{*}_{s}(\beta_{j}+1)$
we have that
$$
|w_{n}(\cdot,0)|_{2^{*}_{s}(\beta_{j}+1)} \leq M_{0}^{\sum_{i=0}^{j} \frac{1}{\beta_{i}+1}} e^{\sum_{i=0}^{j} \frac{1}{\sqrt{\beta_{i}+1}}}|w_{n}(\cdot,0)|_{2(\beta_{0}+1)}.
$$
We note that 
\begin{equation}\label{BETAj}
\beta_{j}=\left(\frac{N}{N-2s}\right)^{j}-1,
\end{equation}
so the series
$$
\sum_{i=0}^{\infty} \frac{1}{\beta_{i}+1} \quad \mbox{ and } \quad \sum_{i=0}^{\infty} \frac{1}{\sqrt{\beta_{i}+1}}
$$
are convergent. Recalling that $|w_{n}(\cdot, 0)|_{q}\leq C$ for all $n\in \mathbb{N}$ and $q\in [2, \infty)$, we get
$$
|w_{n}(\cdot,0)|_{\infty}=\lim_{j\rightarrow \infty} |w_{n}(\cdot,0)|_{2^{*}_{s}(\beta_{j}+1)}\leq M \quad \mbox{ for all } n\in \mathbb{N}.
$$
This proves the $L^{\infty}$-desired estimate for the trace.
At this point, we prove that there exists $R>0$ such that 
\begin{align}\label{COLACOLA1}
\|w_{n}\|_{L^{\infty}(\R^{N+1}_{+})}\leq R \quad \mbox{ for all } n\in \mathbb{N}.
\end{align} 
Using \eqref{i3JMP} with $\lambda=1$ and that $|w_{n}(\cdot, 0)|_{q}\leq C$ for all $q\in [2, \infty]$, we deduce that
\begin{align*}
\|w_{n}^{\beta+1}\|^{2}_{\x}\leq  \tilde{c}c_{\beta}C^{2(\beta+1)} \quad \mbox{ for all } n\in \mathbb{N},
\end{align*}
for some $\tilde{c}, C>0$ independent on $\beta$ and $n$. On the other hand, from \eqref{weightedE}, we obtain that
\begin{align*}
\left(\iint_{\R^{N+1}_{+}} y^{1-2s} w_{n}^{2\gamma(\beta+1)}\, dx dy\right)^{\frac{\beta+1}{2\gamma(\beta+1)}}=\|w_{n}^{\beta+1}\|_{L^{2\gamma}(\R^{N+1}_{+}, y^{1-2s})}\leq C_{*}\|w_{n}^{\beta+1}\|_{\x}
\end{align*}
which combined with the previous inequality yields
\begin{align*}
\|w_{n}\|_{L^{2\gamma(\beta+1)}(\R^{N+1}_{+}, y^{1-2s})}\leq  C'(\tilde{C}_{*}c_{\beta})^{\frac{1}{2(\beta+1)}} \quad \mbox{ for all } n\in \mathbb{N}.
\end{align*}
Since 
$$
\frac{1}{\beta+1}\log\left(1+\frac{\beta}{2}\right)\leq \frac{\beta}{2(\beta+1)}\leq \frac{1}{2} \quad \mbox{ for all } \beta>0,
$$
we can see that there exists $\bar{C}>0$ such that $C'(\tilde{C}_{*}c_{\beta})^{\frac{1}{2(\beta+1)}}\leq \bar{C}$ for all $\beta>0$, and so 
$$
\|w_{n}\|_{L^{2\gamma(\beta+1)}(\R^{N+1}_{+}, y^{1-2s})}\leq \bar{C} \quad \mbox{ for all } n\in \mathbb{N}, \beta>0.
$$
Now, fix $R>\bar{C}$ and define $\Sigma_{n}:=\{(x, y)\in \R^{N+1}_{+}: w_{n}(x, y)>R\}$. Hence, for all $n, j\in \mathbb{N}$, we get
\begin{align*}
\bar{C}\geq \left(\iint_{\R^{N+1}_{+}} y^{1-2s} w_{n}^{2\gamma(\beta_{j}+1)}\, dx dy\right)^{\frac{1}{2\gamma(\beta_{j}+1)}}\geq R^{\frac{\gamma(\beta_{j}+1)-1}{\gamma(\beta_{j}+1)}} \left(\iint_{\Sigma_{n}} y^{1-2s} w_{n}^{2}\, dx dy\right)^{\frac{1}{2\gamma(\beta_{j}+1)}},
\end{align*}
which yields
$$
\left( \frac{\bar{C}}{R} \right)^{\beta_{j}+1-\frac{1}{\gamma}}\geq \frac{1}{\bar{C}^{\frac{1}{\gamma}}} \left(\iint_{\Sigma_{n}} y^{1-2s} w_{n}^{2}\, dx dy\right)^{\frac{1}{2\gamma}},
$$
where $\beta_{j}$ is given in \eqref{BETAj}.
Letting $j\rightarrow \infty$, we have that $\beta_{j}\ri \infty$ and then
$$
\iint_{\Sigma_{n}} y^{1-2s} w_{n}^{2}\, dx dy=0  \quad\mbox{ for all } n\in \mathbb{N},
$$
which implies that $|\Sigma_{n}|=0$ for all $n\in \mathbb{N}$. Consequently, \eqref{COLACOLA1} holds true.
\end{proof}

\begin{lem}\label{lem2.6AM}
The sequence $(w_{n})$ satisfies $w_{n}(\cdot, 0)\ri 0$ as $|x|\ri \infty$ uniformly in $n\in \mathbb{N}$.
\end{lem}
\begin{proof}
By Lemma \ref{moser} and Proposition \ref{PROPFF}-$(iii)$, we obtain that each $w_{n}$ is continuous in $\overline{\R^{N+1}_{+}}$. Note that, by Lemma \ref{moser}, $w_{n}(\cdot, 0)\rightarrow w(\cdot, 0)$ in $L^{q}(\R^{N})$ for all $q\in [2, \infty)$. On the other hand, from \eqref{weightedE} and $w_{n}\ri w$ in $\x$, we have that $w_{n}\ri w$ in $L^{2\gamma}(\R^{N+1}_{+}, y^{1-2s})$. 
Fix $\bar{x}\in \R^{N}$. Using $(V_1)$ and \eqref{growthg}, we see that $w_{n}$ is a weak subsolution to 
\begin{equation*}
\left\{
\begin{array}{ll}
-{\rm div}(y^{1-2s} \nabla w_{n})+m^{2}y^{1-2s}w_{n}= 0  &\mbox{ in } Q_{1}(\bar{x},0):=B_{1}(\bar{x})\times (0,1), \\
\frac{\partial w_{n}}{\partial \nu^{1-2s}}= (V_{1}+\eta)w_{n}(\cdot, 0)+C_{\eta}w_{n}^{\2-1}(\cdot, 0) &\mbox{ on } B_{1}(\bar{x}),
\end{array}
\right.
\end{equation*}
where $\eta\in (0, m^{2s}-V_{1})$ is fixed. Applying Proposition \ref{PROPFF}-$(i)$ and observing that $L^{2\gamma}(A, y^{1-2s})\subset L^{2}(A, y^{1-2s})$ for any bounded set $A\subset \R^{N}$, we get
$$
0\leq \sup_{Q_{\frac{1}{2}}(\bar{x},0)}w_{n}\leq C(\|w_{n}\|_{L^{2\gamma}(Q_{1}(\bar{x},0), y^{1-2s})}+|w_{n}^{2^{*}_{s}-1}(\cdot, 0)|_{L^{q}(B_{1}(\bar{x}))}) \quad \mbox{ for all } n\in \mathbb{N},
$$
where $q>\frac{N}{2s}$ is fixed and $C>0$ is a constant depending only on $N, m, s, q, \gamma$ and independent of $n\in \mathbb{N}$ and $\bar{x}$. Note that $q(2^{*}_{s}-1)\in (2, \infty)$ because $N>2s$ and $q>\frac{N}{2s}$. Using the strong convergence of $(w_{n})$ in $L^{2\gamma}(\R^{N+1}_{+}, y^{1-2s})$ and $(w_{n}(\cdot, 0))$ in $L^{q}(\R^{N})$, respectively, 
we infer that $w_{n}(\bar{x}, 0)\rightarrow 0$ as $|\bar{x}|\rightarrow \infty$ uniformly in $n\in \mathbb{N}$.
\end{proof}

\begin{remark}
The proofs of Lemmas \ref{moser} and \ref{lem2.6AM} can be also adapted to the case $m=0$. In this way, we can give an alternative proof of Lemma $2.6$ in \cite{AM} which, on the contrary, is based on some properties of the kernel $\mathcal{K}_{s}(x)=\mathcal{F}^{-1}((|k|^{2s}+1)^{-1})(x)$ proved in \cite{FQT}.
\end{remark}

The next comparison principle for $(-\Delta+m^{2})^{s}$ in general domains will be very useful later.
\begin{thm}(Comparison principle)\label{Comparison}
Let $\Omega\subset \R^{N}$ be an open set, $\gamma<m^{2s}$, $u_{1}, u_{2}\in H^{s}(\R^{N})$ be such that $u_{1}\leq u_{2}$ in $\R^{N}\setminus \Omega$ and $(-\Delta+m^{2})^{s}u_{1}-\gamma u_{1}\leq (-\Delta+m^{2})^{s}u_{2}-\gamma u_{2}$ in $\Omega$, that is
\begin{align}\label{CANNUCCIA}
&(m^{2s}-\gamma)\int_{\R^{N}} u_{1}(x)v(x)\, dx+C(N, s) m^{\frac{N+2s}{2}} \iint_{\R^{N}} \frac{(u_{1}(x)-u_{1}(y))(v(x)-v(y))}{|x-y|^{\frac{N+2s}{2}}}K_{\frac{N+2s}{2}}(m|x-y|)\, dxdy \nonumber\\
&\leq (m^{2s}-\gamma)\int_{\R^{N}} u_{2}(x)v(x)\, dx+C(N, s) m^{\frac{N+2s}{2}} \iint_{\R^{N}} \frac{(u_{2}(x)-u_{2}(y))(v(x)-v(y))}{|x-y|^{\frac{N+2s}{2}}}K_{\frac{N+2s}{2}}(m|x-y|)\, dxdy
\end{align}
for all $v\in H^{s}(\R^{N})$ such that $v\geq 0$ in $\R^{N}$ and $v=0$ in $\R^{N}\setminus \Omega$.
Then, $u_{1}\leq u_{2}$ in $\R^{N}$.
\end{thm}
\begin{proof}
Set $w:=u_{1}-u_{2}$ and take $v=w^{+}$ as test function in \eqref{CANNUCCIA}. Then, observing that 
$$
|x^{+}-y^{+}|^{2}\leq (x-y)(x^{+}-y^{+}) \quad\mbox{ for all } x, y\in\R,
$$ 
we have
\begin{align*}
&(m^{2s}-\gamma)\int_{\R^{N}} v^{2}(x)\, dx+C(N,s) m^{\frac{N+2s}{2}} \iint_{\R^{N}} \frac{|v(x)-v(y)|^{2}}{|x-y|^{\frac{N+2s}{2}}}K_{\frac{N+2s}{2}}(m|x-y|)\, dxdy\\
&\leq(m^{2s}-\gamma)\int_{\R^{N}} w(x)v(x)\, dx+C(N,s) m^{\frac{N+2s}{2}} \iint_{\R^{N}} \frac{(w(x)-w(y))(v(x)-v(y))}{|x-y|^{\frac{N+2s}{2}}}K_{\frac{N+2s}{2}}(m|x-y|)\, dxdy\\
&\leq 0
\end{align*}
from which $v=0$ in $\R^{N}$ and thus we get the thesis.
\end{proof}

Now we have all tools to give the proof of our first main result of this work.
\begin{proof}[Proof of Theorem \ref{thm1}]
We begin by proving that there exists $\e_{0}>0$ such that, for any $\e \in (0, \e_{0})$ and any mountain pass solution $u_{\e} \in X_{\e}$ of \eqref{MEP}, it holds 
\begin{equation}\label{inftyAMPA}
|u_{\e}(\cdot, 0)|_{L^{\infty}(\Lambda^{c}_{\e})}<a. 
\end{equation}
Assume by contradiction that for some subsequence $(\e_{n})$ such that $\e_{n}\rightarrow 0$, we can find $u_{n}:=u_{\e_{n}}\in X_{\e_{n}}$ such that $J_{\e_{n}}(u_{n})=c_{\e_{n}}$, $J'_{\e_{n}}(u_{n})=0$ and 
\begin{equation}\label{eeeAMPA}
|u_{n}(\cdot, 0)|_{L^{\infty}(\Lambda^{c}_{\e_{n}})}\geq a.
\end{equation} 
In view of Lemma \ref{lem2.5AM}, we can find $(y_{n})\subset \mathbb{R}^{N}$ such that $w_{n}(x, y):=u_{n}(x+y_{n}, y)\rightarrow w$ in $\x$ and $\e_{n}y_{n}\rightarrow x_{0}$ for some $x_{0}\in \Lambda$ such that $V(x_{0})=-V_{0}$. 

Now, if we choose $r>0$ such that $B_{r}(x_{0})\subset B_{2r}(x_{0})\subset \Lambda$, we can see that $B_{\frac{r}{\e_{n}}}(\frac{x_{0}}{\e_{n}})\subset \Lambda_{\e_{n}}$. Then, for any $x\in B_{\frac{r}{\e_{n}}}(y_{n})$ it holds
\begin{align*}
\left|x - \frac{x_{0}}{\e_{n}}\right| \leq |x- y_{n}|+ \left|y_{n} - \frac{x_{0}}{\e_{n}}\right|<\frac{1}{\e_{n}}(r+o_{n}(1))<\frac{2r}{\e_{n}}\quad \mbox{ for } n \mbox{ sufficiently large. }
\end{align*}
Therefore, 
\begin{equation}\label{ernAMPA}
\Lambda^{c}_{\e_{n}}\subset B^{c}_{\frac{r}{\e_{n}}}(y_{n})
\end{equation}
for any $n$ big enough.
Using Lemma \ref{lem2.6AM}, we see that 
\begin{equation}\label{freddiAMPA}
w_{n}(x, 0)\rightarrow 0 \quad \mbox{ as } |x|\rightarrow \infty \quad \mbox{ uniformly in } n\in \mathbb{N}.
\end{equation}
Therefore, there exists $R>0$ such that 
$$
w_{n}(x, 0)<a \quad \mbox{ for any } |x|\geq R, \, n\in \mathbb{N}.
$$ 
Hence, $u_{n}(x, 0)=w_{n}(x-y_{n}, 0)<a$ for any $x\in B^{c}_{R}(y_{n})$ and $n\in \mathbb{N}$. On the other hand, by \eqref{ernAMPA}, there exists $n_{0} \in \mathbb{N}$ such that for any $n\geq n_{0}$ we have
$$
\Lambda^{c}_{\e_{n}}\subset B^{c}_{\frac{r}{\e_{n}}}(y_{n})\subset B^{c}_{R}(y_{n}),
$$
which implies that $u_{n}(x, 0)<a$ for any $x\in \Lambda^{c}_{\e_{n}}$ and $n\geq n_{0}$. This is impossible according to  \eqref{eeeAMPA}. 
Since $u_{\e}\in X_{\e}$ satisfies \eqref{inftyAMPA}, by the definition of $g$ it follows that $u_{\e}$ is a solution of \eqref{EP} for $\e \in (0, \e_{0})$. From Proposition \ref{PROPFF}-$(ii)$, we conclude that $u_{\e}(\cdot, 0)>0$ in $\R^{N}$.

In what follows, we study the behavior of the maximum points of solutions to problem \eqref{P}. 
Take $\e_{n}\rightarrow 0$ and let $(u_{n})\subset X_{\e_{n}}$ be a sequence of solutions to \eqref{MEP} as above. 
Consider the translated sequence $w_{n}(x, y):=u_{n}(x+y_{n}, y)$, where $(y_{n})$ is given by Lemma \ref{lem2.5AM}. 
Let us prove that there exists $\delta>0$ such that 
\begin{equation}\label{DELTA}
|w_{n}(\cdot, 0)|_{\infty}\geq \delta \quad\mbox{ for all } n\in \mathbb{N}.
\end{equation} 
Assume by contradiction  that $|w_{n}(\cdot, 0)|_{\infty}\ri 0$ as $n\ri \infty$. Using $(f_1)$, we can find $\nu\in \mathbb{N}$ such that
$$
\frac{f(|w_{n}(\cdot, 0)|_{\infty})}{|w_{n}(\cdot, 0)|_{\infty}}<\frac{V_{1}}{\kappa} \quad \mbox{ for all } n\geq \nu.
$$
From $\langle J'_{\e_{n}}(u_{n}), u_{n}\rangle=0$, $(g_2)$ and $(f_4)$, we can see that for all $n\geq \nu$
\begin{align*}
 \|w_{n}\|_{\x}^{2}-V_{1}|w_{n}(\cdot, 0)|_{2}^{2}&= \|u_{n}\|_{\x}^{2}-V_{1}|u_{n}(\cdot, 0)|_{2}^{2}\\
 &\leq\int_{\mathbb{R}^{N}} f(u_{n}(x, 0)) u_{n}(x, 0) \,dx \\
 &\leq \int_{\mathbb{R}^{N}} \frac{f(|w_{n}(\cdot, 0)|_{\infty})}{|w_{n}(\cdot, 0)|_{\infty}} w^{2}_{n}(x, 0) \,dx \\
 &\leq \frac{V_{1}}{\kappa} \int_{\mathbb{R}^{N}} w_{n}^{2}(x, 0) \, dx
\end{align*}
which combined with \eqref{m-ineq} yields 
$$
\left(1-\frac{V_{1}}{m^{2s}}\left(1+\frac{1}{\kappa}\right)  \right)\|w_{n}\|^{2}_{\x}\leq 0. 
$$
Since $\kappa>\frac{V_{1}}{m^{2s}-V_{1}}$, we get $\|w_{n}\|_{\x}=0$ for all $n\geq \nu$, which is a contradiction. 
Therefore, if $q_{n}$ is a global maximum point of $w_{n}(\cdot, 0)$, we deduce from Lemma \ref{moser} and \eqref{DELTA} that there exists $R_{0}>0$ such that $|q_{n}|<R_{0}$ for all $n\in \mathbb{N}$. Thus $x_{n}:=q_{n}+y_{n}$ is a global maximum point of $u_{n}(\cdot, 0)$, and $\e_{n}x_{n}\ri x_{0}\in M$. This fact combined with the continuity of $V$ yields 
$$
\lim_{n\rightarrow \infty} V(\e_{n} x_{n})=V(x_{0})=-V_{0}.
$$

Finally, we prove a decay estimate for $u_{n}(\cdot, 0)$. 
Using $(f_1)$, the definition of $g$ and \eqref{freddiAMPA}, we can find $R_{1}>2$ sufficiently large such that
\begin{align}\label{TERESA}
g(\e_{n} x+\e_{n}y_{n}, w_{n}(x,0))w_{n}(x, 0)\leq \delta w_{n}^{2}(x, 0) \quad \mbox{ for } |x|> R_{1},
\end{align}
where $\delta\in (0, m^{2s}-V_{1})$ is fixed.
Pick a smooth cut-off function $\phi$  defined in $\R^{N}$ such that $0\leq \phi\leq 1$, $\phi(x)=0$ for $|x|\geq 1$, and $\phi\not\equiv 0$. By Riesz representation theorem, there exists a unique function $\bar{w}\in \h$ such that
\begin{equation}\label{EqVD}
(-\Delta+m^{2})^{s}\bar{w}-(V_{1}+\delta) \bar{w}=\phi  \quad \mbox{ in } \R^{N}.
\end{equation}
Since $\bar{w}=\mathcal{B}_{2s,m}*\phi$, for some positive kernel $\mathcal{B}_{2s, m}$ whose expression is given below, we can see that $\bar{w}\geq 0$ in $\R^{N}$ and $\bar{w}\not\equiv 0$.
Denote by $\overline{W}$ the extension of $\bar{w}$, namely $\overline{W}(x, y):=(P_{s, m}(\cdot, y)*\bar{w})(x)$. Fix $x\in \R^{N}$. Then, from Young's inequality and \eqref{Nkernel}, we can see that
\begin{align*}
\|\overline{W}\|_{L^{2}(B_{1}(x)\times [0, 1], y^{1-2s})}&\leq  \|\overline{W}\|_{L^{2}(\R^{N}\times [0, 1], y^{1-2s})}=\|P_{s, m}(\cdot, y)*\bar{w}\|_{L^{2}(\R^{N}\times [0, 1], y^{1-2s})}  \\
&\leq \left(\int_{0}^{1}  |P_{s, m}(\cdot, y)*\bar{w}|_{2}^{2}  \,y^{1-2s}\, dy\right)^{\frac{1}{2}} \\
&\leq \left(\int_{0}^{1} |P_{s, m}(\cdot, y)|_{1}^{2} |\bar{w}|_{2}^{2} \, y^{1-2s} \, dy\right)^{\frac{1}{2}} \\
&\leq  |\bar{w}|_{2} \left(\int_{0}^{1} \vartheta^{2}(my) y^{1-2s} \, dy\right)^{\frac{1}{2}}\\
&\leq c_{s, m} |\bar{w}|_{2},
\end{align*}
for some constant $c_{s, m}>0$.
Therefore, by Proposition \ref{PROPFF}-$(i)$, we deduce that  $\bar{w}\in L^{\infty}(\R^{N})$, and thus, by interpolation, $\bar{w}\in L^{q}(\R^{N})$ for any $q\in [2, \infty]$.
Hence, $(-\Delta+m^{2})^{s}\bar{w}=(V_{1}+\delta) \bar{w}+\phi\in L^{\infty}(\R^{N})$, and
applying  Corollary \ref{SilvestreLinfty} (see also Theorem \ref{Besselembedding}-$(v)$ and Proposition \ref{PROPFF}-$(iii)$) we obtain that $\bar{w}$ is H\"older continuous in $\R^{N}$. 
Using Proposition \ref{PROPFF}-$(ii)$, we have that $\bar{w}>0$ in $\R^{N}$. 
Moreover, we can prove that there exist $c, C>0$ such that
\begin{align}\label{HZ1AMPA}
0<\bar{w}(x)\leq Ce^{-c|x|} \quad \mbox{ for all } x\in \R^{N}.
\end{align} 
Assume for the moment that \eqref{HZ1AMPA} holds and we postpone the proof of it after showing how \eqref{HZ1AMPA} yields the desired decay estimate for $u_{n}(\cdot, 0)$.
We know that
\begin{align}\label{HZ2AMPA}
(-\Delta+m^{2})^{s} \bar{w}-(V_{1}+\delta)\bar{w}= 0 \quad \mbox{ in } \R^{N}\setminus \overline{B_{R_{1}}}. 
\end{align}
On the other hand, by $(V_1)$ and \eqref{TERESA}, we can see that
\begin{align}\label{HZ3AMPA}
(-\Delta+m^{2})^{s} w_{n}(\cdot, 0)-(V_{1}+\delta) w_{n}(\cdot,0) \leq 0 \quad \mbox{ in } \R^{N}\setminus \overline{B_{R_{1}}}. 
\end{align}
Set $b:=\min_{x\in \overline{B_{R_{1}}}  } \bar{w}(x)>0$ and $z_{n}:=(\ell+1)\bar{w}-bw_{n}(\cdot, 0)$, where $\ell:=\sup_{n\in \mathbb{N}} |w_{n}(\cdot, 0)|_{\infty}$. 
We note that $z_{n}\geq 0$ in $\overline{B_{R_{1}}}$ and that
$$
(-\Delta+m^{2})^{s}z_{n}-(V_{1}+\delta)z_{n}\geq 0 \quad\mbox{ in } \R^{N}\setminus \overline{B_{R_{1}}}.
$$
Since $V_{1}+\delta<m^{2s}$, we can apply Theorem \ref{Comparison} with $\Omega=\R^{N}\setminus \overline{B_{R_{1}}}$ to deduce that $z_{n}\geq 0$ in $\R^{N}$.
In the light of (\ref{HZ1AMPA}), we obtain that there exist $c, C>0$ such that
\begin{align*}
0\leq w_{n}(x)\leq Ce^{-c|x|} \quad  \mbox{ for all } x\in \R^{N}, n\in \mathbb{N},
\end{align*}
which combined with $u_{n}(x, 0)=w_{n}(x-y_{n}, 0)$, $x_{n}=q_{n}+y_{n}$ and $|q_{n}|<R$ yields
\begin{align*}
u_{n}(x, 0)=w_{n}(x-y_{n},0)\leq C' e^{-c|x-x_{n}|} \quad \mbox{ for all } x\in \R^{N}, n\in \mathbb{N}.
\end{align*}
In what follows, we focus our attention on the estimate \eqref{HZ1AMPA}.
We recall that $\bar{w}=\mathcal{B}_{2s, m}*\phi$, where 
$$
\mathcal{B}_{2s, m}(x):=(2\pi)^{-\frac{N}{2}}\mathcal{F}^{-1}([(|k|^{2}+m^{2})^{s}-(V_{1}+\delta)]^{-1})(x).
$$ 
Since $\phi$ has compact support, the exponential decay of $\bar{w}$ at infinity follows if we show the exponential decay of $\mathcal{B}_{2s, m}(x)$ for big values of $|x|$. After that, due to the fact that $\bar{w}$ is continuous in $\R^{N}$, we can deduce the exponential decay of $\bar{w}$ in the whole of $\R^{N}$. Next we prove the exponential decay of $\mathcal{B}_{2s, m}(x)$ for $|x|$ large. Then,
\begin{align}\label{RY1}
\mathcal{B}_{2s, m}(x)&=\frac{1}{(2\pi)^{N}}  \int_{\R^{N}} e^{\imath k\cdot x} \frac{1}{[(|k|^{2}+m^{2})^{s}-(V_{1}+\delta)]}\, dk \nonumber\\
&=\frac{1}{(2\pi)^{N}}  \int_{\R^{N}} e^{\imath k\cdot x}\left(\int_{0}^{\infty} e^{-t[(|k|^{2}+m^{2})^{s}-(V_{1}+\delta)]}\, dt\right)\, dk \nonumber\\
&= \int_{0}^{\infty} e^{-\gamma t}  \left(  \frac{1}{(2\pi)^{N}}   \int_{\R^{N}}  e^{\imath k\cdot x} e^{-t[(|k|^{2}+m^{2})^{s}-m^{2s}]}\, dk\right) \, dt \nonumber\\
&=\int_{0}^{\infty} e^{-\gamma t} \, p_{s, m}(x, t)\, dt,
\end{align}
where 
$$
\gamma:=m^{2s}-(V_{1}+\delta)>0,
$$
and 
$$
p_{s, m}(x, t):=e^{m^{2s}t} \int_{0}^{\infty} \frac{1}{(4\pi z)^{\frac{N}{2}}} e^{-\frac{|x|^{2}}{4z}} e^{-m^{2}z} \vartheta_{s}(t, z)\, dz
$$
is the $2s$-stable relativistic density with parameter $m$ (see pag. 4 formula (7) in \cite{ryznar}, and pag. 4875 formula (2.12) and Lemma 2.2 in \cite{BMR}), and $\vartheta_{s}(t, z)$ is the density function of the $s$-stable process whose Laplace transform is $e^{-t\lambda^{s}}$. 
Using the scaling property $p_{s, m}(x, t)=m^N p_{s, 1}(mx, m^{2s} t)$ (see pag. 4876 formula (2.15) in \cite{BMR}) and Lemma $2.2$ in \cite{GR}, we can see that for some $C>0$ (depending only on $N, s, m$)
\begin{align}\label{GRestimate}
p_{s, m}(x, t)\leq C   \left(g_{m^{2s}t}\left(\frac{mx}{\sqrt{2}}\right)+t \nu^{1}\left(\frac{mx}{\sqrt{2}} \right) \right)  \quad \mbox{ for all } x\in \R^{N}, t>0,
\end{align}
where 
$$
g_{t}(x):=\frac{1}{(4\pi t)^{\frac{N}{2}}} e^{-\frac{|x|^{2}}{4t}},
$$
and $\nu^{\mathfrak{m}}$ is the L\'evy measure of relativistic process with parameter $\mathfrak{m}>0$ given by
$$
\nu^{\mathfrak{m}}(x):=\frac{2s2^{\frac{2s-N}{2}}}{\pi^{\frac{N}{2}} \Gamma(1-s)} \left( \frac{\mathfrak{m}}{|x|} \right)^{\frac{N+2s}{2}} K_{\frac{N+2s}{2}}(\mathfrak{m}|x|)
$$
(see pag. 4877 formula (2.17) in \cite{BMR}).
Therefore, \eqref{RY1} and \eqref{GRestimate} yield
\begin{align}\label{RY2}
\mathcal{B}_{2s, m}(x)&\leq C\int_{0}^{\infty} e^{-\gamma t} g_{m^{2s}t}\left(\frac{mx}{\sqrt{2}}\right)\, dt+C\int_{0}^{\infty} e^{-\gamma t} \, t \nu^{1} \left(\frac{mx}{\sqrt{2}}\right) \,dt  \nonumber\\
&=:I_{1}(x)+I_{2}(x).
\end{align}
We start with the estimate of $I_{1}(x)$ for $|x|\geq 2$. Observing that 
\begin{align*}
\gamma t+\frac{m^{2-2s}}{8t}|x|^{2}\geq \gamma t+\frac{m^{2-2s}}{2t} \quad \mbox{ for all } |x|\geq 2, t>0,
\end{align*}
and that $ab\leq \epsilon a^{2}+\frac{1}{4\epsilon} b^{2}$ for all $a, b\geq 0$ and $\epsilon>0$ gives
\begin{align*}
\gamma t+\frac{m^{2-2s}}{8t}|x|^{2}\geq \frac{m^{1-s}}{\sqrt{2}} |x|\sqrt{\gamma} \quad \mbox{ for all } x\in \R^{N}, t>0,
\end{align*}
we deduce that for all $|x|\geq 2$ and $t>0$
\begin{align*}
\gamma t+\frac{m^{2-2s}}{8t}|x|^{2}\geq \gamma \frac{t}{2}+\frac{m^{2-2s}}{4t}+\frac{m^{1-s}}{2\sqrt{2}} |x|\sqrt{\gamma}.
\end{align*}
Thus, using the definition of $g_{t}$, we can see that for all $|x|\geq 2$
\begin{align}\label{RY3}
I_{1}(x)&\leq C\int_{0}^{\infty} \frac{e^{-\gamma\frac{t}{2}}}{t^{\frac{N}{2}}} e^{-\frac{m^{2-2s}}{4t}} e^{-\frac{m^{1-s}}{2\sqrt{2}} |x|\sqrt{\gamma}}\, dt \nonumber \\
&\leq C e^{-c|x|}  \int_{0}^{\infty} \frac{e^{-\gamma\frac{t}{2}}}{t^{\frac{N}{2}}} e^{-\frac{m^{2-2s}}{4t}} \,dt\leq C_{1} e^{-C_{2}|x|} ,
\end{align}
where we used the fact that 
$$
 \int_{0}^{\infty} \frac{e^{-\alpha t}}{t^{\upsilon}} e^{-\frac{\beta}{t}} \,dt<\infty \quad \forall \alpha, \beta, \upsilon>0.
$$
Now we estimate $I_{2}(x)$ for large values of $|x|$.
Recalling formula \eqref{Watson2} concerning the asymptotic behavior of $K_{\nu}$ at infinity, we deduce that there exists $r_{0}>0$ such that 
$$
\frac{K_{\nu}(r)}{r^{\nu}}\leq C'\frac{e^{-r}}{r^{\nu+\frac{1}{2}}}   \quad \mbox{ for all } r\geq r_{0},
$$ 
and thus
\begin{align*}
 \frac{K_{\frac{N+2s}{2}}(\frac{m}{\sqrt{2}}|x|)}{\left(\frac{m}{\sqrt{2}}|x|\right)^{\frac{N+2s}{2}}}\leq \bar{C} \frac{e^{-\bar{c}|x|}}{|x|^{\frac{N+2s+1}{2}}} \quad \mbox{ for all } |x|\geq r'_{0}:=\frac{\sqrt{2}}{m}r_{0}. 
\end{align*}
Consequently, using the definition of $\nu^{1}$, for all $|x|\geq r'_{0}$ we get
\begin{align}\label{RY4}
I_{2}(x)\leq \bar{C}\frac{e^{-\bar{c} |x|}}{|x|^{\frac{N+2s+1}{2}}} \int_{0}^{\infty} t \,e^{-\gamma t}\, dt\leq C_{3}\frac{e^{-C_{4}|x|}}{|x|^{\frac{N+2s+1}{2}}}. 
\end{align} 
Gathering \eqref{RY1}, \eqref{RY2}, \eqref{RY3}, and \eqref{RY4}, we find that for any $|x|\geq \max\{r'_{0}, 2\}$
$$
\mathcal{B}_{2s, m}(x)\leq C_{1} e^{-C_{2}|x|}+C_{3}\frac{e^{-C_{4}|x|}}{|x|^{\frac{N+2s+1}{2}}}\leq C_{5}e^{-C_{6}|x|}.
$$
Then \eqref{HZ1AMPA} holds true and this ends the proof of Theorem \ref{thm1}.
\end{proof}
\begin{remark}
When $s=\frac{1}{2}$, $p_{\frac{1}{2},m}(x, t)$ can be calculated explicitly (see \cite{BMR, LL}) and is given by
$$
p_{\frac{1}{2},m}(x, t)=2\left(\frac{m}{2\pi}\right)^{\frac{N+1}{2}} t\, e^{m t} (|x|^{2}+t^{2})^{-\frac{N+1}{4}} K_{\frac{N+1}{2}}(m\sqrt{|x|^{2}+t^{2}}).
$$
\end{remark}
\begin{remark}
By the definitions of $\mathcal{B}_{2s, m}$ and $p_{s, m}$, it follows that  $\mathcal{B}_{2s, m}$ is radial, positive, decreasing in $|x|$, and smooth on $\R^{N}\setminus \{0\}$. 
\end{remark}

\section{A multiplicity result}
This section is devoted to the multiplicity of positive solutions to \eqref{P}. 
In this case, we define $\tilde{f}$ as in Section $3$ by replacing $V_{1}$ by $V_{0}$, and we choose $\kappa>\frac{2V_{0}}{m^{2s}-V_{0}}$.

\subsection{The Nehari manifold approach}
Let us consider
\begin{equation*}
X_{\e}^{+}:= \{u\in X_{\e} : |{\rm supp}(u^{+}(\cdot, 0)) \cap \La_{\e}|>0\}. 
\end{equation*}
Let $\mathbb{S}_{\e}$ be the unit sphere of $X_{\e}$ and we define $\mathbb{S}_{\e}^{+}:= \mathbb{S}_{\e}\cap X_{\e}$.
We observe that $X_{\e}^{+}$ is open in $X_{\e}$.
By the definition of $\mathbb{S}_{\e}^{+}$ and the fact that $X_{\e}^{+}$ is open in $X_{\e}$, it follows that $\mathbb{S}_{\e}^{+}$ is a non-complete $C^{1,1}$-manifold of codimension $1$, modeled on $X_{\e}$ and contained in the open 
$X_{\e}^{+}$; see \cite{SW}. Then, $X_{\e}=T_{u} \mathbb{S}_{\e}^{+} \oplus \mathbb{R} u$ for each $u\in \mathbb{S}_{\e}^{+}$, where
\begin{equation*}
T_{u} \mathbb{S}_{\e}^{+}:= \{v \in X_{\e} : \langle u, v\rangle_{\e}=0\}.
\end{equation*} 
Since $f$ is only continuous, the next results will be fundamental to overcome the non-differentiability of $\mathcal{N}_{\e}$ and the incompleteness of $\mathbb{S}_{\e}^{+}$.
\begin{lem}\label{lem2.3MAW}
Assume that $(V'_1)$-$(V'_2)$ and $(f_1)$-$(f_4)$ hold. Then:
\begin{compactenum}[$(i)$]
\item For each $u\in X_{\e}^{+}$, let $h:\mathbb{R}_{+}\rightarrow \mathbb{R}$ be defined by $h_{u}(t):= J_{\e}(tu)$. Then, there is a unique $t_{u}>0$ such that 
\begin{align*}
&h'_{u}(t)>0 \quad \mbox{ in } (0, t_{u}),\\
&h'_{u}(t)<0 \quad \mbox{ in } (t_{u}, \infty).
\end{align*}
\item There exists $\tau>0$ independent of $u$ such that $t_{u}\geq \tau$ for any $u\in \mathbb{S}_{\e}^{+}$. Moreover, for each compact set $\mathbb{K}\subset \mathbb{S}_{\e}^{+}$ there is a positive constant $C_{\mathbb{K}}$ such that $t_{u}\leq C_{\mathbb{K}}$ for any $u\in \mathbb{K}$.
\item The map $\hat{m}_{\e}: X_{\e}^{+}\rightarrow \mathcal{N}_{\e}$ given by $\hat{m}_{\e}(u):= t_{u}u$ is continuous and $m_{\e}:= \hat{m}_{\e}|_{\mathbb{S}_{\e}^{+}}$ is a homeomorphism between $\mathbb{S}_{\e}^{+}$ and $\mathcal{N}_{\e}$. Moreover, $m_{\e}^{-1}(u)=\frac{u}{\|u\|_{\e}}$.
\item If there is a sequence $(u_{n})\subset \mathbb{S}_{\e}^{+}$ such that ${\rm dist}(u_{n}, \partial \mathbb{S}_{\e}^{+})\rightarrow 0$, then $\|m_{\e}(u_{n})\|_{\e}\rightarrow \infty$ and $J_{\e}(m_{\e}(u_{n}))\rightarrow \infty$.
\end{compactenum}

\end{lem}

\begin{proof}
$(i)$ Clearly, $h_{u}\in C^{1}(\mathbb{R}_{+}, \mathbb{R})$, and arguing as in the proof of Lemma \ref{lemma1}, it is easy to verify that $h_{u}(0)=0$, $h_{u}(t)>0$ for $t>0$ small enough and $h_{u}(t)<0$ for $t>0$ sufficiently  large. Therefore, $\max_{t\geq 0} h_{u}(t)$ is achieved at some $t_{u}>0$ verifying $h'_{u}(t_{u})=0$ and $t_{u}u\in \mathcal{N}_{\e}$. 
Now, we note that
\begin{align*}
h'_{u}(t)=0 \Longleftrightarrow \|u\|^{2}_{\e}=\int_{\R^{N}} \frac{g_{\e}(x, tu)}{t} u\, dx=:\zeta(t).
\end{align*} 
By the definition of $g$, $(g_4)$ and $u\in X^{+}_{\e}$, we see that the function $t\mapsto \zeta(t)$ is increasing in $(0, \infty)$. Hence, there exists a unique $t_{u}>0$ satisfying the desired properties.

\noindent
$(ii)$ Let $u\in \mathbb{S}_{\e}^{+}$. By $(i)$ there exists $t_{u}>0$ such that $h_{u}'(t_{u})=0$, that is
\begin{equation*}
t_{u}= \int_{\mathbb{R}^{N}} g_{\e}(x, t_{u}u) u \,dx. 
\end{equation*}
By \eqref{growthg} and Theorem \ref{Sembedding}, for all $\eta>0$ there exists $C_{\eta}>0$ such that
\begin{align*}
t_{u}\leq \eta t_{u}C_{1} + C_{\eta} C_{2} t^{\2-1}_{u}, 
\end{align*}
with $C_{1}, C_{2}>0$ independent of $\eta$.
Taking $\eta>0$ sufficiently small, we obtain that there exists $\tau>0$, independent of $u$, such that $t_{u}\geq \tau$. Now, let $\mathbb{K}\subset \mathbb{S}_{\e}^{+}$ be a compact set and we show that $t_{u}$ can be estimated from above by a constant depending on $\mathbb{K}$. Assume by contradiction that there exists a sequence $(u_{n})\subset \mathbb{K}$ such that $t_{n}:=t_{u_{n}}\rightarrow \infty$. Therefore, there exists $u\in \mathbb{K}$ such that $u_{n}\rightarrow u$ in $X_{\e}$. From $(iii)$ in Lemma \ref{lemma1}, we get 
\begin{equation}\label{cancMAW}
J_{\e}(t_{n}u_{n})\rightarrow -\infty. 
\end{equation}
Fix $v\in \mathcal{N}_{\e}$. Then, using $\langle J_{\e}'(v), v \rangle=0$, $(V'_1)$, $(g_{3})$ and \eqref{m-ineq}, we can infer
\begin{align}\label{cancanMAW}
J_{\e}(v)&= J_{\e}(v)- \frac{1}{\theta} \langle J_{\e}'(v), v \rangle \nonumber\\
&=\left(\frac{1}{2}-\frac{1}{\theta}\right) \|v\|_{\e}^{2}+\frac{1}{\theta} \int_{\mathbb{R}^{N}}  [g_{\e}(x, v(x, 0))v(x,0)-  \theta G_{\e}(x, v(x, 0))]\, dx \nonumber\\
&\geq \left(\frac{1}{2}-\frac{1}{\theta}\right)\|v\|_{\e}^{2} -\left(\frac{1}{2}-\frac{1}{\theta}\right) \frac{1}{\kappa}\int_{\Lambda^{c}_{\e}} V_{0} v^{2}(x, 0)\,dx \nonumber\\
&=\left(\frac{1}{2}-\frac{1}{\theta}\right) \|v\|^{2}_{\x}+ \left(\frac{1}{2}-\frac{1}{\theta}\right) \left[\int_{\R^{N}} (V_{\e}(x)+V_{0}) v^{2}(x, 0)\, dx-\left(1+\frac{1}{\kappa}\right) V_{0}\int_{\R^{N}} v^{2}(x, 0) \,dx  \right] \nonumber \\
& \geq \left(\frac{1}{2}-\frac{1}{\theta}\right)  \|v\|_{\x}^{2}-\left(\frac{1}{2}-\frac{1}{\theta}\right)\left(1+\frac{1}{\kappa}\right)V_{0}\int_{\R^{N}} v^{2}(x, 0) \,dx \nonumber\\
&=\left(\frac{1}{2}-\frac{1}{\theta}\right)  \|v\|_{\x}^{2}-\left(\frac{1}{2}-\frac{1}{\theta}\right)\left(1+\frac{1}{\kappa}\right)\frac{V_{0}}{m^{2s}}\int_{\R^{N}} m^{2s} v^{2}(x, 0) \,dx \nonumber\\
&\geq \left(\frac{1}{2}-\frac{1}{\theta}\right)\left(1-\left(1+\frac{1}{\kappa}\right)\frac{V_{0}}{m^{2s}}  \right)  \|v\|_{\x}^{2}>0
\end{align}
because $\kappa>\frac{V_{0}}{m^{2s}-V_{0}}$.
Taking into account that $(t_{u_{n}}u_{n})\subset \mathcal{N}_{\e}$ and $\|t_{u_{n}}u_{n}\|_{\e}=t_{u_{n}}\ri \infty$, from \eqref{cancanMAW} we deduce that \eqref{cancMAW} does not hold. 

\noindent
$(iii)$ Firstly, we note that $\hat{m}_{\e}$, $m_{\e}$ and $m_{\e}^{-1}$ are well-defined. Indeed, by $(i)$, for each $u\in X_{\e}^{+}$ there exists a unique $m_{\e}(u)\in \mathcal{N}_{\e}$. On the other hand, if $u\in \mathcal{N}_{\e}$ then $u\in X_{\e}^{+}$. Otherwise, if $u\notin X_{\e}^{+}$, we have
\begin{equation*}
|{\rm supp} (u^{+}(\cdot, 0)) \cap \Lambda_{\e}|=0, 
\end{equation*}
which together with $(g_3)$-$(ii)$ and \eqref{m-ineq} gives 
\begin{align*}
\|u\|_{\e}^{2}&= \int_{\mathbb{R}^{N}} g_{\e}(x, u(x, 0)) u(x, 0) \, dx \nonumber \\
&= \int_{\La_{\e}} g_{\e}(x, u(x, 0)) u(x, 0) \, dx + \int_{\La^{c}_{\e}} g_{\e}(x, u(x, 0))u(x, 0) \, dx \nonumber \\
&= \int_{\La^{c}_{\e}} g_{\e}(x, u^{+}(x, 0)) u^{+}(x, 0) \, dx \nonumber \\
&\leq \frac{V_{0}}{\kappa} \int_{\Lambda^{c}_{\e}} u^{2}(x, 0) \,dx.  
\end{align*}
Therefore,
\begin{align*}
\|u\|_{\x}^{2}&\leq -\int_{\R^{N}} (V_{\e}(x)+V_{0})u^{2}(x, 0) \,dx+\left(1+\frac{1}{\kappa}\right)V_{0}\int_{\R^{N}}u^{2}(x, 0) \,dx  \nonumber \\
&\leq \left(1+\frac{1}{\kappa}\right)\frac{V_{0}}{m^{2s}}\|u\|^{2}_{\x}, 
\end{align*}
and using $u\not\equiv 0$ and $\kappa>\frac{V_{0}}{m^{2s}-V_{0}}$, we find
\begin{align*}
0<\left(1-\left(1+\frac{1}{\kappa}\right)\frac{V_{0}}{m^{2s}}\right) \|u\|_{\x}^{2}\leq 0,
\end{align*}
that is a contradiction. 
Consequently, $m_{\e}^{-1}(u)= \frac{u}{\|u\|_{\e}}\in \mathbb{S}_{\e}^{+}$, $m_{\e}^{-1}$ is well-defined and continuous. \\
Let $u\in \mathbb{S}_{\e}^{+}$. Then,
\begin{align*}
m_{\e}^{-1}(m_{\e}(u))= m_{\e}^{-1}(t_{u}u)= \frac{t_{u}u}{\|t_{u}u\|_{\e}}= \frac{u}{\|u\|_{\e}}=u
\end{align*}
from which $m_{\e}$ is a bijection. Now, our aim is to prove that $\hat{m}_{\e}$ is a continuous function. 
Let $(u_{n})\subset X_{\e}^{+}$ and $u\in X_{\e}^{+}$ such that $u_{n}\rightarrow u$ in $X_{\e}^{+}$. 
Hence,
\begin{equation*}
\frac{u_{n}}{\|u_{n}\|_{\e}}\rightarrow \frac{u}{\|u\|_{\e}} \quad \mbox{ in } X_{\e}.
\end{equation*} 
Set $v_{n}:= \frac{u_{n}}{\|u_{n}\|_{\e}}$ and $t_{n}:=t_{v_{n}}$. By $(ii)$ there exists $t_{0}>0$ such that $t_{n}\rightarrow t_{0}$. Since $t_{n}v_{n}\in \mathcal{N}_{\e}$ and $\|v_{n}\|_{\e}=1$, we have 
\begin{equation*}
t^{2}_{n}= \int_{\mathbb{R}^{N}} g_{\e}(x, t_{n}v_{n}(x,0)) t_{n}v_{n}(x, 0) \, dx. 
\end{equation*}
Passing to the limit as $n\rightarrow \infty$, we obtain
\begin{equation*}
t^{2}_{0}= \int_{\mathbb{R}^{N}} g_{\e}(x, t_{0}v(x, 0)) t_{0} v(x, 0)\, dx, 
\end{equation*}
where $v:= \frac{u}{\|u\|_{\e}}$, and thus $t_{0}v\in \mathcal{N}_{\e}$. By $(i)$ we deduce that $t_{v}= t_{0}$, and this shows that
\begin{equation*}
\hat{m}_{\e}(u_{n})= \hat{m}_{\e}\left(\frac{u_{n}}{\|u_{n}\|_{\e}}\right)\rightarrow \hat{m}_{\e}\left(\frac{u}{\|u\|_{\e}}\right)=\hat{m}_{\e}(u) \quad \mbox{ in } X_{\e}.  
\end{equation*}
Therefore, $\hat{m}_{\e}$ and $m_{\e}$ are continuous functions. 

\noindent
$(iv)$ Let $(u_{n})\subset \mathbb{S}_{\e}^{+}$ be such that ${\rm dist}(u_{n}, \partial \mathbb{S}_{\e}^{+})\rightarrow 0$. 
Then, by Theorem \ref{Sembedding} and $(V'_{1})$-$(V'_{2})$, for each $q\in [2, 2^{*}_{s}]$ there exists $C_{q}>0$ such that
\begin{align*}
|u^{+}_{n}(\cdot, 0)|^{q}_{L^{q}(\La_{\e})} &\leq \inf_{v\in \partial \mathbb{S}_{\e}^{+}} |u_{n}(\cdot, 0)- v(\cdot, 0)|^{q}_{L^{q}(\La_{\e})}\\
&\leq C_{q} \inf_{v\in \partial \mathbb{S}_{\e}^{+}} \|u_{n}- v\|^{q}_{\e} \\
&\leq C_{q} \, {\rm dist}(u_{n}, \partial \mathbb{S}_{\e}^{+})^{q}, \quad \mbox{ for all } n\in \mathbb{N}.
\end{align*}
Thus, by $(g_{1})$, $(g_{2})$, and $(g_{3})$-$(ii)$, we can infer that, for all $t>0$,
\begin{align}\label{ter2MAW}
\int_{\mathbb{R}^{N}} G_{\e}(x, tu_{n}(x, 0))\, dx &= \int_{\La^{c}_{\e}} G_{\e}(x, tu_{n}(x, 0))\, dx + \int_{\La_{\e}} G_{\e}(x, tu_{n}(x, 0))\, dx  \nonumber\\
&\leq \frac{t^{2}}{\kappa} V_{0} \int_{\La^{c}_{\e}} u_{n}^{2}(x, 0) \,dx+ \int_{\La_{\e}} F(tu_{n}(x, 0))\, dx \nonumber\\
&\leq \frac{t^{2}}{\kappa} V_{0} \int_{\R^{N}} u_{n}^{2}(x, 0) \,dx + C_{1}t^{2} \int_{\La_{\e}} (u_{n}^{+}(x, 0))^{2} \,dx+ C_{2} t^{\2} \int_{\La_{\e}} (u^{+}_{n}(x, 0))^{\2} \,dx \nonumber\\
&\leq \frac{t^{2}}{\kappa} V_{0} \int_{\R^{N}} u_{n}^{2}(x, 0) \,dx+ C_{1}' t^{2} {\rm dist}(u_{n}, \partial \mathbb{S}_{\e}^{+})^{2} + C_{2}' t^{\2} {\rm dist}(u_{n}, \partial \mathbb{S}_{\e}^{+})^{\2}.
\end{align}
Bearing in mind the definition of $m_{\e}(u_{n})$, and using \eqref{m-ineq}, \eqref{equivalent}, \eqref{ter2MAW}, ${\rm dist}(u_{n}, \partial \mathbb{S}_{\e}^{+})\rightarrow 0$ and $\|u_{n}\|_{\e}=1$, we have
\begin{align*}
\liminf_{n\rightarrow \infty} J_{\e}(m_{\e}(u_{n}))&\geq \liminf_{n\rightarrow \infty} J_{\e}(tu_{n})\\
&\geq \liminf_{n\rightarrow \infty} \Bigl\{t^{2} \left[ \frac{1}{2}\|u_{n}\|_{\e}^{2}- \frac{V_{0}}{\kappa} \int_{\R^{N}} u_{n}^{2}(x, 0) \,dx \right]- C_{1}' t^{2} {\rm dist}(u_{n}, \partial \mathbb{S}_{\e}^{+})^{2} - C_{2}' t^{\2} {\rm dist}(u_{n}, \partial \mathbb{S}_{\e}^{+})^{\2} \Bigr\}\\
&=\liminf_{n\rightarrow \infty} t^{2} \left[ \frac{1}{2}\|u_{n}\|_{\e}^{2}-\frac{V_{0}}{\kappa}\int_{\R^{N}} u_{n}^{2}(x, 0) \,dx \right]\\
&\geq \liminf_{n\rightarrow \infty} t^{2} \left[ \frac{1}{2}\|u_{n}\|_{\e}^{2}-\frac{V_{0}}{\kappa m^{2s}} \|u_{n}\|^{2}_{\x} \right]\\
&\geq \liminf_{n\rightarrow \infty} t^{2}  \left(\frac{1}{2}-\frac{V_{0}}{\kappa(m^{2s}-V_{0})}\right)\|u_{n}\|_{\e}^{2}\\
&=t^{2} \left(\frac{1}{2}-\frac{V_{0}}{\kappa(m^{2s}-V_{0})} \right) \quad \mbox{ for all } t>0.
\end{align*}
Recalling that $\kappa>\frac{2V_{0}}{m^{2s}-V_{0}}$ and sending $t\ri \infty$, we get
\begin{equation*}
\lim_{n\rightarrow \infty} J_{\e}(m_{\e}(u_{n}))= \infty. 
\end{equation*}
In particular, by the definition of $J_{\e}$ and $(g_{3})$, we obtain
\begin{equation*}
\liminf_{n\ri \infty} \frac{1}{2}\|m_{\e}(u_{n})\|^{2}_{\e}\geq \liminf_{n\ri \infty} J_{\e}(m_{\e}(u_{n}))=\infty
\end{equation*}
which gives $\|m_{\e}(u_{n})\|_{\e}\rightarrow \infty$ as $n\rightarrow \infty$. The proof of the lemma is now completed.
\end{proof}

\noindent
Let us define the maps 
\begin{equation*}
\hat{\psi}_{\e}: X_{\e}^{+} \rightarrow \mathbb{R} \quad \mbox{ and } \quad \psi_{\e}: \mathbb{S}_{\e}^{+}\rightarrow \mathbb{R}, 
\end{equation*}
by $\hat{\psi}_{\e}(u):= J_{\e}(\hat{m}_{\e}(u))$ and $\psi_{\e}:=\hat{\psi}_{\e}|_{\mathbb{S}_{\e}^{+}}$. \\
From Lemma \ref{lem2.3MAW} and arguing as in the proofs of Proposition $9$ and Corollary $10$ in \cite{SW}, we may obtain the following result.
\begin{prop}\label{prop2.1MAW}
Assume that $(V'_{1})$-$(V'_{2})$ and $(f_{1})$-$(f_{4})$ hold. Then, 
\begin{compactenum}[$(a)$]
\item $\hat{\psi}_{\e} \in C^{1}(X_{\e}^{+}, \mathbb{R})$ and 
\begin{equation*}
\langle \hat{\psi}_{\e}'(u), v\rangle = \frac{\|\hat{m}_{\e}(u)\|_{\e}}{\|u\|_{\e}} \langle J_{\e}'(\hat{m}_{\e}(u)), v\rangle, 
\end{equation*}
for every $u\in X_{\e}^{+}$ and $v\in X_{\e}$. 
\item $\psi_{\e} \in C^{1}(\mathbb{S}_{\e}^{+}, \mathbb{R})$ and 
\begin{equation*}
\langle \psi_{\e}'(u), v \rangle = \|m_{\e}(u)\|_{\e} \langle J_{\e}'(m_{\e}(u)), v\rangle, 
\end{equation*}
for every $v\in T_{u}\mathbb{S}_{\e}^{+}$.
\item If $(u_{n})$ is a Palais-Smale sequence for $\psi_{\e}$, then $(m_{\e}(u_{n}))$ is a Palais-Smale sequence for $J_{\e}$. If $(u_{n})\subset \mathcal{N}_{\e}$ is a bounded Palais-Smale sequence for $J_{\e}$, then $(m_{\e}^{-1}(u_{n}))$ is a Palais-Smale sequence for  $\psi_{\e}$. 
\item $u$ is a critical point of $\psi_{\e}$ if and only if $m_{\e}(u)$ is a nontrivial critical point for $J_{\e}$. Moreover, the corresponding critical values coincide and 
\begin{equation*}
\inf_{u\in \mathbb{S}_{\e}^{+}} \psi_{\e}(u)= \inf_{u\in \mathcal{N}_{\e}} J_{\e}(u).  
\end{equation*}
\end{compactenum}
\end{prop}

\begin{remark}\label{rem3MAW}
As in \cite{SW}, we have the following variational characterization of the infimum of $J_{\e}$ over $\mathcal{N}_{\e}$:
\begin{align*}
c_{\e}&=\inf_{u\in \mathcal{N}_{\e}} J_{\e}(u)=\inf_{u\in X_{\e}^{+}} \max_{t>0} J_{\e}(tu)=\inf_{u\in \mathbb{S}_{\e}^{+}} \max_{t>0} J_{\e}(tu).
\end{align*}
\end{remark}

\begin{cor}\label{cor2.1MAW}
The functional $\psi_{\e}$ satisfies the $(PS)_{c}$ condition on $\mathbb{S}_{\e}^{+}$ at any level $c\in \R$. 
\end{cor}

\begin{proof}
Let $(u_{n})$ be a Palais-Smale sequence for $\psi_{\e}$ at the level $c$. Then
\begin{equation*}
\psi_{\e}(u_{n})\rightarrow c \quad \mbox{ and } \quad \|\psi_{\e}'(u_{n})\|_{*}\rightarrow 0,
\end{equation*}
as $n\ri \infty$, where $\|\cdot\|_{*}$ is the norm in the dual space $(T_{u_{n}} \mathbb{S}_{\e}^{+})^{*}$.
It follows from Proposition \ref{prop2.1MAW}-$(c)$ that $(m_{\e}(u_{n}))$ is a Palais-Smale sequence for $J_{\e}$ in $X_{\e}$ at the level $c$. Then, using Lemma \ref{lemma2}, 
there exists $u\in \mathbb{S}_{\e}^{+}$ such that, up to a subsequence, 
\begin{equation*}
m_{\e}(u_{n})\rightarrow m_{\e}(u) \quad \mbox{ in } X_{\e}. 
\end{equation*}
Applying Lemma \ref{lem2.3MAW}-$(iii)$ we can infer that $u_{n}\rightarrow u$ in $\mathbb{S}_{\e}^{+}$.
\end{proof}

It is clear that the previous results can be easily extended for the autonomous case.  More precisely, for $\mu>-m^{2s}$, 
we define
\begin{equation*}
Y_{\mu}^{+}:= \{u\in Y_{\mu} : |{\rm supp}(u^{+}(\cdot, 0))|>0\}. 
\end{equation*}
Let $\mathbb{S}_{\mu}$ be the unit sphere of $Y_{\mu}$ and we put $\mathbb{S}_{\mu}^{+}:= \mathbb{S}_{\e}\cap Y_{\mu}^{+}$.
We note that $\mathbb{S}_{\mu}^{+}$ is a non-complete $C^{1,1}$-manifold of codimension $1$, modeled on $Y_{\mu}$ and contained in the open $Y_{\mu}^{+}$. Then, $Y_{\mu}=T_{u} \mathbb{S}_{\mu}^{+} \oplus \mathbb{R} u$ for each $u\in \mathbb{S}_{\mu}^{+}$, where
\begin{equation*}
T_{u} \mathbb{S}_{\mu}^{+}:= \{v \in Y_{\mu} : \langle u, v\rangle_{Y_{\mu}}=0\}.
\end{equation*} 
Arguing as in the proof of Lemma \ref{lem2.3MAW} we have the following result.
\begin{lem}\label{lem2.3AMAW}
Assume that $(f_1)$-$(f_4)$ hold. Then, 
\begin{compactenum}[$(i)$]
\item For each $u\in Y_{\mu}^{+}$, let $h:\mathbb{R}_{+}\rightarrow \mathbb{R}$ be defined by $h_{u}(t):= L_{\mu}(tu)$. Then, there is a unique $t_{u}>0$ such that 
\begin{align*}
&h'_{u}(t)>0 \quad \mbox{ in } (0, t_{u}),\\
&h'_{u}(t)<0 \quad \mbox{ in } (t_{u}, \infty).
\end{align*}
\item There exists $\tau>0$ independent of $u$ such that $t_{u}\geq \tau$ for any $u\in \mathbb{S}_{\mu}^{+}$. Moreover, for each compact set $\mathbb{K}\subset \mathbb{S}_{\mu}^{+}$ there is a positive constant $C_{\mathbb{K}}$ such that $t_{u}\leq C_{\mathbb{K}}$ for any $u\in \mathbb{K}$.
\item The map $\hat{m}_{\mu}: Y_{\mu}^{+}\rightarrow \mathcal{M}_{\mu}$ given by $\hat{m}_{\mu}(u):= t_{u}u$ is continuous and $m_{\mu}:= \hat{m}_{\mu}|_{\mathbb{S}_{\mu}^{+}}$ is a homeomorphism between $\mathbb{S}_{\mu}^{+}$ and $\mathcal{M}_{\mu}$. Moreover, $m_{\mu}^{-1}(u)=\frac{u}{\|u\|_{Y_{\mu}}}$.
\item If there is a sequence $(u_{n})\subset \mathbb{S}_{\mu}^{+}$ such that ${\rm dist}(u_{n}, \partial \mathbb{S}_{\mu}^{+})\rightarrow 0$ then $\|m_{\mu}(u_{n})\|_{Y_{\mu}}\rightarrow \infty$ and $L_{\mu}(m_{\mu}(u_{n}))\rightarrow \infty$.
\end{compactenum}
\end{lem}

Let us define the maps 
\begin{equation*}
\hat{\psi}_{\mu}: Y_{\mu}^{+} \rightarrow \mathbb{R} \quad \mbox{ and } \quad \psi_{\mu}: \mathbb{S}_{\mu}^{+}\rightarrow \mathbb{R}, 
\end{equation*}
by $\hat{\psi}_{\mu}(u):= L_{\mu}(\hat{m}_{\mu}(u))$ and $\psi_{\mu}:=\hat{\psi}_{\mu}|_{\mathbb{S}_{\mu}^{+}}$. \\
The next result is a direct consequence of Lemma \ref{lem2.3AMAW} and some arguments found in Proposition $9$ and Corollary $10$ in \cite{SW}. 
\begin{prop}\label{prop2.1AMAW}
Assume that  $(f_{1})$-$(f_{4})$ hold. Then, 
\begin{compactenum}[$(a)$]
\item $\hat{\psi}_{\mu} \in C^{1}(Y_{\mu}^{+}, \mathbb{R})$ and 
\begin{equation*}
\langle \hat{\psi}_{\mu}'(u), v\rangle = \frac{\|\hat{m}_{\mu}(u)\|_{Y_{\mu}}}{\|u\|_{Y_{\mu}}} \langle L_{\mu}'(\hat{m}_{\mu}(u)), v\rangle, 
\end{equation*}
for every $u\in Y_{\mu}^{+}$ and $v\in Y_{\mu}$. 
\item $\psi_{\mu} \in C^{1}(\mathbb{S}_{\mu}^{+}, \mathbb{R})$ and 
\begin{equation*}
\langle \psi_{\mu}'(u), v \rangle = \|m_{\mu}(u)\|_{Y_{\mu}} \langle L_{\mu}'(m_{\mu}(u)), v\rangle, 
\end{equation*}
for every $v\in T_{u}\mathbb{S}_{\mu}^{+}$.
\item If $(u_{n})$ is a Palais-Smale sequence for $\psi_{\mu}$, then $(m_{\mu}(u_{n}))$ is a Palais-Smale sequence for $L_{\mu}$. If $(u_{n})\subset \mathcal{M}_{\mu}$ is a bounded Palais-Smale sequence for $L_{\mu}$, then $(m_{\mu}^{-1}(u_{n}))$ is a Palais-Smale sequence for  $\psi_{\mu}$. 
\item $u$ is a critical point of $\psi_{\mu}$ if and only if $m_{\mu}(u)$ is a nontrivial critical point for $L_{\mu}$. Moreover, the corresponding critical values coincide and 
\begin{equation*}
\inf_{u\in \mathbb{S}_{\mu}^{+}} \psi_{\mu}(u)= \inf_{u\in \mathcal{M}_{\mu}} L_{\mu}(u).  
\end{equation*}
\end{compactenum}
\end{prop}

\begin{remark}\label{rem3AMAW}
As in \cite{SW}, we have the following variational characterization of the infimum of $L_{\mu}$ over $\mathcal{M}_{\mu}$:
\begin{align*}
d_{\mu}&=\inf_{u\in \mathcal{M}_{\mu}} L_{\mu}(u)=\inf_{u\in Y_{\mu}^{+}} \max_{t>0} L_{\mu}(tu)=\inf_{u\in \mathbb{S}_{\mu}^{+}} \max_{t>0} L_{\mu}(tu).
\end{align*}
\end{remark}

The next lemma is a compactness result for the autonomous problem which will be used later.
\begin{lem}\label{lem3.3MAW}
Let $(u_{n})\subset \mathcal{M}_{\mu}$ be a sequence such that $L_{\mu}(u_{n})\rightarrow d_{\mu}$. Then  
$(u_{n})$ has a convergent subsequence in $Y_{\mu}$.
\end{lem}
\begin{proof}
Since $(u_{n})\subset \mathcal{M}_{\mu}$ and $L_{\mu}(u_{n})\rightarrow d_{\mu}$, we can apply Lemma \ref{lem2.3AMAW}-$(iii)$, Proposition \ref{prop2.1AMAW}-$(d)$ and the definition of $d_{\mu}$ to infer that
$$
v_{n}:=m^{-1}_{\mu}(u_{n})=\frac{u_{n}}{\|u_{n}\|_{Y_{\mu}}}\in \mathbb{S}_{\mu}^{+}, \, \mbox{ for all } n\in \mathbb{N},
$$
and
$$
\psi_{\mu}(v_{n})=L_{\mu}(u_{n})\rightarrow d_{\mu}=\inf_{v\in \mathbb{S}_{\mu}^{+}}\psi_{\mu}(v).
$$
Let us introduce the following map $\mathcal{F}: \overline{\mathbb{S}}_{\mu}^{+}\rightarrow \mathbb{R}\cup \{\infty\}$ defined by setting
$$
\mathcal{F}(u):=
\begin{cases}
\psi_{\mu}(u)& \text{ if $u\in \mathbb{S}_{\mu}^{+}$}, \\
\infty   & \text{ if $u\in \partial \mathbb{S}_{\mu}^{+}$}.
\end{cases}
$$ 
We note that 
\begin{itemize}
\item $(\overline{\mathbb{S}}_{\mu}^{+}, \delta_{\mu})$, where $\delta_{\mu}(u, v):=\|u-v\|_{Y_{\mu}}$, is a complete metric space;
\item $\mathcal{F}\in C(\overline{\mathbb{S}}_{\mu}^{+}, \mathbb{R}\cup \{\infty\})$, by Lemma \ref{lem2.3AMAW}-$(iv)$;
\item $\mathcal{F}$ is bounded below, by Proposition \ref{prop2.1AMAW}-$(d)$.
\end{itemize}
Hence, by applying Ekeland's variational principle, we can find a Palais-Smale sequence $(\hat{v}_{n})\subset \mathbb{S}_{\mu}^{+}$ for $\psi_{\mu}$ at the level $d_{\mu}$ with $\|\hat{v}_{n}-v_{n}\|_{Y_{\mu}}=o_{n}(1)$.
Now the remainder of the proof follows from Proposition \ref{prop2.1AMAW}, Theorem \ref{EGS} and arguing as in the proof of Corollary \ref{cor2.1MAW}. 
\end{proof}

\subsection{Technical results}
In this subsection we make use of the Ljusternik-Schnirelman category theory to obtain multiple solutions to \eqref{P}. 
In particular, we relate the number of positive solutions of \eqref{MEP} to the topology of the set $M$.
For this reason, we take $\delta>0$ such that
$$
M_{\delta}=\{x\in \R^{N}: {\rm dist}(x, M)\leq \delta\}\subset \Lambda,
$$
and consider a smooth nonicreasing function $\eta:[0, \infty)\rightarrow \R$ such that $\eta(t)=1$ if $0\leq t\leq \frac{\delta}{2}$, $\eta(t)=0$ if $t\geq \delta$, $0\leq \eta\leq 1$ and $|\eta'(t)|\leq c$ for some $c>0$.\\
For any $z\in M$, we introduce 
$$
\Psi_{\e, z}(x, y):=\eta(|(\e x-z, y)|) w\left(\frac{\e x-z}{\e}, y\right),
$$
where $w\in Y_{V(0)}$ is a positive ground state solution to the autonomous problem \eqref{AEP} with $\mu=V(0)=-V_{0}$, whose existence is guaranteed by Theorem \ref{EGS}. Let $t_{\e}>0$ be the unique number such that 
$$
J_{\e}(t_{\e} \Psi_{\e, z}):=\max_{t\geq 0} J_{\e}(t \Psi_{\e, z}). 
$$
Finally, we consider $\Phi_{\e}: M\rightarrow \N_{\e}$ defined by setting
$$
\Phi_{\e}(z):= t_{\e} \Psi_{\e, z}.
$$
\begin{lem}\label{lem3.4}
The functional $\Phi_{\e}$ satisfies the following limit
\begin{equation*}
\lim_{\e\rightarrow 0} J_{\e}(\Phi_{\e}(z))=d_{V(0)}, \quad \mbox{ uniformly in } z\in M.
\end{equation*}
\end{lem}
\begin{proof}
Assume by contradiction that there exist $\delta_{0}>0$, $(z_{n})\subset M$ and $\e_{n}\rightarrow 0$ such that 
\begin{equation}\label{puac}
|J_{\e_{n}}(\Phi_{\e_{n}}(z_{n}))-d_{V(0)}|\geq \delta_{0}.
\end{equation}
Applying the dominated convergence theorem, we know that 
\begin{align}\begin{split}\label{nio3}
&\| \Psi_{\e_{n}, z_{n}} \|^{2}_{\e_{n}}\rightarrow \|w\|^{2}_{Y_{V(0)}}\in (0, \infty),
\end{split}\end{align}
and
\begin{align}\label{nio3F}
\int_{\R^{N}} F(\Psi_{\e_{n}, z_{n}}(x, 0))\, dx\ri \int_{\R^{N}} F(w(x, 0))\, dx.
\end{align}
Note that, for all $n\in \mathbb{N}$ and for all $x'\in B_{\frac{\delta}{\e_{n}}}$, we have $\e_{n}x'\in B_{\delta}$ and thus $\e_{n}x'+z_{n}\in B_{\delta}(z_{n})\subset M_{\delta}\subset \Lambda$.
By using the change of variable $x':=\frac{\e_{n}x-z_{n}}{\e_{n}}$, the definition of $\eta$ and the fact that $G=F$ in $\Lambda\times \R$, we can write
\begin{align}\label{ARKA1}
J_{\e_{n}}(\Phi_{\e_{n}}(z_{n}))&=\frac{t^{2}_{\e_{n}}}{2}\left(\iint_{\R^{N+1}_{+}} y^{1-2s} (|\nabla \Psi_{\e_{n}, z_{n}}|^{2}+m^{2}  \Psi^{2}_{\e_{n}, z_{n}})\, dxdy+\int_{\R^{N}} V_{\e_{n}}(x) \Psi^{2}_{\e_{n}, z_{n}}(x, 0)\, dx\right) \nonumber\\
&\quad-\int_{\R^{N}} G(\Psi_{\e_{n}, z_{n}}(x, 0))\, dx \nonumber\\
&=\frac{t^{2}_{\e_{n}}}{2}\left(\iint_{\R^{N+1}_{+}} y^{1-2s} (|\nabla (\eta(|(\e_{n}x',y)|)w(x', y))|^{2}+m^{2} (\eta(|(\e_{n}x',y)|)w(x', y))^{2}\, dx'dy \right)  \nonumber\\
&\quad+\frac{t^{2}_{\e_{n}}}{2} \int_{\R^{N}} V(\e_{n}x'+z_{n}) (\eta(|(\e_{n}x',0)|)w(x', 0))^{2}\, dx'-\int_{\R^{N}} F(t_{\e_{n}}\eta(|(\e_{n}x',0)|)w(x', 0))\, dx'. 
\end{align}
On the other hand, since $\langle J'_{\e_{n}}(\Phi_{\e_{n}}(z_{n})),\Phi_{\e_{n}}(z_{n})\rangle=0$ and $g=f$ in $\Lambda\times \R$, we have that
\begin{align}\label{ARKA}
&t^{2}_{\e_{n}}\left(\iint_{\R^{N+1}_{+}} y^{1-2s} (|\nabla (\eta(|(\e_{n}x',y)|)w(x', y))|^{2}+m^{2} (\eta(|(\e_{n}x',y)|)w(x', y))^{2}\, dx'dy \right) \nonumber\\
&+t^{2}_{\e_{n}}\int_{\R^{N}} V(\e_{n}x'+z_{n}) (\eta(|(\e_{n}x',0)|)w(x', 0))^{2}\, dx' \nonumber\\
&=\int_{\R^{N}} f(t_{\e_{n}}\eta(|(\e_{n}x',0)|)w(x', 0)) t_{\e_{n}}\eta(|(\e_{n}x',0)|)w(x', 0)\, dx'.
\end{align}
Let us show that $t_{\e_{n}}\ri 1$ as $n\ri \infty$.
If by contradiction $t_{\e_{n}}\ri \infty$, observing that $\eta(|x|)=1$ for $x\in B_{\frac{\delta}{2}}$ and that $B_{\frac{\delta}{2}}\subset B_{\frac{\delta}{2\e_{n}}}$ for all $n$ large enough, it follows from $(f_4)$ that
\begin{align}\label{nioo}
\|\Psi_{\e_{n}, z_{n}}\|_{\e_{n}}^{2} &\geq \int_{B_{\frac{\delta}{2}}} \frac{f(t_{\e_{n}}w(x', 0))}{t_{\e_{n}}w(x', 0)} w^{2}(x', 0)\, dx' \nonumber \\
&\geq  \frac{f(t_{\e_{n}}w(\hat{x}, 0))}{t_{\e_{n}}w(\hat{x}, 0)} \int_{B_{\frac{\delta}{2}}} w^{2}(x', 0)\,dx', 
\end{align}
where $w(\hat{x}, 0):=\min_{x\in \overline{B_{\frac{\delta}{2}}}} w(x, 0)>0$.
Combining  \eqref{nio3} with \eqref{nioo} and using $(f_3)$, we obtain a contradiction.

Hence, $(t_{\e_{n}})$ is bounded in $\R$ and, up to subsequence, we may assume that $t_{\e_{n}}\rightarrow t_{0}$ for some $t_{0}\geq 0$.  
In particular, $t_{0}>0$. In fact, if $t_{0}=0$, we see that \eqref{uNr} and \eqref{ARKA} imply that
$$
r\leq \int_{\R^{N}} f(t_{\e_{n}}\eta(|(\e_{n}x',0)|)w(x', 0)) t_{\e_{n}}\eta(|(\e_{n}x',0)|)w(x', 0)\, dx'.
$$
Using $(f_1)$, $(f_2)$, \eqref{nio3} and the above inequality, we obtain an absurd. Therefore, $t_{0}>0$.
By \eqref{nio3} and \eqref{nio3F}, we can pass to the limit in \eqref{ARKA} as $n\ri \infty$ to see that
\begin{align*}
\|w\|^{2}_{V(0)}=\int_{\R^{N}} \frac{f(t_{0} w(x, 0))}{t_{0}w(x, 0)} w^{2}(x, 0) \, dx.
\end{align*}
Bearing in mind that $w\in \M_{V(0)}$ and using $(f_4)$, we infer that $t_{0}=1$.
Letting $n\rightarrow \infty$ in \eqref{ARKA1} and using $t_{\e_{n}}\rightarrow 1$, we conclude that
$$
\lim_{n\rightarrow \infty} J_{\e_{n}}(\Phi_{\e_{n}}(z_{n}))=L_{V(0)}(w)=d_{V(0)},
$$
which is in contrast with \eqref{puac}.
\end{proof}

\noindent
Let us fix $\rho=\rho(\delta)>0$ satisfying $M_{\delta}\subset B_{\rho}$, and we consider $\varUpsilon: \R^{N}\rightarrow \R^{N}$ given by
\begin{equation*}
\varUpsilon(x):=
\left\{
\begin{array}{ll}
x &\mbox{ if } |x|<\rho, \\
\frac{\rho x}{|x|} &\mbox{ if } |x|\geq \rho.
\end{array}
\right.
\end{equation*}
Then we define the barycenter map $\beta_{\e}: \N_{\e}\rightarrow \R^{N}$ as follows
\begin{align*}
\beta_{\e}(u):=\frac{\displaystyle{\int_{\R^{N}} \varUpsilon(\e x)u^{2}(x, 0) \,dx}}{\displaystyle{\int_{\R^{N}} u^{2}(x, 0) \,dx}}.
\end{align*}

\noindent
\begin{lem}\label{lem3.5N}
The functional $\Phi_{\e}$ satisfies the following limit
\begin{equation}\label{3.3}
\lim_{\e \rightarrow 0} \, \beta_{\e}(\Phi_{\e}(z))=z, \quad \mbox{ uniformly in } z\in M.
\end{equation}
\end{lem}
\begin{proof}
Suppose by contradiction that there exist $\delta_{0}>0$, $(z_{n})\subset M$ and $\e_{n}\rightarrow 0$ such that
\begin{equation}\label{4.4}
|\beta_{\e_{n}}(\Phi_{\e_{n}}(z_{n}))-z_{n}|\geq \delta_{0}.
\end{equation}
Using the definitions of $\Phi_{\e_{n}}(z_{n})$, $\beta_{\e_{n}}$, $\eta$ and the change of variable $x'= \frac{\e_{n} x-z_{n}}{\e_{n}}$, we see that
$$
\beta_{\e_{n}}(\Phi_{\e_{n}}(z_{n}))=z_{n}+\frac{\int_{\mathbb{R}^{N}}[\varUpsilon(\e_{n}x'+z_{n})-z_{n}] (\eta(|(\e_{n}x', 0)|) \omega(x', 0))^{2} \, dx'}{\int_{\mathbb{R}^{N}} (\eta(|(\e_{n}x', 0)|) \omega(x', 0))^{2}\, dx'}.
$$
Taking into account $(z_{n})\subset M\subset B_{\rho}$ and the dominated convergence theorem, we infer that
$$
|\beta_{\e_{n}}(\Phi_{\e_{n}}(z_{n}))-z_{n}|=o_{n}(1)
$$
which contradicts (\ref{4.4}).
\end{proof}

\noindent
The next compactness result will play a fundamental role to prove that the solutions of \eqref{MEP} are also solution to \eqref{EP}.
\begin{prop}\label{prop3.3}
Let $\e_{n}\rightarrow 0$ and $(u_{n})\subset \mathcal{N}_{\e_{n}}$ be such that $J_{\e_{n}}(u_{n})\rightarrow d_{V(0)}$. Then there exists $(\tilde{z}_{n})\subset \R^{N}$ such that $v_{n}(x, y):=u_{n}(x+\tilde{z}_{n}, y)$ has a convergent subsequence in $Y_{V(0)}$. Moreover, up to a subsequence, $z_{n}:=\e_{n} \tilde{z}_{n}\rightarrow z_{0}$ for some $z_{0}\in M$.
\end{prop}
\begin{proof}
Since $\langle J'_{\e_{n}}(u_{n}), u_{n} \rangle=0$ and $J_{\e_{n}}(u_{n})\rightarrow d_{V(0)}$, it is easy to see that $\|u_{n}\|_{\e_{n}}\leq C$ for all $n\in \mathbb{N}$. 
Let us observe that $\|u_{n}\|_{\e_{n}}\nrightarrow 0$ since $d_{V(0)}>0$. 
Therefore, arguing as in the proof of Lemma \ref{Lions2}, 
we can find a sequence $(\tilde{z}_{n})\subset \mathbb{R}^{N}$ and constants $R, \beta>0$ such that
\begin{equation*}
\liminf_{n\rightarrow \infty}\int_{B_{R}(\tilde{z}_{n})} u^{2}_{n}(x, 0) \,dx\geq \beta.
\end{equation*}
Set $v_{n}(x, y):=u_{n}(x+ \tilde{z}_{n}, y)$. Then, $(v_{n})$ is bounded in $Y_{V(0)}$, and we may assume that 
\begin{equation*}
v_{n}\rightharpoonup v  \quad \mbox{ in } Y_{V(0)},  
\end{equation*}
for some $v\not\equiv 0$.
Let $t_{n}>0$ be such that $\tilde{v}_{n}:=t_{n}v_{n} \in \mathcal{M}_{V(0)}$, and set $z_{n}:=\e_{n}\tilde{z}_{n}$.  
Then, using $(g_2)$ and $u_{n}\in \mathcal{N}_{\e_{n}}$, we see that
\begin{align*}
d_{V(0)}\leq L_{V(0)}(\tilde{v}_{n})\leq J_{\e_{n}}(t_{n}u_{n}) \leq J_{\e_{n}}(u_{n})= d_{V(0)}+ o_{n}(1),
\end{align*}
which gives 
\begin{align}\label{3.21CCA}
L_{V(0)}(\tilde{v}_{n})\rightarrow d_{V(0)} \,\mbox{ and } \,(\tilde{v}_{n})\subset \mathcal{M}_{V(0)}. 
\end{align}
In particular, \eqref{3.21CCA} yields that $(\tilde{v}_{n})$ is bounded in $Y_{V(0)}$, so we may assume that $\tilde{v}_{n}\rightharpoonup \tilde{v}$ in $Y_{V(0)}$. 
Since $v_{n}\nrightarrow 0$ in $Y_{V(0)}$ and $(\tilde{v}_{n})$ is bounded in $Y_{V(0)}$, we deduce that 
$(t_{n})$ is bounded in $\R$ and, up to a subsequence, we may assume that $t_{n}\rightarrow t_{0}\geq 0$. If $t_{0}=0$, from the boundedness of $(v_{n})$, we get $\|\tilde{v}_{n}\|_{Y_{V(0)}}= t_{n}\|v_{n}\|_{Y_{V(0)}} \rightarrow 0$, which yields $L_{V(0)}(\tilde{v}_{n})\rightarrow 0$ in contrast with the fact $d_{V(0)}>0$. Then, $t_{0}>0$. From the uniqueness of the weak limit we have $\tilde{v}=t_{0} v$ and $\tilde{v}\not\equiv 0$. Using Lemma \ref{lem3.3MAW} we deduce that 
\begin{align}\label{3.22CCA}
\tilde{v}_{n}\rightarrow \tilde{v} \quad \mbox{ in } Y_{V(0)},
\end{align}
which implies that $v_{n}\rightarrow v$ in $Y_{V(0)}$ and
\begin{equation*}
L_{V(0)}(\tilde{v})=d_{V(0)} \, \mbox{ and } \, \langle L'_{V(0)}(\tilde{v}), \tilde{v}\rangle=0.
\end{equation*}
Now, we show that $(z_{n})$ admits a subsequence, still denoted by $(z_{n})$, such that $z_{n}\rightarrow z_{0}$ for some $z_{0}\in M$. 
Assume by contradiction that $(z_{n})$ is not bounded in $\R^{N}$. Then there exists a subsequence, still denoted by $(z_{n})$, such that $|z_{n}|\rightarrow \infty$ as $n\ri \infty$. 
Take $R>0$ such that $\Lambda \subset B_{R}$, and assume that $|z_{n}|>2R$ for $n$ large. Thus, for any $x\in B_{R/\e_{n}}$, we get $|\e_{n} x+z_{n}|\geq |z_{n}|-|\e_{n} x|>R$ for all $n$ large enough.
Hence, by using the definition of $g$, we have
\begin{align*}
\|v_{n}\|_{V(0)}^{2}&\leq \int_{\R^{N}} g(\e_{n}x+z_{n}, v_{n}(x, 0))v_{n}(x, 0) \, dx \\
&\leq\int_{B_{R/\e_{n}}}  \tilde{f}(v_{n}(x, 0)) v_{n}(x, 0) \,dx+\int_{\mathbb{R}^{N}\setminus B_{R/\e_{n}}} f(v_{n}(x, 0)) v_{n}(x, 0) \, dx.
\end{align*}
Since $v_{n}\rightarrow v$ in $Y_{V(0)}$ as $n\ri \infty$, it follows from the dominated convergence theorem  that 
\begin{align*}
\int_{\mathbb{R}^{N}\setminus B_{R/\e_{n}}} f(v_{n}(x, 0)) v_{n}(x, 0) \, dx=o_{n}(1).
\end{align*}
Therefore,
\begin{align*}
\|v_{n}\|^{2}_{V(0)}\leq \frac{V_{0}}{\kappa} \int_{B_{R/\e_{n}}}  v^{2}_{n}(x, 0) \,dx+o_{n}(1),
\end{align*}
which combined with \eqref{m-ineq} and $V(0)=-V_{0}$ yields
$$
\left(1-\frac{V_{0}}{m^{2s}}\left(1+\frac{1}{\kappa}\right)\right)\|v_{n}\|^{2}_{\x}\leq o_{n}(1)
$$
and this gives a contradiction due to $v_{n}\rightarrow v$ in $\x$, $v\not\equiv 0$ and $\kappa>\frac{V_{0}}{m^{2s}-V_{0}}$.
Thus, $(z_{n})$ is bounded in $\R^{N}$ and, up to a subsequence, we may assume that $z_{n}\rightarrow z_{0}$. 
If $z_{0}\notin \overline{\Lambda}$, then there exists $r>0$ such that $z_{n}\in B_{r/2}(z_{0})\subset \overline{\Lambda}^{c}$ for any $n$ large enough. Reasoning as before, we reach $v_{n}\ri 0$ in $\x$, that is a contradiction. 
Hence, $z_{0}\in \overline{\Lambda}$.
Now, we show that $V(z_{0})=V(0)$. Assume by contradiction that $V(z_{0})>V(0)$.
Taking into account \eqref{3.22CCA}, $V(z_{0})>V(0)=-V_{0}$, Fatou's lemma and the invariance of $\mathbb{R}^{N}$ by translations, we have
\begin{align*}
d_{V(0)}=L_{V(0)}(\tilde{v})&<L_{V(z_{0})}(\tilde{v}) \\
&=\frac{1}{2}\Bigl(\iint_{\R^{N+1}_{+}} y^{1-2s}(|\nabla\tilde{v}|^{2}+m^{2}\tilde{v}^{2})\, dx dy +\frac{1}{2}\int_{\R^{N}}(V(z_{0})+V_{0}) \tilde{v}^{2}_{n}(x, 0)\,dx-\frac{V_{0}}{2} \int_{\R^{N}} \tilde{v}^{2}_{n}(x, 0)\,dx \Bigr) \\
&\quad-\int_{\mathbb{R}^{N}} F(\tilde{v}_{n}(x, 0))\,dx\\
&\leq \liminf_{n\rightarrow \infty} \Biggl[ \frac{1}{2}\left(\iint_{\R^{N+1}_{+}} y^{1-2s}(|\nabla\tilde{v}_{n}|^{2}+m^{2}\tilde{v}_{n}^{2})\, dx dy+\frac{1}{2}\int_{\R^{N}}V(\e_{n} x+z_{n}) \tilde{v}^{2}_{n}(x, 0)\,dx \right)\\
&\quad- \int_{\mathbb{R}^{N}} F(\tilde{v}_{n}(x, 0))\,dx \Biggr] \nonumber\\
&=\liminf_{n\rightarrow \infty} \Biggl[ \frac{t_{n}^{2}}{2}\left(\iint_{\R^{N+1}_{+}} y^{1-2s}(|\nabla u_{n}|^{2}+m^{2}u_{n}^{2})\, dx dy+\frac{t_{n}^{2}}{2}\int_{\R^{N}}V(\e_{n} x) u^{2}_{n}(x, 0)\,dx \right)\\
&\quad- \int_{\mathbb{R}^{N}} F(t_{n}u_{n}(x, 0))\,dx \Biggr] \nonumber\\
&\leq \liminf_{n\rightarrow \infty} J_{\e_{n}}(t_{n}u_{n}) \leq \liminf_{n\rightarrow \infty} J_{\e_{n}} (u_{n})=d_{V(0)}
\end{align*}
which gives a contradiction. Therefore, in view of $(V'_2)$, we can conclude that $z_{0}\in M$.
\end{proof}

\noindent
Now, we consider the following subset of $\N_{\e}$:
$$
\widetilde{\N}_{\e}:=\left \{u\in \N_{\e}: J_{\e}(u)\leq d_{V(0)}+h(\e)\right\},
$$
where $h(\e):=\sup_{z\in M}|J_{\e}(\Phi_{\e}(z))-d_{V(0)}|\rightarrow 0$ as $\e \rightarrow 0$ as a consequence of Lemma \ref{lem3.4}. By the definition of $h(\e)$, we know that, for all $z\in M$ and $\e>0$, $\Phi_{\e}(z)\in \widetilde{\N}_{\e}$ and $\widetilde{\N}_{\e}\neq \emptyset$. 
We present below an interesting relation between $\widetilde{\mathcal{N}}_{\e}$ and the barycenter map $\beta_{\e}$.
\begin{lem}\label{lem3.5}
For any $\delta>0$, there holds that
$$
\lim_{\e \rightarrow 0} \sup_{u\in \widetilde{\mathcal{N}}_{\e}} {\rm dist}(\beta_{\e}(u), M_{\delta})=0.
$$
\end{lem}
\begin{proof}
Let $\e_{n}\rightarrow 0$ as $n\rightarrow \infty$. By definition, there exists $(u_{n})\subset \widetilde{\mathcal{N}}_{\e_{n}}$ such that
$$
\sup_{u\in \widetilde{\mathcal{N}}_{\e_{n}}} \inf_{z\in M_{\delta}}|\beta_{\e_{n}}(u)-z|=\inf_{z\in M_{\delta}}|\beta_{\e_{n}}(u_{n})-z|+o_{n}(1).
$$
Therefore, it suffices to prove that there exists $(z_{n})\subset M_{\delta}$ such that
\begin{equation}\label{3.13DCDS}
\lim_{n\rightarrow \infty} |\beta_{\e_{n}}(u_{n})-z_{n}|=0.
\end{equation}
Observing that $L_{V(0)}(t u_{n})\leq J_{\e_{n}}(t u_{n})$ for all $t\geq 0$ and $(u_{n})\subset  \widetilde{\mathcal{N}}_{\e_{n}}\subset  \mathcal{N}_{\e_{n}}$, we deduce that
$$
d_{V(0)}\leq c_{\e_{n}}\leq  J_{\e_{n}}(u_{n})\leq d_{V(0)}+h(\e_{n}),
$$
which implies that $J_{\e_{n}}(u_{n})\rightarrow d_{V(0)}$. Using Proposition \ref{prop3.3}, there exists $(\tilde{z}_{n})\subset \mathbb{R}^{N}$ such that $z_{n}=\e_{n}\tilde{z}_{n}\in M_{\delta}$ for $n$ sufficiently large. Thus
$$
\beta_{\e_{n}}(u_{n})=z_{n}+\frac{\int_{\mathbb{R}^{N}}[\varUpsilon(\e_{n}x+z_{n})-z_{n}] v^{2}_{n}(x, 0) \, dx}{\int_{\mathbb{R}^{N}} v^{2}_{n}(x, 0) \, dx}=z_{n}+o_{n}(1)
$$
since $v_{n}(x, y):=u_{n}(\cdot+\tilde{z}_{n}, y)$ strongly converges in $Y_{V(0)}$  and $\e_{n}x+z_{n}\ri z_{0}\in M_{\delta}$.  
Therefore, (\ref{3.13DCDS}) holds true.
\end{proof}

We end this section by proving a multiplicity result for \eqref{MEP}. 
Since $\mathbb{S}^{+}_{\e}$ is not a completed metric space, we make use of the abstract category result in \cite{SW}; see also \cite{Aampa, FJ}.
\begin{thm}\label{multiple}
For any $\delta>0$ such that $M_{\delta}\subset \Lambda$, there exists $\tilde{\e}_{\delta}>0$ such that, for any $\e\in (0, \tilde{\e}_{\delta})$, problem \eqref{MEP} has at least $cat_{M_{\delta}}(M)$ positive solutions.
\end{thm}
\begin{proof}
For any $\e>0$, we consider the map $\alpha_{\e} : M \rightarrow \mathbb{S}_{\e}^{+}$ defined as $\alpha_{\e}(z):= m_{\e}^{-1}(\Phi_{\e}(z))$. 
Using Lemma \ref{lem3.4}, we see that
\begin{equation}\label{FJSMAW}
\lim_{\e \rightarrow 0} \psi_{\e}(\alpha_{\e}(z)) = \lim_{\e \rightarrow 0} J_{\e}(\Phi_{\e}(z))= d_{V(0)}, \quad \mbox{ uniformly in } z\in M. 
\end{equation}  
Set
$$
\widetilde{\mathcal{S}}^{+}_{\e}:=\{ w\in \mathbb{S}_{\e}^{+} : \psi_{\e}(w) \leq d_{V(0)} + h(\e)\}, 
$$
where $h(\e):=\sup_{z\in M}|\psi_{\e}(\alpha_{\e}(z))-d_{V(0)}|$.
It follows from \eqref{FJSMAW} that $h(\e)\rightarrow 0$ as $\e\rightarrow 0$. Then, $\widetilde{\mathcal{S}}^{+}_{\e}\neq \emptyset$ for all $\e>0$.
From the above considerations and in the light of Lemma \ref{lem3.4}, Lemma \ref{lem2.3MAW}-$(iii)$, Lemma \ref{lem3.5N} and Lemma \ref{lem3.5}, we can find $\bar{\e}= \bar{\e}_{\delta}>0$ such that the diagram of continuous applications bellow is well-defined for $\e \in (0, \bar{\e})$
\begin{equation*}
M\stackrel{\Phi_{\e}}{\rightarrow} \Phi_{\e}(M) \stackrel{m_{\e}^{-1}}{\rightarrow} \alpha_{\e}(M)\stackrel{m_{\e}}{\rightarrow} \Phi_{\e}(M) \stackrel{\beta_{\e}}{\rightarrow} M_{\delta}.
\end{equation*}    
Thanks to Lemma \ref{lem3.5N}, and decreasing $\bar{\e}$ if necessary, we see that $\beta_{\e}(\Phi_{\e}(z))= z+ \theta(\e, z)$ for all $z\in M$, for some function $\theta(\e, z)$ satisfying $|\theta(\e, z)|<\frac{\delta}{2}$ uniformly in $z\in M$ and for all $\e \in (0, \bar{\e})$. Define $H(t, z):= z+ (1-t)\theta(\e, z)$ for $(t, z)\in [0, 1]\times M$. Then $H: [0,1]\times M\rightarrow M_{\delta}$ is continuous. Clearly,  $H(0, z)=\beta_{\e}(\Phi_{\e}(z))$ and $H(1,z)=z$ for all $z\in M$. Consequently, $H(t, z)$ is a homotopy between $\beta_{\e} \circ \Phi_{\e} = (\beta_{\e} \circ m_{\e}) \circ \alpha_{\e}$ and the inclusion map $id: M \rightarrow M_{\delta}$. This fact yields 
\begin{equation}\label{catMAW}
cat_{\alpha_{\e}(M)} \alpha_{\e}(M)\geq cat_{M_{\delta}}(M).
\end{equation}
Applying Corollary \ref{cor2.1MAW}, Lemma \ref{lem2.3AM} (see also Remark \ref{REMARKHEZOU}), and Theorem $27$ in \cite{SW} with $c= c_{\e}\leq d_{V(0)}+h(\e) =d$ and $K= \alpha_{\e}(M)$, we obtain that $\psi_{\e}$ has at least $cat_{\alpha_{\e}(M)} \alpha_{\e}(M)$ critical points on $\widetilde{\mathcal{S}}^{+}_{\e}$.
Taking into account Proposition \ref{prop2.1MAW}-$(d)$ and \eqref{catMAW}, we infer that $J_{\e}$ admits at least $cat_{M_{\delta}}(M)$ critical points in $\widetilde{\mathcal{N}}_{\e}$.    
\end{proof}

\noindent
We also have the following fundamental result.
\begin{lem}\label{moser2} 
Let $\e_{n}\rightarrow 0$ and $(u_{n})\subset \widetilde{\mathcal{N}}_{\e_{n}}$ be a sequence of solutions to \eqref{MEP}. 
Then, $v_{n}(x, y):=u_{n}(x+\tilde{z}_{n}, y)$ satisfies $v_{n}(\cdot, 0)\in L^{\infty}(\R^{N})$ and there exists $C>0$ such that 
\begin{equation}\label{UBu}
|v_{n}(\cdot, 0)|_{\infty}\leq C  \quad\mbox{ for all } n\in \mathbb{N},
\end{equation}
where $(\tilde{z}_{n})$ is given by Proposition \ref{prop3.3}.
Moreover,
\begin{align}\label{vanishing}
\lim_{|x|\rightarrow \infty} v_{n}(x, 0)=0  \quad\mbox{ uniformly in } n\in \mathbb{N}.
\end{align}
\end{lem}
\begin{proof}
Since $J_{\e_{n}}(u_{n})\leq d_{V(0)}+h(\e_{n})$ with $h(\e_{n})\rightarrow 0$ as $n\rightarrow \infty$, we can argue as at the beginning of the proof of Proposition \ref{prop3.3} to deduce that $J_{\e_{n}}(u_{n})\rightarrow d_{V(0)}$. Thus we may invoke  Proposition \ref{prop3.3} to obtain a sequence $(\tilde{z}_{n})\subset \R^{N}$ such that $\e_{n}\tilde{z}_{n}\rightarrow z_{0}\in M$ and $v_{n}(x, y):=u_{n}(x+\tilde{z}_{n}, y)$ strongly converges in $Y_{V(0)}$. At this point we can proceed as in the proofs of Lemmas \ref{moser} and \ref{lem2.6AM} to obtain the assertions.
\end{proof}

\noindent
We conclude this section by giving the proof of Theorem \ref{thm2}.
\begin{proof}[Proof of Theorem \ref{thm2}]
Let $\delta>0$ be such that $M_{\delta}\subset \Lambda$. Firstly, we claim that there exists  $\hat{\e}_{\delta}>0$ such that for any $\e\in (0, \hat{\e}_{\delta})$ and any solution $u\in \widetilde{\mathcal{N}}_{\e}$ of \eqref{MEP}, it holds
\begin{equation}\label{Ua}
|u(\cdot, 0)|_{L^{\infty}(\Lambda^{c}_{\e})}<a.
\end{equation}
We argue by contradiction and assume that for some subsequence $\e_{n}\rightarrow 0$ we can obtain $u_{n}=u_{\e_{n}}\in \widetilde{\mathcal{N}}_{\e_{n}}$ such that $J'_{\e_{n}}(u_{n})=0$ and
\begin{equation}\label{AbsAFF}
|u_{n}(\cdot, 0)|_{L^{\infty}(\Lambda^{c}_{\e_{n}})}\geq a.
\end{equation}
Since $J_{\e_{n}}(u_{n})\leq d_{V(0)}+h(\e_{n})$, we can argue as in the first part of Proposition \ref{prop3.3} to deduce that $J_{\e_{n}}(u_{n})\rightarrow d_{V(0)}$.
In view of Proposition \ref{prop3.3}, there exists $(\tilde{z}_{n})\subset \R^{N}$ such that $\e_{n}\tilde{z}_{n}\rightarrow z_{0}$ for some $z_{0} \in M$ and $v_{n}(x, y):=u_{n}(x+\tilde{z}_{n}, y)$ strongly converges in $Y_{V(0)}$. Then we can proceed as in the proof of Theorem \ref{thm1} to get a contradiction.
Let $\tilde{\e}_{\delta}>0$ be given by Theorem \ref{multiple} and we set $\e_{\delta}:=\min\{\tilde{\e}_{\delta}, \hat{\e}_{\delta} \}$. Fix $\e\in (0, \e_{\delta})$. Applying Theorem \ref{multiple}, we obtain at least $cat_{M_{\delta}}(M)$ positive solutions to \eqref{MEP}.
If $u\in X_{\e}$ is one of these solutions, then $u\in \widetilde{\mathcal{N}}_{\e}$, and in view of \eqref{Ua} and the definition of $g$, we  infer that $u$ is also a solution to \eqref{EP}. Then, $u(\cdot, 0)$ is a solution to (\ref{P}), and we deduce that \eqref{P} has at least $cat_{M_{\delta}}(M)$ positive solutions. To study the behavior of the maximum points of the solutions, we argue as in the proof of Theorem \ref{thm1}.
\end{proof}

\section{Appendix A: Bessel potentials}

In this appendix we collect some useful results concerning Bessel potentials. For more details we refer to \cite{ArS, Calderon, Grafakos, stein}. 
\begin{defn}
Let $\alpha>0$. The Bessel potential of order $\alpha$ of $u\in \mathcal{S}(\R^{N})$ is defined as
\begin{align*}
\mathcal{J}_{\alpha}u(x):=(1-\Delta)^{-\frac{\alpha}{2}}u(x)=(\mathcal{G}_{\alpha}*u)(x)=\int_{\R^{N}} \mathcal{G}_{\alpha}(x-y)u(y)\, dy, 
\end{align*}
where $\mathcal{G}_{\alpha}$ given through the Fourier transform
$$
\mathcal{F}\mathcal{G}_{\alpha}(k):=(2\pi)^{-\frac{N}{2}} (1+|k|^{2})^{-\frac{\alpha}{2}}
$$
is the so-called Bessel kernel.
\end{defn}
\begin{remark}
If $\alpha\in \mathbb{R}$ (or $\alpha\in \mathbb{C}$), then we may define the Bessel potential of a temperate distribution $u\in \mathcal{S}'(\R^{N})$  (see \cite{Calderon}) by setting 
$$
\mathcal{F}(\mathcal{J}_{\alpha}u)(k):=(1+|k|^{2})^{-\frac{\alpha}{2}} \mathcal{F}u(k).
$$
\end{remark}
\noindent
It is possible to prove (see \cite{ArS}) that 
$$
\mathcal{G}_{\alpha}(x)=\frac{1}{2^{\frac{N+\alpha-2}{2}}\pi^{\frac{N}{2}}\Gamma(\frac{\alpha}{2})} K_{\frac{N-\alpha}{2}}(|x|) |x|^{\frac{\alpha-N}{2}}.
$$
Thus, $\mathcal{G}_{\alpha}(x)$ is a positive, decreasing function of $|x|$, analytic except at $x=0$, and for $x\in \R^{N}\setminus \{0\}$, $\mathcal{G}_{\alpha}(x)$ is an entire function of $\alpha$. Moreover, from \eqref{Watson1} and \eqref{Watson2}, we have
\begin{equation*}
\mathcal{G}_{\alpha}(x)\sim \left\{
\begin{array}{ll}
\displaystyle{\frac{\Gamma(\frac{N-\alpha}{2})}{2^{\alpha} \pi^{\frac{N}{2}} \Gamma(\frac{\alpha}{2})} |x|^{\alpha-N}}, &\mbox{ if } 0<\alpha<N, \\
\\
\displaystyle{ \frac{1}{2^{N-1} \pi^{\frac{N}{2}} \Gamma(\frac{N}{2})}\log\left(\frac{1}{|x|}\right)}, &\mbox{ if } \alpha=N, \\
\\
\displaystyle{\frac{\Gamma(\frac{\alpha-N}{2})}{2^{N}\pi^{\frac{N}{2}} \Gamma(\frac{\alpha}{2})}}, &\mbox{ if } \alpha>N, 
\end{array}
\right. 
\end{equation*}
as $|x|\rightarrow 0$, and
\begin{align*}
\mathcal{G}_{\alpha}(x)\sim \frac{1}{2^{\frac{N+\alpha-1}{2}} \pi^{\frac{N-1}{2}} \Gamma(\frac{\alpha}{2})} |x|^{\frac{\alpha-N-1}{2}} e^{-|x|} 
\end{align*}
as $|x|\ri \infty$.
Hence, $\mathcal{G}_{\alpha}\in L^{1}(\R^{N})$ for all $\alpha>0$, and 
$$
\int_{\R^{N}} \mathcal{G}_{\alpha}(x) \, dx=(2\pi)^{\frac{N}{2}}\mathcal{F}\mathcal{G}_{\alpha}(0)=1.
$$ 
By the definition of $\mathcal{G}_{\alpha}$, it is evident that the following composition formula holds:
$$
\mathcal{G}_{\alpha+\beta}=\mathcal{G}_{\alpha}*\mathcal{G}_{\beta}, \quad \alpha, \beta>0.
$$
We also mention the integral formula (see \cite{stein})
$$
\mathcal{G}_{\alpha}(x)=\frac{1}{(4\pi)^{\frac{N}{2}}\Gamma\left(\frac{\alpha}{2}\right)} \int_{0}^{\infty} e^{-\frac{|x|^{2}}{4\delta}} e^{-\delta} \delta^{\frac{\alpha-N}{2}} \frac{d\delta}{\delta}.
$$
One the most interesting facts concerning Bessel potentials is they can be employed to define the Bessel potential spaces; see \cite{Adams, ArS, Calderon, Grafakos, stein}. 
For $p\in [1, \infty]$ and $\alpha\in \R$, we define the Banach space
\begin{align*}
\mathscr{L}^{p}_{\alpha}:=\mathcal{J}_{\alpha}(L^{p}(\R^{N}))=\{u: u=\mathcal{G}_{\alpha}*f, \, f\in L^{p}(\R^{N})\}
\end{align*}
endowed with the norm
$$
\|u\|_{\mathscr{L}^{p}_{\alpha}}:=|f|_{p} \quad \mbox{ with } u:=\mathcal{G}_{\alpha}*f.
$$
Thus $\mathscr{L}^{p}_{\alpha}$ is a subspace of $L^{p}(\R^{N})$ for all $\alpha\geq 0$. 
We summarize some of its properties.
\begin{thm}\cite{Adams, Calderon, stein}\label{Besselembedding}
\begin{compactenum}[$(i)$]
\item If $\alpha\geq 0$ and $1\leq p< \infty$, then $\mathcal{D}(\R^{N})$ is dense in $\mathscr{L}^{p}_{\alpha}$.
\item If $1<p<\infty$ and $p'$ its conjugate exponent, then the dual of $\mathscr{L}^{p}_{\alpha}$ is isometrically isomorphic to $\mathscr{L}^{p'}_{-\alpha}$.
\item If $\beta< \alpha$, then $\mathscr{L}^{p}_{\alpha}$ is continuously embedded in $\mathscr{L}^{p}_{\beta}$.
\item If $\beta\leq \alpha$ and if either $1<p\leq q\leq \frac{Np}{N-(\alpha-\beta) p}<\infty$ or $p=1$ and $1\leq q<\frac{N}{N-\alpha+\beta}$, then $\mathscr{L}^{p}_{\alpha}$ is continuously embedded in $\mathscr{L}^{q}_{\beta}$.
\item If $0< \mu\leq \alpha-\frac{N}{p}<1$, then $\mathscr{L}^{p}_{\alpha}$ is continuously embedded in $C^{0, \mu}(\R^{N})$.
\item $\mathscr{L}^{p}_{k}=W^{k,p}(\R^{N})$ for all $k\in \mathbb{N}$ and $1<p<\infty$, $\mathscr{L}^{2}_{\alpha}=W^{\alpha, 2}(\R^{N})$ for any $\alpha$.
\item If $1<p<\infty$ and $\e>0$, then for every $\alpha$ we have the following continuous embeddings: 
$$
\mathscr{L}^{p}_{\alpha+\e}\subset W^{\alpha, p}(\R^{N})\subset \mathscr{L}^{p}_{\alpha-\e}.
$$
\end{compactenum}
\end{thm}
In order to accomplish some useful regularity results for equations driven by $(-\Delta+m^{2})^{s}$, with $m>0$ and $s\in (0, 1)$, we introduce the H\"older-Zygmund (or Lipschitz) spaces $\Lambda_{\alpha}$; see \cite{Calderon, Grafakos, stein}. 
If $\alpha>0$ and $\alpha\notin \mathbb{N}$, then we set $\Lambda_{\alpha}:=C^{[\alpha], \alpha-[\alpha]}(\R^{N})$. If $\alpha=k\in \mathbb{N}$, then we set $\Lambda_{\alpha}:=\Lambda^{*}_{k}$ where 
$$
\Lambda^{*}_{1}:=\left\{u\in L^{\infty}(\R^{N})\cap C(\R^{N}): \sup_{x, h\in \R^{N}, |h|>0} \frac{|u(x+h)+u(x-h)-2u(x)|}{|h|}<\infty\right\} \quad\mbox{ if } k=1,
$$
and
$$
\Lambda^{*}_{k}:=\left\{u\in C^{k-1}(\R^{N}): D^{\gamma}u\in \Lambda^{*}_{1} \,\, \mbox{ for all } |\gamma|\leq k-1\right\} \quad\mbox{ if } k\in \mathbb{N}, \, k\geq 2.
$$
The definition of $\Lambda_{\alpha}$ spaces can be also extended for $\alpha\leq 0$; see \cite{HormanderL, Taibleson}.
We recall the following useful result (for a proof we refer to Proposition 8.6.6. in \cite{HormanderL} and Theorem $6$ in \cite{Taibleson}; see also \cite{Calderon, stein}).
\begin{thm}\cite{Calderon, HormanderL, stein, Taibleson}\label{regBesseloperators}
Let $\alpha, \beta\in \R$. Then $\mathcal{J}_{2\alpha}$ maps $\Lambda_{\beta}$ isomorphically onto $\Lambda_{\beta+2\alpha}$.
\end{thm}
As a consequence of Theorem \ref{regBesseloperators} and the definition of H\"older-Zygmund spaces, we easily deduce the following regularity result.
\begin{cor}\label{COROLLARIO}
Let $m>0$, $s\in (0, 1)$ and $\alpha\in (0, 1)$. Assume that $f\in C^{0, \alpha}(\R^{N})$ and that $u\in L^{\infty}(\R^{N})$ is a solution to $(-\Delta+m^{2})^{s}u=f$ in $\R^{N}$. 
\begin{itemize}
\item If $\alpha+2s<1$, then $u\in C^{0, \alpha+2s}(\R^{N})$. 
\item If $1<\alpha+2s<2$, then $u\in C^{1, \alpha+2s-1}(\R^{N})$. 
\item If $2<\alpha+2s<3$, then $u\in C^{2, \alpha+2s-2}(\R^{N})$. 
\item If $\alpha+2s=k\in \{1, 2\}$, then $u\in \Lambda_{k}^{*}$. 
\end{itemize}
\end{cor}

We also have a regularity result for solutions $u\in L^{\infty}(\R^{N})$ to $(-\Delta+m^{2})^{s}u=f$ in $\R^{N}$ with $f\in L^{\infty}(\R^{N})$.
Indeed, arguing as in the proof of Theorem $4$ at page $149$ in \cite{stein} (see also Theorem $5$ in \cite{Taibleson}) and using formula $(59)$ with $l=2$ and $\beta=2s$ in this theorem, we can use the characterization of $\Lambda_{\alpha}$ in terms of the Poisson integral given in \cite{stein, Taibleson} to infer that the solutions $u$ of the previous equation belong to $\Lambda_{2s}$. Then, in view of the definition of $\Lambda_{\alpha}$ given above, we deduce the following result.
\begin{cor}\label{SilvestreLinfty}
Let $m>0$ and $s\in (0, 1)$. Assume that $f\in L^{\infty}(\R^{N})$ and that $u\in L^{\infty}(\R^{N})$ is a solution to $(-\Delta+m^{2})^{s}u=f$ in $\R^{N}$. 
\begin{itemize}
\item If $2s< 1$, then $u\in C^{0, 2s}(\R^{N})$.
\item If $2s=1$, then $u\in \Lambda_{1}^{*}$.
\item If $2s>1$, then $u\in C^{1, 2s-1}(\R^{N})$.
\end{itemize}
\end{cor}

\begin{remark}
When $m=0$, the conclusions in Corollary \ref{COROLLARIO} and Corollary \ref{SilvestreLinfty} can be found in \cite{Stinga}.
\end{remark}

\section{Appendix B: An integral representation formula for $(-\Delta+m^{2})^{s}$}

Bearing in mind the asymptotic estimates \eqref{Watson1} and \eqref{Watson2} for $K_{\nu}$, we are able to gain an integral representation formula for $(-\Delta+m^{2})^{s}$, with $m>0$ and $s\in (0, 1)$, in the spirit of Lemma $3.2$ in \cite{DPV}.
\begin{thm}\label{INTEGRALE}
Let $m>0$ and $s\in (0, 1)$. Then, for all $u\in \mathcal{S}(\R^{N})$,
$$
(-\Delta+m^{2})^{s}u(x)=m^{2s}u(x)+\frac{C(N,s)}{2} m^{\frac{N+2s}{2}} \int_{\R^{N}} \frac{2u(x)-u(x+y)-u(x-y)}{|y|^{\frac{N+2s}{2}}}K_{\frac{N+2s}{2}}(m|y|)\, dy.
$$
\end{thm}
\begin{proof}
Choosing the substitution $z=y-x$ in \eqref{FFdef}, we obtain
\begin{align}\label{3.4DPV}
(-\Delta+m^{2})^{s}u(x)&=m^{2s}u(x)+C(N,s) m^{\frac{N+2s}{2}} P. V. \int_{\R^{N}} \frac{u(x)-u(y)}{|x-y|^{\frac{N+2s}{2}}}K_{\frac{N+2s}{2}}(m|x-y|)\, dy \nonumber\\
&=m^{2s}u(x)+C(N,s) m^{\frac{N+2s}{2}} P. V. \int_{\R^{N}} \frac{u(x)-u(x+z)}{|z|^{\frac{N+2s}{2}}}K_{\frac{N+2s}{2}}(m|z|)\, dz.
\end{align}
By substituting $\tilde{z}=-z$ in the last term in \eqref{3.4DPV}, we get 
\begin{align}\label{3.5DPV}
P. V. \int_{\R^{N}} \frac{u(x+z)-u(x)}{|z|^{\frac{N+2s}{2}}}K_{\frac{N+2s}{2}}(m|z|)\, dz=P. V. \int_{\R^{N}} \frac{u(x-\tilde{z})-u(x)}{|\tilde{z}|^{\frac{N+2s}{2}}}K_{\frac{N+2s}{2}}(m|\tilde{z}|)\, d\tilde{z},
\end{align}
and so, after relabeling $\tilde{z}$ as $z$, 
\begin{align}\label{3.6DPV}
&2P. V. \int_{\R^{N}} \frac{u(x+z)-u(x)}{|z|^{\frac{N+2s}{2}}}K_{\frac{N+2s}{2}}(m|z|)\, dz \nonumber\\
&=P. V. \int_{\R^{N}} \frac{u(x+z)-u(x)}{|z|^{\frac{N+2s}{2}}}K_{\frac{N+2s}{2}}(m|z|)\, dz \nonumber\\
&\quad+P. V. \int_{\R^{N}} \frac{u(x-z)-u(x)}{|z|^{\frac{N+2s}{2}}}K_{\frac{N+2s}{2}}(m|z|)\, dz \nonumber\\
&=P. V. \int_{\R^{N}} \frac{u(x+z)+u(x-z)-2u(x)}{|z|^{\frac{N+2s}{2}}}K_{\frac{N+2s}{2}}(m|z|)\, dz.
\end{align}
Hence, if we rename $z$ as $y$ in \eqref{3.4DPV} and  \eqref{3.6DPV}, we can write $(-\Delta+m^{2})^{s}$  as
\begin{align}\label{FINALE}
(-\Delta+m^{2})^{s}u(x)=m^{2s}u(x)+\frac{C(N,s)}{2} m^{\frac{N+2s}{2}} P. V. \int_{\R^{N}} \frac{2u(x)-u(x+y)-u(x-y)}{|y|^{\frac{N+2s}{2}}}K_{\frac{N+2s}{2}}(m|y|)\, dy.
\end{align}
Now, by using a second order Taylor expansion, we see that
\begin{align*}
\left|\frac{2u(x)-u(x+y)-u(x-y)}{|y|^{\frac{N+2s}{2}}}K_{\frac{N+2s}{2}}(m|y|) \right|\leq \frac{|D^{2}u|_{\infty}}{|y|^{\frac{N+2s-4}{2}}} K_{\frac{N+2s}{2}}(m|y|). 
\end{align*}
From \eqref{Watson1} we deduce that 
$$
\frac{|D^{2}u|_{\infty}}{|y|^{\frac{N+2s-4}{2}}} K_{\frac{N+2s}{2}}(m|y|)\sim \frac{C}{|y|^{N+2s-2}} \quad\mbox{ as } |y|\ri 0
$$
which implies the integrability  near $0$. On the other hand, using \eqref{Watson2}, we get
$$
\left|\frac{2u(x)-u(x+y)-u(x-y)}{|y|^{\frac{N+2s}{2}}}K_{\frac{N+2s}{2}}(m|y|) \right|\leq \frac{C|u|_{\infty}}{|y|^{\frac{N+2s}{2}}} K_{\frac{N+2s}{2}}(m|y|)\sim \frac{C}{|y|^{\frac{N+2s+1}{2}}}e^{-m|y|} \quad\mbox{ as } |y|\ri \infty
$$
which gives the integrability near $\infty$.
Therefore, we can remove the $P. V.$ in \eqref{FINALE}.
\end{proof}

\section*{Acknowledgements}
The author would like to thank the anonymous referee for her/his careful reading of the manuscript and valuable suggestions.

\end{document}